\documentclass[a4paper, 11pt]{amsart}
\usepackage{amsmath, amssymb, amsthm, mathrsfs}
\usepackage[mathscr]{eucal}
\usepackage[all]{xy}
\usepackage[dvipdfmx]{graphicx}

\usepackage{bm, comment}
\usepackage{mathtools}

\theoremstyle{plain}
 \newtheorem{thm}{Theorem}[section]
 \newtheorem{lem}[thm]{Lemma}
 \newtheorem{cor}[thm]{Corollary}
 \newtheorem{prop}[thm]{Proposition}
 \newtheorem{claim}[thm]{Claim}
\theoremstyle{definition}
  \newtheorem{defn}[thm]{Definition}
  \newtheorem{ass}[thm]{Assumption}
  
  \newtheorem{question}[thm]{Question}

\theoremstyle{remark}
  \newtheorem{rem}[thm]{Remark}
  \newtheorem{ex}[thm]{Example}

\newcommand{\R}{\mathbb{R}}
\newcommand{\C}{\mathbb{C}}
\newcommand{\cal}{\mathcal}
\newcommand{\N}{\mathbb{N}}
\newcommand{\Z}{\mathbb{Z}}
\newcommand{\calr}{\mathcal{R}}
\newcommand{\calb}{\mathcal{B}}
\newcommand{\calg}{\mathcal{G}}
\newcommand{\calh}{\mathcal{H}}
\newcommand{\cals}{\mathcal{S}}
\newcommand{\cali}{\mathcal{I}}

\newcommand{\cala}{\mathcal{A}}

\newcommand{\ve}{\varepsilon}

\newcommand{\aut}{\mathrm{Aut}}
\newcommand{\Iso}{\mathrm{Iso}}

\newcommand{\ci}[2]{\cite[#1]{#2}}
\renewcommand{\c}{\curvearrowright}

  \makeatletter
  \@addtoreset{equation}{section}
  \makeatother

\newcommand{\defeq}{\coloneqq}

\begin{document}

\title{Inner amenable groupoids and central sequences}
\author{Yoshikata Kida and Robin Tucker-Drob}
\address{Graduate School of Mathematical Sciences, the University of Tokyo, Komaba, Tokyo 153-8914, Japan}
\email{kida@ms.u-tokyo.ac.jp}
\address{Department of Mathematics, Texas A\&M University, College Station, TX 77843, USA}
\email{rtuckerd@math.tamu.edu}
\date{January 15, 2020}
\thanks{The first author was supported by JSPS Grant-in-Aid for Scientific Research, 17K05268. The second author was supported by NSF grant DMS 1600904.}

\begin{abstract}
We introduce inner amenability for discrete p.m.p.\ groupoids and investigate its basic properties, examples, and the connection with central sequences in the full group of the groupoid or central sequences in the von Neumann algebra associated with the groupoid.
Among other things, we show that every free ergodic p.m.p.\ compact action of an inner amenable group gives rise to an inner amenable orbit equivalence relation.
We also obtain an analogous result for compact extensions of equivalence relations which either are stable or have a non-trivial central sequence in their full group.
\end{abstract}

\maketitle


\section{Introduction}

A discrete countable group $G$ is called \textit{inner amenable} if there exists a sequence $(\xi_n)$ of non-negative unit vectors in $\ell^1(G)$ such that $\Vert \xi_n^g-\xi_n\Vert_1\to 0$ and $\xi_n(g)\to 0$ for all $g\in G$, where the function $\xi_n^g$ is defined by $\xi_n^g(h)=\xi_n(ghg^{-1})$ for $h\in G$.
This notion was introduced by Effros \cite{effros} who first observed its connection with property Gamma of the group von Neumann algebra.
This connection has since become a common theme: inner amenability of a group $G$ can often be deduced from the existence of a certain central sequence, either in the von Neumann algebra associated with $G$, or in the full group of a probability-measure-preserving (p.m.p.)\ action of $G$ (e.g., \cite{choda}, \cite{effros} and \cite{js}).

In this paper, we introduce inner amenability for discrete p.m.p.\ groupoids. We investigate its basic properties and examine its connection with central sequences, both in the full group of the groupoid and in the von Neumann algebra associated with the groupoid, highlighting many examples along the way.
We expect results in this paper to accelerate the understanding of free ergodic p.m.p.\ actions of inner amenable groups and their orbit equivalence relations.
We refer to \cite{dv}, \cite{is}, \cite{kida-inn}--\cite{kida-sce}, \cite{mar}, \cite{pv}, \cite{td} and \cite{vaes} for recent progress on related topics.

We briefly outline results of this paper.
We refer to Section \ref{sec-pre} for notation and terminology on discrete p.m.p.\ groupoids.
In Section \ref{sec-iag}, inner amenability of a discrete p.m.p.\ groupoid $\calg$ is defined (Definition \ref{def:groupoid}), generalizing the definition given above for groups; if the groupoid $\calg$ is ergodic, then its inner amenability is equivalent to the existence of a sequence $(\xi_n)$ of non-negative unit vectors in $L^1(\calg)$ which is asymptotically conjugation-invariant and diffuse (see Theorem \ref{thm:equiv} for various equivalent characterizations).

We observe permanence of inner amenability under several groupoid constructions.
In Subsection \ref{subsec-perma}, we prove the permanence under inflations, restrictions, measure-preserving extensions, ergodic decompositions, and inverse limits, and also prove that for every finite-index inclusion $\calh <\calg$ of discrete p.m.p.\ groupoids, inner amenability of $\calh$ implies inner amenability of $\calg$. 
The permanence under measure-preserving extensions is specialized into the following assertion: for every p.m.p.\ action $G\c (X, \mu)$ of a countable group $G$, inner amenability of the translation groupoid $G\ltimes (X, \mu)$ implies inner amenability of the group $G$ (Proposition \ref{prop:Ginn}).
Of particular interest is that we establish permanence of inner amenability under compact extensions, which is thoroughly discussed in Section \ref{sec-cpt-ia}.
This implies that inner amenability passes to finite-index Borel subgroupoids, and that every compact free p.m.p.\ action of an inner amenable group gives rise to an inner amenable orbit equivalence relation (Corollaries \ref{cor-finite-index-cpt-ext}--\ref{cor:distal}).

One motivating example of an inner amenable groupoid is an ergodic discrete p.m.p.\ equivalence relation $\calr$ which is Schmidt.
Here we say that $\calr$ is \textit{Schmidt} if there exists a non-trivial central sequence in its full group $[\calr]$.
In Subsections \ref{subsec-schmidt} and \ref{subsec-gamma}, we discuss relationship between inner amenability of $\calr$ and the existence of a non-trivial central sequence in $[\calr]$ or in the von Neumann algebra associated with $\calr$.
Especially, generalizing \cite[Theorem (ii)]{choda} for group-measure-space $\textrm{II}_1$ factors, we prove that for a free, strongly ergodic, p.m.p.\ action $G\c (X, \mu)$, if the associated von Neumann algebra $G\ltimes L^\infty(X)$ has property Gamma, then the orbit equivalence relation of the action is inner amenable (Corollary \ref{cor-gamma}).
The converse does not hold.
An example to indicate this is constructed by using the Vaes group \cite{vaes}, which is known as a countable, ICC, inner amenable group such that the associated group von Neumann algebra does not have property Gamma (Example \ref{ex-vaes}).

In Section \ref{sec-cpt-cs}, we show that for discrete p.m.p.\ equivalence relations, being Schmidt and being stable are both preserved under compact extensions.
Here we call a discrete p.m.p.\ equivalence relation \textit{stable} if it absorbs the ergodic hyperfinite p.m.p.\ equivalence relation on a non-atomic probability space, under direct product.
Hence being Schmidt and being stable pass to finite-index subrelations.
However for a finite-index inclusion $\cals <\calr$ of ergodic discrete p.m.p.\ equivalence relations, while inner amenability passes from $\cals$ to $\calr$, neither being Schmidt nor being stable passes from $\cals$ to $\calr$ in general.
In Section \ref{sec-finite-index}, we show this by giving examples, and also give a sufficient condition for this passage to hold, in terms of the algebra of asymptotically invariant sequences for $\cals$ and the action of $F$ on it when $\calr$ is written as the crossed product $\calr =\cals \rtimes F$ for some finite group $F$ acting on $\cals$ by automorphisms.
These results should be compared with the result of Pimsner-Popa on property Gamma and the McDuff stability for inclusions of $\text{II}_1$ factors \cite[Proposition 1.11]{pp}.

In \cite{sch-prob}, Schmidt asked whether every countable, inner amenable group admits a free ergodic p.m.p.\ action whose orbit equivalence relation is Schmidt. We say that a countable group has the \textit{Schmidt property} if it admits such an action. Schmidt's question remains open and is one of the questions motivating the present work. We call a countable group \textit{orbitally inner amenable} if it admits a free ergodic p.m.p.\ action whose orbit equivalence relation is inner amenable. As mentioned above, every countable group with the Schmidt property is orbitally inner amenable, and every countable, orbitally inner amenable group is inner amenable. While we do not know whether every inner amenable group is orbitally inner amenable, it follows from our aforementioned result on compact actions that this implication holds under the additional assumption that the group in question is residually finite. It also follows from \cite[Theorem 15]{td}, using a different method, that every inner amenable, linear group has the Schmidt property and is hence orbitally inner amenable.
We refer to Subsection \ref{subsec-schmidt} for more details.

We call a countable group \textit{stable} if it admits a free ergodic p.m.p.\ action whose orbit equivalence relation is stable.
Since, as we mentioned above, stability passes to finite-index subrelations, it follows that stability of a countable group passes to finite-index subgroups as well.
Combining this with a result of the first author \cite{kida-sce}, we obtain the corollary that stability of a countable group is invariant under virtual isomorphism (Corollary \ref{cor-fi}).
Although we can also ask the same question for the Schmidt property and orbital inner amenability of a countable group, it remains unsolved.
More precisely, we do not know whether those two properties are invariant under central group-extension with finite central group (Question \ref{q-finite-central}).

In Section \ref{sec-sg}, we obtain examples of discrete p.m.p.\ equivalence relations $\calr$ which are either not inner amenable, or not Schmidt, by imposing spectral gap properties and mixing properties on p.m.p.\ actions which generate $\calr$.
Among other things, for every countable non-amenable group $G$, the Bernoulli shift of $G$ gives rise to an orbit equivalence relation which is not inner amenable (Corollary \ref{cor-ber}).
Moreover, the product of the Bernoulli shift of $G$ with an arbitrary ergodic p.m.p.\ action of $G$ also gives rise to a non-inner-amenable orbit equivalence relation (Proposition \ref{prop-diagonal-action}).
There also exist inner amenable equivalence relations which are not Schmidt.
Such examples are found in Remarks \ref{rem-h-g-schmidt}, \ref{rem-wreath} and \ref{rem-product}.


In Section \ref{sec-ex}, we collect miscellaneous examples of orbitally inner amenable groups and free ergodic p.m.p.\ actions such that central sequences in the full group are well controlled via spectral gap properties and mixing properties.

In Section \ref{sec-top}, for a discrete p.m.p.\ groupoid $\calg$, we discuss topologies of the full group $[\calg]$ and the automorphism group $\aut(\calg)$ of $\calg$.
We prove that there exists a non-trivial central sequence in $[\calg]$ if and only if the set of inner automorphisms of $\calg$ given by conjugation of an element of $[\calg]$ is not closed in $\aut(\calg)$, under the assumption that the equivalence relation associated to $\calg$, $\{ \, (r(g), s(g))\mid g\in \calg\, \}$, is aperiodic (Proposition \ref{prop-cs-nc}).
Here $r$ and $s$ are the range and source maps of $\calg$, respectively.
This is already known for principal groupoids and proved in \cite[Sections 6 and 7]{kec} (and we refer to \cite[Theorem 3.1]{connes-ap} for a von Neumann algebraic counterpart).
Our definition of the topology of $\aut(\calg)$ for general groupoids $\calg$ follows \cite{kec}, and we will use some results in it.

Throughout the paper, unless otherwise mentioned, all relations among Borel sets and maps are understood to hold up to sets of measure zero.

\medskip

\noindent \textit{Acknowledgments.}
We thank Adrian Ioana for his valuable remarks on the first author's earlier note and for kindly allowing us to incorporate them into Lemma \ref{lem-p} and Corollary \ref{cor-c} (ii).
The second author would like to thank Andrew Marks for a suggestion which helped to simplify the proof of Theorem \ref{thm:equiv}.
We are deeply grateful to the anonymous referee for reading the first version of the paper carefully and for suggesting us a large number of improvements, corrections, and different proofs, which include: the presentation of Introduction, the definition of a local section of a groupoid, Remark \ref{rem:suffice}, Lemma \ref{lem:diffuse}, Remark \ref{rem-gp}, Remark \ref{rem-central}, Remark \ref{rem-iagroupoid}, Remark \ref{rem-proj}, and Remark \ref{rem-ssg}.
We are also indebted to the referee for his or her suggesting the result of Section \ref{sec-top}.


\section{Groupoid preliminaries and notation}\label{sec-pre}

For a groupoid $\calg$, we denote the unit space of $\mathcal{G}$ by $\mathcal{G}^0$, and denote the source and range maps of $\mathcal{G}$ by $s$ and $r$, respectively.
For a subset $D\subset \mathcal{G}$ and $x, y\in \mathcal{G}^0$, we set $D_x \defeq D\cap s^{-1}(x)$, $D^y \defeq D\cap r^{-1}(y)$ and $D_x^y\defeq D_x\cap D^y$, and we say that $D$ is \textit{bounded} if there is some $N\in \N$ with $|D_x|\leq N$ and $|D^x|\leq N$ for all $x\in \mathcal{G}^0$.
For subsets $A, B\subset \mathcal{G}^0$, we set $\mathcal{G}_A \defeq r^{-1}(A)\cap s^{-1}(A)$ and $\mathcal{G}_{A,B} \defeq r^{-1}(A)\cap s^{-1}(B)$.
Then $\calg_A$ is a groupoid with unit space $A$, with respect to the product inherited from $\calg$.
For $x\in \calg^0$, we call the group $\calg_x^x$ the \textit{isotropy group} of $\mathcal{G}$ at $x$.

A \textit{discrete Borel groupoid} is a groupoid $\mathcal{G}$ such that $\mathcal{G}$ is a standard Borel space, $\mathcal{G}^0$ is a Borel subset of $\mathcal{G}$, the source and range maps $s,r\colon \mathcal{G}\rightarrow \mathcal{G}^0$ are Borel and countable-to-one, and the multiplication map $\{ \, (\gamma ,\delta ) \in \mathcal{G}\times \mathcal{G}\mid s(\gamma )=r(\delta )\, \} \rightarrow \mathcal{G}$, $(\gamma ,\delta )\mapsto \gamma \delta$ and the inverse map $\gamma \mapsto \gamma ^{-1}$ are both Borel.
A \textit{cocycle} $\alpha \colon \calg \to L$ into a standard Borel group $L$ is a Borel groupoid-homomorphism, i.e., a Borel map satisfying $\alpha (\gamma \delta )= \alpha (\gamma )\alpha (\delta )$ whenever $s(\gamma )=r(\delta )$.

A \textit{discrete p.m.p.\ groupoid} is a pair $(\mathcal{G},\mu )$ of a discrete Borel groupoid $\calg$ and a Borel probability measure $\mu$ on $\mathcal{G}^0$ satisfying $\int _{\mathcal{G}^0}c^s_x \, d\mu(x)= \int _{\mathcal{G}^0} c^r_x\, d\mu(x)$, where $c^s_x$ and $c^r_x$ are the counting measures on $\calg_x$ and $\calg^x$, respectively. 
We will write $\mu ^1$ for this common measure: $\mu ^1 \defeq \int _{\mathcal{G}^0}c^s_x \, d\mu(x) = \int _{\mathcal{G}^0} c^r_x\, d\mu(x)$.

Let $(\mathcal{G},\mu )$ be a discrete p.m.p.\ groupoid.
We say that $(\mathcal{G},\mu )$ is \textit{aperiodic} if $\calg_x$ is infinite for $\mu$-almost every $x\in \mathcal{G}^0$.
A Borel subset $A\subset \calg^0$ is called \textit{$\calg$-invariant} if $r(\calg_x)\subset A$ for $\mu$-almost every $x\in A$.
We say that $(\calg, \mu)$ is \textit{ergodic} if for every $\calg$-invariant Borel subset $A\subset \calg^0$, we have $\mu(A)=0$ or $1$.
A \textit{local section} of $\mathcal{G}$ is a Borel subset $\phi \subset \calg$ on which the source and range maps $s$, $r$ are both injective.
We call $s(\phi)$ and $r(\phi)$ the domain and range of $\phi$, respectively.
For a local section $\phi$ of $\calg$ and points $x\in s(\phi)$, $y\in r(\phi)$, we denote by $\phi_x, \phi^y\in \phi$ the unique elements such that $s(\phi_x)=x$ and $r(\phi^y)=y$.
Let $\phi^o\colon s(\phi)\to r(\phi)$ denote the associated map given by $x\mapsto r(\phi_x)$.
For a Borel subset $A\subset \calg^0$, we set
\[\phi A\defeq \phi^o(A\cap s(\phi))=\{ \, r(\phi_x)\mid x\in A\cap s(\phi)\, \}.\]
We identify two local sections if they agree up to a $\mu$-null set.
The \textit{composition} of two local sections $\phi$, $\psi$ is the local section $\phi \psi \defeq \{ \, \phi_{r(\psi_x)} \psi_x\mid x\in s(\psi)\cap (\psi^o)^{-1}(s(\phi))\, \}$.
The \textit{inverse} of a local section $\phi$ is the local section $\phi^{-1}\defeq \{ \, (\phi_x)^{-1}\mid x\in s(\phi)\, \}$, whose domain is $r(\phi)$.
Let $[\mathcal{G}]$ denote the group of all local sections $\phi$ of $\mathcal{G}$ with $s(\phi )=\mathcal{G}^0$ (and hence $r(\phi)=\calg^0$), and call $[\calg]$ the \textit{full group} of $\calg$.
We have the \textit{uniform topology} on $[\calg]$ induced by the metric $\delta_u(\phi , \psi ) \defeq \mu ( \{ \, x\in \mathcal{G}^0 \mid \phi_x \neq \psi_x \, \} )$.
This metric is complete and makes $[\calg]$ into a Polish group (Lemma \ref{lem-polish}).


Let $\phi$ be a local section of $\mathcal{G}$.
For $\gamma \in \mathcal{G}_{r(\phi)}$, we set
\[\gamma ^{\phi } \defeq [\phi^{r(\gamma)}]^{-1}\gamma \phi^{s(\gamma)} \in \mathcal{G}_{s(\phi )}.\]
For a subset $D\subset \calg$, we set
\[D^\phi \defeq \{ \, \gamma^\phi \mid \gamma \in D\cap \calg_{r(\phi)}\, \}.\]
For a function $f\colon \mathcal{G} \to \C$, we define $f^{\phi }\colon \mathcal{G}\rightarrow \C$ by
\[
f^\phi (\gamma ) \defeq
\begin{cases}
f(\gamma ^{[\phi ^{-1}]} )=f(\phi_{r(\gamma)}\gamma [\phi_{s(\gamma)}]^{-1}) &\text{if }\gamma\in \mathcal{G}_{s(\phi )}, \\
0 &\text{otherwise}.
\end{cases}
\]
If $\psi$ is another local section of $\calg$, then $(\gamma^{\phi})^\psi =\gamma^{\phi \psi}$ and $(f^\phi)^\psi =f^{\phi \psi}$.

A discrete p.m.p.\ groupoid is called \textit{principal} if the map $\gamma \mapsto (r(\gamma ), s(\gamma ))$ is injective.
Let $\mathcal{R}$ be a p.m.p.\ countable Borel equivalence relation on a standard probability space $(X,\mu )$.
Then the pair $(\mathcal{R},\mu )$ is naturally a principal discrete p.m.p.\ groupoid with unit space $\mathcal{R}^0 \defeq \{ \, (x,x) \mid x\in X \, \}$, which are simply identified with $X$ itself when there is no cause for confusion.
The source and range maps are given by $s(y,x)=x$ and $r(y,x)=y$, respectively, and the multiplication and inverse operations are given by $(z,y)(y,x)=(z,x)$ and $(y,x)^{-1}=(x,y)$, respectively.
We mean by a \textit{discrete p.m.p.\ equivalence relation} on a standard probability space $(X, \mu)$ a p.m.p.\ countable Borel equivalence relation on $(X, \mu)$ equipped with this structure of a discrete p.m.p.\ groupoid.
For such an equivalence relation $\calr$, each local section $\phi$ of $\mathcal{R}$ is identified with the graph $\{ \, (\phi^o(x), x)\mid x\in s(\phi)\, \}$ of the associated map $\phi^o$.
We will abuse notation and identify $\phi$ and $\phi^o$ when there is no cause for confusion.
Then for all $(y,x)\in \mathcal{R}$ and $\phi \in [\mathcal{R}]$, we have $(y,x)^{\phi} = (\phi ^{-1}(y),\phi ^{-1}(x) )$.

The \textit{translation groupoid} associated to a p.m.p.\ action $G\curvearrowright (X,\mu )$ of a countable group $G$ is the groupoid $G\ltimes (X,\mu ) = (\mathcal{G},\mu )$ defined as follows:
The set of groupoid elements is $\mathcal{G} \defeq G\times X$ with unit space $\mathcal{G}^0\defeq \{ 1_G \} \times X$, which are once again identified with $X$ itself when there is no cause for confusion.
The source and range maps $s,r\colon \mathcal{G}\rightarrow \mathcal{G}^0$ are given by $s(g,x)=x$ and $r(g,x)=gx$, respectively, and the multiplication and inverse operations are given by $(g,hx)(h,x)= (gh,x)$ and $(g,x)^{-1}= (g^{-1}, gx )$, respectively.
The group $G$ embeds into $[\mathcal{G}]$ via the map $g\mapsto \phi _g\defeq \{ g\} \times X$. 
Then for all $(h,x)\in G\times X$ and $g\in G$, we have $(h,x)^{\phi _g}= (g^{-1}hg,g^{-1}x)$.
If the action $G\c (X, \mu)$ is essentially free, i.e., the stabilizer of $\mu$-almost every point of $X$ is trivial, then the groupoid $G\ltimes (X,\mu )$ is naturally isomorphic to the orbit equivalence relation
\[\mathcal{R}(G\curvearrowright (X,\mu ))\defeq \{ \, (gx,x) \mid g\in G,\, x \in X\, \}\]
of the action.

For a standard probability space $(X, \mu)$, let $\aut(X, \mu)$ be the group of Borel automorphisms of $X$ preserving $\mu$, where two such automorphisms are identified if they agree $\mu$-almost everywhere.
Unless otherwise stated, we endow $\aut(X, \mu)$ with the \textit{weak topology}, whose open basis is given by the sets
\[\{ \, S\in \aut(X, \mu)\mid \mu(S(A_i)\bigtriangleup T(A_i))<\ve \ \text{for all} \ i= 1,\ldots, n\, \}\]
for $T\in \aut(X, \mu)$, Borel subsets $A_1,\ldots, A_n\subset X$ and $\ve >0$.
We refer to \cite[Section 1]{kec} for details.



\section{Inner amenable groupoids}\label{sec-iag}

\subsection{Definition and equivalent conditions} 

We define inner amenability for discrete p.m.p.\ groupoids and state several conditions equivalent to it.
The proof of their equivalence is postponed to Subsection \ref{subsec-proof}, following the preliminary Subsections \ref{subsec-conj} and \ref{subsec-ame}.

\begin{defn}\label{def:groupoid}
A discrete p.m.p.\ groupoid $(\mathcal{G},\mu )$ is \textit{inner amenable} if there exists a sequence $(\xi _n )_{n\in \N}$ of non-negative unit vectors in $L^1(\mathcal{G} , \mu ^1 )$ such that
\begin{enumerate}
\item[(i)] $\| 1_{\mathcal{G}_A} \xi _n \| _1 \to \mu (A)$ for every Borel subset $A\subset \mathcal{G}^0$;
\item[(ii)] $\| \xi _n ^\phi - \xi _n  \| _1 \to 0$ for every $\phi \in [\mathcal{G}]$; 
\item[(iii)] $\| 1_{D}\xi _n \| _1 \to 0$ for every Borel subset $D\subset \mathcal{G}$ with $\mu ^1 (D)<\infty$; and
\item[(iv)] $\sum _{\gamma \in \calg^x}\xi _n (\gamma ) = 1 = \sum _{\gamma \in \calg_x}\xi _n (\gamma )$ for $\mu$-almost every $x\in \mathcal{G}^0$ and every $n\in \N$.
\end{enumerate}
Such a sequence $(\xi _n ) _{n\in \N}$ is called an \textit{inner amenability sequence} for $(\mathcal{G},\mu )$.
\end{defn}

\begin{rem}
A discrete countable group $G$, being a discrete p.m.p.\ groupoid on a singleton, is inner amenable in the above sense if there exists a sequence $(\xi_n)$ of non-negative unit vectors in $\ell^1(G)$ such that for every $g\in G$, we have $\Vert \xi_n^g-\xi_n\Vert_1\to 0$ and $\xi_n(g)\to 0$, where the function $\xi_n^g$ on $G$ is given by $\xi_n^g(h)=\xi_n(ghg^{-1})$.
\end{rem}

\begin{rem}
We will see in Lemma \ref{lem:bal} that if $(\mathcal{G} ,\mu )$ is ergodic and non-amenable, then every sequence $(\xi_n)$ of non-negative unit vectors in $L^1(\calg, \mu^1)$ satisfying condition (ii) automatically satisfies condition (i).
We call a sequence $(\xi_n)$ satisfying condition (i) \textit{balanced}.
\end{rem}

\begin{defn}
A \textit{mean} on a discrete p.m.p.\ groupoid $(\mathcal{G},\mu )$ is a finitely additive, probability measure $\bm{m}$ on $\mathcal{G}$ which is defined on the algebra of all $\mu ^1$-measurable subsets of $\mathcal{G}$ and is absolutely continuous with respect to $\mu ^1$.
Equivalently, a mean on $\mathcal{G}$ is a state on $L^\infty (\mathcal{G},\mu ^1 )$.
A mean $\bm{m}$ on $(\mathcal{G},\mu )$ is called
\begin{itemize}
\item \textit{balanced} if $\bm{m}(\mathcal{G}_A ) = \mu (A)$ for every Borel subset $A\subset \calg^0$;
\item \textit{conjugation-invariant} if $\bm{m}(D^{\phi}) = \bm{m}(D)$ for every $\phi \in [\mathcal{G}]$ and every Borel subset $D\subset \mathcal{G}$;
\item \textit{diffuse} if $\bm{m}(D)=0$ for every Borel subset $D\subset \mathcal{G}$ with $\mu ^1 (D)<\infty$; and
\item \textit{symmetric} if $\bm{m}(D) = \bm{m}(D^{-1})$ for every Borel subset $D\subset \mathcal{G}$.
\end{itemize}
\end{defn}

\begin{rem}\label{rem-balanced}
Let $\bm{m}$ be a balanced mean on a discrete p.m.p.\ groupoid $(\calg, \mu)$.
Then for every Borel subset $A\subset \calg^0$, we have $1=\mu(A)+\mu(\calg^0\setminus A)=\bm{m}(\calg_A)+\bm{m}(\calg_{\calg^0\setminus A})$, and hence $\bm{m}(\calg_{\calg^0\setminus A, A})=0$.
Moreover, for every countable partition $\calg^0=\bigsqcup_kA_k$ into Borel subsets, we have $1=\sum_k\mu(A_k)=\sum_k\bm{m}(\calg_{A_k})$, and hence given a Borel subset $D_k\subset \calg_{A_k}$ for each $k$, we have $\bm{m}(\bigsqcup_k D_k)=\sum_k\bm{m}(D_k)$.
\end{rem}

\begin{thm}\label{thm:equiv}
Let $(\mathcal{G},\mu )$ be a discrete p.m.p.\ groupoid. If $(\mathcal{G},\mu )$ is ergodic, then the following conditions (1)--(6) are equivalent:
\begin{enumerate}
\item[(1)] The groupoid $(\mathcal{G}, \mu )$ is inner amenable.
\item[(2)] There exists a net $(\xi _i )$ (as opposed to a sequence) of non-negative unit vectors in $L^1(\mathcal{G},\mu ^1 )$ satisfying conditions (i)--(iv) of Definition \ref{def:groupoid}.
\item[(3)] There exists a net $(\xi _i )$ of non-negative unit vectors in $L^1(\mathcal{G},\mu ^1 )$ satisfying conditions (ii) and (iii) of Definition \ref{def:groupoid}.

\item[(4)] There exists a diffuse, conjugation-invariant mean on $(\mathcal{G},\mu )$.
\item[(5)] There exists a diffuse, conjugation-invariant mean on $(\mathcal{G},\mu )$ which is symmetric and balanced.

\item[(6)] There exists a positive linear map $P \colon L^\infty (\mathcal{G}, \mu ^1 ) \rightarrow L^\infty (\mathcal{G}^0 ,\mu )$ such that
\begin{itemize}
\item $P(1_{\mathcal{G}_A})=1_A$ for every Borel subset $A\subset \mathcal{G}^0$;
\item $P(F)=P(F^{-1})$ and $P(F^{\phi})= P(F)\circ \phi^o$ for every $F\in L^\infty (\mathcal{G},\mu ^1 )$ and every $\phi \in [\mathcal{G}]$, where the function $F^{-1}$ is defined by $F^{-1}(\gamma)=F(\gamma^{-1})$ for $\gamma \in \calg$; and 
\item $P(F)=0$ for every $F \in L^1(\mathcal{G},\mu ^1 )\cap L^{\infty}(\mathcal{G},\mu ^1 )$. 
\end{itemize}
\end{enumerate}
In general, even without assuming that $(\mathcal{G},\mu )$ is ergodic, conditions (1), (2), (5) and (6) are all equivalent.
\end{thm}

\begin{rem}\label{rem:nonerg}
In general, in the absence of ergodicity, condition (4) does not imply condition (5), since any groupoid $(\mathcal{G},\mu )\defeq (\mathcal{G}_0\sqcup  \mathcal{G}_1 , \mu _0/2 + \mu _1/2)$, with $(\mathcal{G}_0, \mu _0)$ ergodic and inner amenable, and $(\mathcal{G}_1,\mu _1 )$ ergodic and not inner amenable, satisfies condition (4) but not condition (5).
\end{rem}

\begin{rem}\label{rem:locsec}
Condition (ii) of Definition \ref{def:groupoid} immediately implies its own strengthening that $\| \xi _n ^{\phi} - 1_{\mathcal{G}_{s(\phi )}}\xi _n  \| _1 \to 0$ for all local sections $\phi$ of $\mathcal{G}$ since every local section of $\mathcal{G}$ can be extended to a local section with conull domain.
Likewise, every conjugation-invariant mean $\bm{m}$ on $(\mathcal{G},\mu )$ satisfies $\bm{m}(D^\phi ) = \bm{m}(D)$ for all local sections $\phi$ of $\mathcal{G}$ and all Borel subsets $D\subset \mathcal{G}_{r(\phi)}$.
It follows that if a discrete p.m.p.\ groupoid $(\mathcal{G},\mu )$ is inner amenable, then so is $(\mathcal{G}_A,\mu _A )$ for every Borel subset $A\subset \mathcal{G}^0$ with positive measure, where $\mu_A$ is the normalized restriction of $\mu$ to $A$.
For the converse, see Proposition \ref{prop:inflate}.
\end{rem}

\begin{rem}\label{rem:suffice}
Let $(\calg, \mu)$ be a discrete p.m.p.\ groupoid and let $G$ be a countable subgroup of $[\calg]$ which covers $\calg$, i.e., $\calg =\bigcup_{g\in G}g$ (for example, this is realized if $(\calg,\mu) =G\ltimes (X, \mu)$ is the translation groupoid associated with a p.m.p.\ action of a countable group $G$ on a standard probability space $(X, \mu)$, and each $g\in G$ is identified with the section $\{ g\} \times X$ of $\calg$).
Let $\bm{m}$ be a balanced mean on $(\mathcal{G},\mu )$ which is invariant under conjugation by all elements of $G$, i.e., $\bm{m}(D^g) = \bm{m}(D)$ for all Borel subsets $D\subset \mathcal{G}$ and all $g\in G$.
Then $\bm{m}$ is in fact invariant under conjugation by all elements of $[\mathcal{G}]$, verified as follows:
Pick $\phi \in [\mathcal{G}]$ and a Borel subset $D\subset \mathcal{G}$.
Since $\calg$ is covered by $G$, we have a decomposition $\calg^0=\bigsqcup_{g\in G}A_g$ into Borel subsets $A_g$ such that $A_g\subset \{ \, x\in \calg^0\mid \phi^x=g^x\, \}$.
Then $\calg^0=\bigsqcup_{g\in G}g^{-1}A_g$ since $\phi^{-1}A_g=g^{-1}A_g$.
By Remark \ref{rem-balanced}, we thus have
\begin{align*}
\bm{m}(D^{\phi} ) &= \bm{m}\Biggl(\bigsqcup _{g\in G} D^\phi \cap \mathcal{G}_{g^{-1}A_g} \Biggr)= \sum_{g\in G}\bm{m}( D^\phi \cap \mathcal{G}_{g^{-1}A_g}) = \sum _{g\in G}\bm{m}((D\cap \mathcal{G}_{A_g})^g) \\
&= \sum _{g\in G}\bm{m}(D\cap \mathcal{G}_{A_g}) = \bm{m}(D),
\end{align*}
which proves the desired conclusion.

To obtain this conclusion, the assumption of $\bm{m}$ being balanced is crucial:
We assume that the groupoid $(\calg, \mu)=G\ltimes (X, \mu)$ is associated with a p.m.p.\ action $G\c (X, \mu)$ of a countable group $G$.
Given a conjugation-invariant mean $m$ on $G$, we have the mean $\bm{m}$ on $(\calg, \mu)$ defined by $\bm{m}(D)\defeq \int_{G}\mu^1(D\cap (\{ g\} \times X))\, dm(g)$ for a Borel subset $D\subset \calg$.
This mean $\bm{m}$ is invariant under conjugation by all elements of $G$.
However $\bm{m}$ is not necessarily invariant under conjugation by all elements of $[\calg]$.
To see this, we assume that $G$ is inner amenable and let $m$ be a diffuse, conjugation-invariant mean on $G$.
Then $\bm{m}$ is also diffuse.
Assume also that $G$ is non-amenable and the action $G\c (X, \mu)$ is given by a Bernoulli shift.
Then the groupoid $G\ltimes (X, \mu)$ is not inner amenable by Corollary \ref{cor-ber}, and therefore admits no diffuse, conjugation-invariant mean by Theorem \ref{thm:equiv}.
Thus $\bm{m}$ is never a mean invariant under conjugation by all elements of $[\calg]$.
\end{rem}


\subsection{Conjugation-invariant means}\label{subsec-conj}

Before proving Theorem \ref{thm:equiv}, we prepare several lemmas saying that under mild assumption, every conjugation-invariant mean is automatically balanced or diffuse.
We will use the following characterization of amenability of an ergodic discrete p.m.p.\ groupoid.

\begin{lem}\label{lem:amengroupoid}
An ergodic discrete p.m.p.\ groupoid $(\mathcal{G},\mu )$ is amenable if and only if there exists a mean $\bm{m}$ on $(\mathcal{G}, \mu )$ which is right-invariant, i.e., satisfies $\bm{m}(R_{\phi}f)=\bm{m}(f)$ for all $\phi \in [\mathcal{G}]$ and $f\in L^{\infty}(\mathcal{G},\mu ^1 )$, where $R_{\phi} \colon L^{\infty}(\mathcal{G},\mu ^1 ) \rightarrow L^{\infty}(\mathcal{G},\mu ^1 )$ is the right translation map defined by $(R_{\phi}f)(\gamma ) \defeq f(\gamma [\phi_{s(\gamma )}] ^{-1})$.
\end{lem}

\begin{proof}
For principal groupoids, this is proved in \cite[Remark 4.67]{kerr-li}, whose proof involves the Connes-Feldman-Weiss theorem \cite{cfw}, however.
The following proof is more direct and applies to general groupoids.

If $(\calg, \mu)$ is amenable, then by its definition (\cite[Definition 3.2.8]{ar}), there exists a unital positive linear map $P\colon L^\infty(\calg, \mu^1)\to L^\infty(\calg^0, \mu)$ which is right-invariant, i.e., satisfies $P(R_\phi f)=P(f)\circ \phi^o$ for all $\phi \in [\calg]$ and $f\in L^{\infty}(\calg,\mu ^1 )$.
Then a right-invariant mean $\bm{m}$ on $(\calg, \mu)$ is defined by $\bm{m}(f)\defeq \int_{\calg^0} P(f)\, d\mu$.

Conversely, assume that there exists a right-invariant mean $\bm{m}$ on $(\calg, \mu)$.
We identify each Borel subset of $\calg$ with its indicator function that belongs to $L^\infty(\calg, \mu^1)$. 
We prove the equation $\bm{m}(s^{-1}(A))=\mu (A)$ holds for all Borel subsets $A\subset \mathcal{G}^0$. 
Since $(\calg, \mu)$ is ergodic, if $A, B\subset \calg^0$ are Borel subsets with $\mu(A)=\mu(B)$, then we have some $\phi \in [\calg]$ with $\phi A=B$ and hence $\bm{m}(s^{-1}(A))=\bm{m}(R_\phi [s^{-1}(B)])= \bm{m}(s^{-1}(B))$.
Therefore there is some function $\theta \colon [0, 1]\to [0, 1]
$ such that $\bm{m}(s^{-1}(A))=\theta(\mu(A))$ for every Borel subset $A\subset \calg^0$.
Given $n\in \N$, we can find a Borel partition $A_1,\dots , A_n$ of $\mathcal{G}^0$ with $\mu (A_1)=\cdots = \mu (A_n)= 1/n$. 
Then $1=\sum _{i=1}^n \bm{m}(s^{-1}(A_i)) = n\theta (1/n)$, so that $\theta (1/n ) = 1/n$ and
\[
\theta (k/n) = \bm{m}\biggl(s^{-1}\biggl( \, \bigsqcup _{i=1}^k A_i\biggr) \biggr) = \sum _{i=1}^k \bm{m}(s^{-1}(A_i)) = k\theta (1/n ) = k/n
\]
for every $k\in \{ 1,\ldots, n\}$.
Therefore $\theta (q)=q$ for all rational $q\in [0,1]$.
Since $\theta$ is monotone increasing, this implies that $\theta (r)=r$ for all $r\in [ 0, 1]$, and hence $\bm{m}(s^{-1}(A))=\mu (A)$ for all Borel subsets $A\subset \mathcal{G}^0$.

For each $f\in L^\infty(\calg, \mu^1)$, we define a (countably additive) finite, complex Borel measure $\mu_f$ on $\calg^0$ by $\mu_f(A)\defeq \bm{m}(1_{s^{-1}(A)}f)$, which is absolutely continuous with respect to $\mu$.
Countable additivity of $\mu_f$ follows from the equation proved in the last paragraph.
Then the map $P\colon L^\infty(\calg, \mu^1)\to L^\infty(\calg^0, \mu)$ defined by $P(f)\defeq d\mu_f/d\mu$ is a unital positive linear map such that $P(R_\phi f)=P(f)\circ \phi^o$ for all $\phi \in [\calg]$ and $f\in L^\infty(\calg, \mu^1)$.
\end{proof}


\begin{lem}\label{lem:nonamenbal}
Let $(\mathcal{G},\mu )$ be an ergodic discrete p.m.p.\ groupoid which is non-amenable, and let $\bm{m}$ be a conjugation-invariant mean on $(\mathcal{G},\mu )$.
If $A\subset \mathcal{G}^0$ is a Borel subset with positive measure, then $\bm{m}(\mathcal{G}_{{\mathcal{G}^0}\setminus A, A}) =0$.
\end{lem}

\begin{proof}
Suppose toward a contradiction that $\bm{m}(\mathcal{G}_{{\mathcal{G}^0}\setminus A, A}) >0$, and let $\bm{m}_0$ denote the normalized restriction of $\bm{m}$ to $\mathcal{G}_{{\mathcal{G}^0}\setminus A, A}$.
For $\phi \in [\mathcal{G}_A]$, let $\phi^\sim \in [\mathcal{G}]$ denote the extension of $\phi$ defined by $(\phi^\sim)_y= y$ for $y\in {\mathcal{G}^0}\setminus A$ (and $(\phi^\sim)_x = \phi_x$ for $x\in A$).
Then for every $\phi \in [\mathcal{G}_A]$, both left and right translations by $\phi^\sim$ fix the set $\mathcal{G}_{{\mathcal{G}^0}\setminus A, A}$, and for every $f\in L^{\infty}(\mathcal{G},\mu ^1 )$ supported on $\mathcal{G}_{{\mathcal{G}^0}\setminus A, A}$, we have $f^{\phi^\sim} = R_{\phi^\sim}f$ and hence
\begin{align}\label{eqn:fvan}
\bm{m}_0(R_{\phi^\sim}f) &= \bm{m}_0(f) .
\end{align}
Since $(\mathcal{G},\mu )$ is ergodic, we can find a Borel map $x\mapsto T(x)\in \calg_x$ with $r(T(x))\in A$ for almost every $x\in {\mathcal{G}^0}$.
For $f\in L^{\infty}(\mathcal{G}_A,\mu ^1 _A )$, we define $L_Tf\in L^{\infty}(\mathcal{G},\mu ^1 )$ by
\[
(L_Tf)(\gamma ) \defeq
\begin{cases}
f(T(r(\gamma ))\gamma ) &\text{if }\gamma \in \mathcal{G}_{{\mathcal{G}^0}\setminus A, A},\\
0 &\text{otherwise}.
\end{cases}
\]
Then $L_TR_{\phi}f=R_{\phi^\sim}L_Tf$ for every $\phi \in [\mathcal{G}_A]$ and every $f\in L^{\infty}(\mathcal{G}_A, \mu ^1 _A )$.
Define a mean $\bm{m}_1$ on $(\mathcal{G}_A, \mu _A)$ by $\bm{m}_1(f) \defeq \bm{m}_0 (L_Tf)$.
Then for every $\phi \in [\mathcal{G}_A]$ and every $f\in L^{\infty}(\mathcal{G}_A,\mu ^1 _A )$, by equation \eqref{eqn:fvan}, we have
\[
\bm{m}_1(R_{\phi}f)= \bm{m}_0 (L_TR_{\phi}f)=\bm{m}_0(R_{\phi^\sim}L_Tf) = \bm{m}_0(L_Tf)= \bm{m}_1(f) .
\]
Thus $\bm{m}_1$ is a right-invariant mean on $(\mathcal{G}_A, \mu _A )$, and hence $(\mathcal{G}_A, \mu _A)$ is amenable by Lemma \ref{lem:amengroupoid}.
Since $A$ has positive measure and $(\mathcal{G},\mu )$ is ergodic, this implies $(\mathcal{G},\mu )$ is amenable, a contradiction.
\end{proof}

\begin{rem}
In Lemma \ref{lem:nonamenbal}, non-amenability of $(\mathcal{G},\mu )$ is necessary:
Let the group $G\defeq \bigoplus_\N \Z /2\Z$ act on the compact group $X\defeq \prod_\N \Z /2\Z$ by translation, and equip $X$ with the normalized Haar measure $\mu$.
Let $\calr$ be the associated orbit equivalence relation.
For $n\in \N$, define the subgroup $F_n\defeq \bigoplus_{k=1}^n\Z /2\Z$ of $G$, and let $\calr_n$ be the subrelation of $\calr$ generated by $F_n$.
Define the non-negative unit vector $\xi_n\defeq 1_{\calr_n}/2^n\in L^1(\calr, \mu^1)$, and let $\bm{m}$ be any weak${}^*$-cluster point of the sequence $(\xi_n)$ in $L^\infty(\calr, \mu^1)^*$.
Then $\bm{m}$ is a mean on $(\calr, \mu)$ which is left and right-invariant and hence conjugation-invariant.
However, if $A\defeq \{ \, (x_k)\in X \mid x_1=0\, \}$, then $\int_{\calr}\xi_n1_{\calr_{A, X\setminus A}}\, d\mu^1=\mu(A)(2^{n-1}/2^n)=1/4$ for every $n$, and therefore $\bm{m}(\calr_{A, X\setminus A})=1/4\neq 0$.
\end{rem}

\begin{lem}\label{lem:bal}
Let $(\mathcal{G},\mu )$ be an ergodic discrete p.m.p.\ groupoid which is non-amenable.
Then every conjugation-invariant mean on $(\mathcal{G},\mu )$ is balanced.

It follows that if $(\xi _n )$ is any sequence of non-negative unit vectors in $L^1(\mathcal{G},\mu ^1)$ satisfying $\| \xi _n ^{\phi} - \xi _n \| _1\rightarrow 0$ for all $\phi \in [\mathcal{G}]$, then $(\xi _n )$ is balanced.
\end{lem}

\begin{proof}
Let $\bm{m}$ be a conjugation-invariant mean on $(\mathcal{G},\mu )$.
We follow the argument in the proof of Lemma \ref{lem:amengroupoid}.
Since $(\mathcal{G},\mu )$ is ergodic, if $A,B\subset \mathcal{G}^0$ are Borel subsets with $\mu (A) = \mu (B)$, then we have some $\phi \in [\mathcal{G}]$ with $\phi A= B$ and hence $\bm{m}(\mathcal{G}_{A}) = \bm{m} ((\mathcal{G}_B)^{\phi})= \bm{m}(\mathcal{G}_B)$.
Therefore there is some function $\theta \colon [0,1]\rightarrow [0,1]$ such that $\bm{m}(\mathcal{G}_A)= \theta (\mu (A))$ for every Borel subset $A\subset  \mathcal{G}^0$.
Given $n\in \N$, we can find a Borel partition $A_1,\dots , A_n$ of $\mathcal{G}^0$ with $\mu (A_1)=\cdots = \mu (A_n)= 1/n$, and Lemma \ref{lem:nonamenbal} implies that $1=\sum _{i=1}^n \bm{m}(\mathcal{G}_{A_i}) = n\theta (1/n)$, so that $\theta (1/n ) = 1/n$ and
\[
\theta (k/n) = \bm{m}(\mathcal{G}_{\bigsqcup _{i=1}^k A_i} ) = \sum _{i=1}^k \bm{m}(\mathcal{G}_{A_i}) = k\theta (1/n ) = k/n
\]
for every $k\in \{ 1,\ldots, n\}$.
Therefore $\theta (q)=q$ for all rational $q\in [0,1]$.
Since $\theta$ is monotone increasing, this implies that $\theta (r)=r$ for all $r\in [ 0, 1]$, and hence $\bm{m}(\mathcal{G}_A)=\mu (A)$ for all Borel subsets $A\subset \mathcal{G}^0$, i.e., $\bm{m}$ is balanced.
\end{proof}

\begin{lem}\label{lem:diffuse}
Let $(\mathcal{G},\mu )$ be a discrete p.m.p.\ groupoid and let
\[\mathcal{G}_{\mathrm{isot}} \defeq \{ \, \gamma \in \mathcal{G} \mid s(\gamma ) =r(\gamma ) \, \}\]
be the isotropy subgroupoid of $\calg$.
Let $\bm{m}$ be a balanced mean on $(\calg, \mu)$.
Then $\bm{m}(D)=0$ for all Borel subsets $D\subset \calg \setminus \calg_{\mathrm{isot}}$ with $\mu^1(D)<\infty$.
If $\bm{m}$ further satisfies $\bm{m}(E)=0$ for all bounded Borel subsets $E\subset \mathcal{G}_{\mathrm{isot}}$, then $\bm{m}$ is diffuse.

It follows that if $\mathcal{R}$ is a discrete p.m.p.\ equivalence relation on $(X,\mu )$, and if $\bm{m}$ is a balanced mean on $(\mathcal{R},\mu )$ satisfying $\bm{m} (\{ \, (x,x) \mid x\in X \, \} ) =0$, then $\bm{m}$ is diffuse.
\end{lem}

\begin{proof}
To prove the first assertion, we suppose toward a contradiction that there is some Borel subset $D\subset \calg \setminus \calg_{\mathrm{isot}}$ with $\bm{m}(D)>0$ and $\mu ^1(D)<\infty$.
Since $\mu ^1(D)<\infty$, we have both $|D^x|<\infty$ and $|D_x| < \infty$ for $\mu$-almost every $x\in \mathcal{G}^0$.
Then the sets
\[A_{n,m} \defeq  \{ \, x\in \mathcal{G}^0\mid |D_x| = n \text{ and } |D^x|=m \, \}\]
with non-negative integers $n$, $m$ partition $\mathcal{G}^0$.
Since $\bm{m}$ is balanced, we have $0<\bm{m}(D)=\sum _{n,m} \bm{m} (\mathcal{G}_{A_{n,m}}\cap D)$.
Thus, after replacing $D$ by one of the sets $\mathcal{G}_{A_{n,m}}\cap D$ if necessary, we may assume without loss of generality that $D$ is bounded.

Since $D$ is bounded, it is covered by finitely many local sections of $\mathcal{G}$, and hence we can find a local section $\phi \subset D$ of $\mathcal{G}$ with $\bm{m}(\phi ) > 0$.
We set $A\defeq s(\phi)$.
Since $D\cap \mathcal{G}_{\mathrm{isot}}=\emptyset$, we have $\phi^o(x)\neq x$ for all $x\in A$.
We can partition $A$ into three Borel subsets $A_0$, $A_1$ and $A_2$ such that $\phi A_i\cap A_i =\emptyset$ (modulo $\mu$) for every $i\in \{ 0, 1, 2\}$.
Indeed, take $A_0$ to be any maximal (modulo $\mu$) Borel subset of $A$ with $\phi A_0\cap A_0 =\emptyset$, and set $A_1\defeq \phi A_0\cap A$ and $A_2 \defeq A\setminus (A_0\cup A_1)$.
This works since we then have $\phi A_1\cap A_1 \subset \phi (A\setminus A_0)\cap \phi A_0=\emptyset$, and $\mu (\phi A_2 \cap A_2 )=0$ by maximality of $A_0$.
We set $\phi_i\defeq \{ \, \phi_x\mid x\in A_i\, \}$ for $i\in \{ 0,1,2\}$.
Then $\phi_i \subset \mathcal{G}_{\mathcal{G}^0\setminus A_i , A_i }$, and hence $\bm{m}(\phi_i)= 0$ since $\bm{m}$ is balanced. It therefore follows that $\bm{m}(\phi)=\bm{m}(\phi_0)+ \bm{m}(\phi_1)+ \bm{m}(\phi_2) = 0$, a contradiction.

For the second assertion, suppose that $\bm{m}$ satisfies $\bm{m}(E)=0$ for all bounded Borel subsets $E\subset \calg_{\mathrm{isot}}$.
Let $D\subset \calg$ be a Borel subset with $\mu^1(D)<\infty$.
We prove $\bm{m}(D)=0$.
The argument in the first paragraph of the proof shows that we may assume that $D$ is bounded.
Then $\bm{m}(D\cap \calg_{\mathrm{isot}})=0$ by the assumption on $\bm{m}$, and $\bm{m}(D\setminus \calg_{\mathrm{isot}})=0$ by the first assertion of the lemma proved above.
Thus $\bm{m}(D)=0$, and $\bm{m}$ is diffuse.
\end{proof}


\subsection{Amenability and inner amenability}\label{subsec-ame}

Every discrete, countably infinite, amenable group is inner amenable.
We extend this to a discrete p.m.p.\ groupoid.

\begin{prop}\label{prop:extend}
Let $(\mathcal{G},\mu )$ be a discrete p.m.p.\ groupoid.
Let
\[
\calg_{\mathrm{isot}}\defeq \{ \, \gamma \in \calg \mid s(\gamma)=r(\gamma)\, \} \quad \text{and}\quad \mathcal{R}_{\mathcal{G}} \defeq  \{ \, (r(\gamma ), s(\gamma ))\mid \gamma \in \mathcal{G} \, \}
\]
be the isotropy subgroupoid of $\calg$ and the equivalence relation associated to $\calg$, respectively.
Suppose that $(\calg_{\mathrm{isot}}, \mu)$ is inner amenable and $(\calr_\calg, \mu)$ is hyperfinite.
Then $(\mathcal{G},\mu )$ is inner amenable, and moreover there exists an inner amenability sequence $(\xi _n)$ for $(\mathcal{G}, \mu)$ such that each $\xi_n$ is supported on $\mathcal{G}_{\mathrm{isot}}$.
\end{prop}

We note that by Proposition \ref{prop:ergdec}, the groupoid $(\calg_{\mathrm{isot}}, \mu)$ being inner amenable is equivalent to the isotropy group $\calg_x^x$ being inner amenable for $\mu$-almost every $x\in \calg^0$.

\begin{proof}[Proof of Proposition \ref{prop:extend}]
By hypothesis, we can write $\mathcal{R}_{\mathcal{G}}=\bigcup _n \mathcal{R}_n$, where $\mathcal{R}_1\subset \mathcal{R}_2\subset \cdots$ is an increasing sequence of Borel equivalence relations on $\mathcal{G}^0$ which are bounded subsets of $\calr_\calg$.
We can also find a Borel section $\sigma \colon \calr_\calg \to \calg$ of the quotient map which is a groupoid-homomorphism.
For each $n$, let $X_n\subset \mathcal{G}^0$ be a Borel transversal for $\mathcal{R}_n$, i.e., a Borel subset of $\calg^0$ which meets each $\mathcal{R}_n$-equivalence class in exactly one point.
We choose a countable dense subset $\{ \phi _n \}_{n\in \N}\subset [\mathcal{G}]$ such that $(x, \phi _i ^o(x)) \in \mathcal{R}_n$ for all $i\leq n$ and all $x\in \mathcal{G}^0$.
Let $D_1\subset D_2\subset \cdots$ be an exhaustion of $\calg_{\mathrm{isot}}$ by its bounded Borel subsets.

For $x\in \calg^0$, let $[x]_{\calr_n}$ be the $\calr_n$-equivalence class of $x$.
For each $n$, we define the set
\[D_n'\defeq \{ \, \sigma(x, y)\gamma \sigma(y, x)\mid x\in X_n,\, y\in [x]_{\calr_n} \ \text{and}\ \ \gamma \in D_n\cap \calg_y^y\, \},\]
which is a bounded Borel subset of $\calg_{\mathrm{isot}}$.
For $n\in \N$ and $x\in X_n$, we also define the set
\[F^n_x\defeq \{ \, \sigma(x, r((\phi_i)_y))(\phi_i)_y\sigma(y, x) \mid y\in [x]_{\mathcal{R}_n} \text{ and }i\leq n  \, \},\]
which is a finite subset of $\calg_x^x$.  
Since $(\calg_{\mathrm{isot}}, \mu)$ is inner amenable, for every $n\in \N$, we can find a Borel family $(\eta _n^ x)_{x\in X_n}$ of non-negative unit vectors $\eta ^x_n \in \ell ^1(\mathcal{G}_x^x)$ such that
\[\int_{X_n}|[x]_{\calr_n}|\sum_{\gamma \in \calg_x^x}\eta_n^x(\gamma)1_{D_n'}(\gamma)\, d\mu(x)<\frac{1}{n},\]
and $\Vert (\eta ^x_n )^{\delta} - \eta ^x_n \Vert_1 < 1/n$ for all $x\in X_n$ and all $\delta \in F_x^n$.
We define a non-negative unit vector $\xi _n \in L^1(\mathcal{G}, \mu^1)$ by $\xi _n (\gamma ) = 0$ for $\gamma \in \calg \setminus \mathcal{G}_{\mathrm{isot}}$, and
\[
\xi _n (\gamma ) = \eta_n^{x}(\sigma(x, y)\gamma \sigma(y, x))
\]
for $\gamma \in \mathcal{G}_y^y$ with $y\in \mathcal{G}^0$, where $x$ is the unique point in $[y]_{\calr_n}\cap X_n$.
It is clear that the sequence $(\xi _n)$ satisfies conditions (i) and (iv) of Definition \ref{def:groupoid}.
Condition (iii) follows from:
\begin{align*}
\Vert 1_{D_n}\xi_n\Vert_1&=\int_{X_n}\sum_{y\in [x]_{\calr_n}}\sum_{\gamma \in \calg_y^y}\eta_n^x(\sigma(x, y)\gamma \sigma(y, x))1_{D_n}(\gamma)\, d\mu(x)\\
&\leq \int_{X_n}\sum_{y\in [x]_{\calr_n}}\sum_{\gamma \in \calg_y^y}\eta_n^x(\sigma(x, y)\gamma \sigma(y, x))1_{D_n'}(\sigma(x, y)\gamma \sigma(y, x))\, d\mu(x)\\
&=\int_{X_n}|[x]_{\calr_n}|\sum_{\gamma \in \calg_x^x}\eta_n^x(\gamma)1_{D_n'}(\gamma)\, d\mu(x)<\frac{1}{n}.
\end{align*}
To verify condition (ii), it suffices to show that $\| \xi _n ^{\phi _i}- \xi _n \| _1 \to 0$ for all $i\in \N$.
For all $i\leq n$, $x\in X_n$, $y\in [x]_{\mathcal{R}_n}$ and $\gamma \in \mathcal{G}_y^y$, we have $(\phi _i)_y\gamma [(\phi_i)_y]^{-1}\in \calg_z^z$, where $z\defeq r((\phi_i)_y)$.
Hence if we put $\delta \defeq \sigma(x, z)(\phi_i)_y\sigma(y, x)\in F^n_x$, then
\begin{align*}
\xi _n^{\phi _i}(\gamma )&= \xi _n ((\phi _i)_y\gamma [(\phi_i)_y]^{-1}) = \eta _n^x (\sigma(x, z)(\phi _i)_y\gamma [(\phi_i)_y]^{-1}\sigma(z, x) )\\
&=(\eta _n ^x )^{\delta} (\sigma(x, y)\gamma \sigma(y, x)),
\end{align*}
and thus
\[
\sum _{\gamma \in \mathcal{G}_y^y} |\xi _n^{\phi _i}(\gamma ) - \xi _n (\gamma ) | = \sum _{\gamma \in \mathcal{G}_x^x} |(\eta _n ^x )^{\delta}(\gamma ) - \eta _n ^x (\gamma )| < 1/n .
\]
It follows that $\| \xi _n ^{\phi _i}- \xi _n \| _1 < 1/n \rightarrow 0$.
\end{proof}

\begin{rem}\label{rem-gp}
The referee pointed out to us another proof of Proposition \ref{prop:extend} using the result of Giordano-Pestov \cite[Proposition 5.3]{gp}:
the full group of a hyperfinite discrete p.m.p.\ equivalence relation, endowed with the uniform topology, is \textit{extremely amenable}, i.e., all its continuous actions on a compact space have a fixed point.

The proof is as follows:
Let $(\calg, \mu)$ be a discrete p.m.p.\ groupoid such that the isotropy subgroupoid $(\calg_{\mathrm{isot}}, \mu)$ is inner amenable and the associated equivalence relation $(\calr_\calg, \mu)$ is hyperfinite.
Let $M$ be the space of balanced diffuse means on $(\calg, \mu)$, which is compact with respect to the weak${}^*$-topology in $L^\infty(\calg, \mu)^*$.
Then the full group $[\calg]$ acts on $M$ by conjugation continuously (note that the balanced property of means is used in deducing this continuity).
Inner amenability of $(\calg_{\mathrm{isot}}, \mu)$ implies that there exists a balanced, diffuse, conjugation-invariant mean on $(\calg_{\mathrm{isot}}, \mu)$, which is identified with a point of $M$ fixed by the full group $[\calg_{\mathrm{isot}}]$.
Let $M_0$ be the space of points of $M$ fixed by $[\calg_{\mathrm{isot}}]$, on which the full group $[\calr_\calg]$ acts naturally (note that $[\calr_\calg]$ is identified with the quotient group of $[\calg]$ divided by $[\calg_{\mathrm{isot}}]$).
Now extreme amenability of $[\calr_\calg]$ implies that it has a fixed point in $M_0$.
That point is fixed by $[\calg]$ and is thus a conjugation-invariant mean on $(\calg, \mu)$.
\end{rem}

\begin{prop}\label{prop:aperamen}
Let $(\mathcal{G},\mu )$ be an amenable discrete p.m.p.\ groupoid which is aperiodic.
Then $(\mathcal{G},\mu )$ is inner amenable.
\end{prop}

\begin{proof} 
By \cite[Corollary 5.3.33]{ar}, the groupoid $(\calg_{\mathrm{isot}}, \mu)$ is amenable, and the equivalence relation $\mathcal{R}_{\mathcal{G}} \defeq  \{ \, (r(\gamma ), s(\gamma )) \mid \gamma \in \mathcal{G}\, \}$ is amenable and is hence hyperfinite by \cite{cfw}.
By restricting $\calg$ to a $\calg$-invariant Borel subset of $\calg^0$, we may assume that there is some $M\in \N \cup \{ \infty \}$ such that $|\mathcal{G}_x^x|=M$ for almost every $x\in \mathcal{G}^0$.
If $M=\infty$, then we are done by Proposition \ref{prop:extend} because infinite amenable groups are inner amenable.

Suppose that $M\in \N$.
Since $\mathcal{G}$ is aperiodic and almost every $\mathcal{G}_x^x$ is finite, the equivalence relation $\mathcal{R}_{\mathcal{G}}$ is aperiodic and hyperfinite.
Hence $[\mathcal{R}_{\mathcal{G}}]$ admits a central sequence $(T_n)_{n\in\N}$ with $T_nx\neq x$ for all $n\in \N$ and all $x\in \mathcal{G}^0$.
For each $n\in \N$, let $\eta _n\in L^1(\calr_\calg, \mu^1)$ be the indicator function of the graph $\{ \, (T_nx,x) \mid x\in \mathcal{G}^0 \, \}$.
Then $(\eta  _n )$ is an inner amenability sequence for $(\mathcal{R}_{\mathcal{G}},\mu )$.
Define $\xi _n \colon \mathcal{G}\rightarrow [0,1]$ by $\xi _n (\gamma ) \defeq \eta _n (r(\gamma ) , s(\gamma ))/M$.
Then $(\xi _n)$ is an inner amenability sequence for $(\mathcal{G},\mu )$.
\end{proof}


\subsection{Proof of Theorem \ref{thm:equiv}}\label{subsec-proof}

The only place where we use ergodicity is in the proof of the implication (4)$\Rightarrow$(5):
Assume that $(\mathcal{G},\mu )$ is ergodic and that condition (4) holds, and let $\bm{m}$ be a diffuse, conjugation-invariant mean on $(\mathcal{G},\mu )$.
Let $\check{\bm{m}}$ be the mean defined by $\check{\bm{m}}(D)\defeq \bm{m}(D^{-1})$.
After replacing $\bm{m}$ by $(\bm{m} + \check{\bm{m}})/2$ if necessary, we may assume without loss of generality that $\bm{m}$ is symmetric.
Since $\bm{m}$ is diffuse, $(\mathcal{G},\mu )$ must be aperiodic.
Thus if $(\mathcal{G},\mu )$ is amenable, then condition (1) holds by Proposition \ref{prop:aperamen} and hence condition (5) holds, where the implication (1)$\Rightarrow$(5) will be proved for a general $(\calg, \mu)$ in the next paragraph.
If $(\mathcal{G},\mu )$ is non-amenable, then $\bm{m}$ is balanced by Lemma \ref{lem:bal}, and hence condition (5) holds in this case as well.

For the rest of the proof, we no longer assume that $(\mathcal{G},\mu )$ is ergodic.
The implications (1)$\Rightarrow$(2)$\Rightarrow$(3) are clear.
The implication (2)$\Rightarrow$(1) follows from separability of $[\mathcal{G}]$ and of $L^1(\mathcal{G}^0,\mu )$. 
The implication (3)$\Rightarrow$(4) follows from weak${}^*$-compactness of the set of means on $\mathcal{G}$, by identifying both $L^1(\mathcal{G},\mu ^1 )$ and the set of means on $(\mathcal{G},\mu ^1 )$ with subsets of $L^{\infty}(\mathcal{G},\mu ^1 )^*$:
If $(\xi _i )$ is a net as in condition (3), then any weak${}^*$-cluster point of $(\xi _i )$ in $L^{\infty}(\mathcal{G},\mu ^1 )^*$ is a diffuse, conjugation-invariant mean on $(\mathcal{G},\mu )$.
The implication (2)$\Rightarrow$(5) is analogous:
If $(\xi _i )$ is a net as in condition (2), then after replacing $\xi _i$ by $\xi _i '\defeq  (\xi _i + \check{\xi} _i)/2$, where $\check{\xi}_i(\gamma ) \defeq  \xi _i(\gamma ^{-1})$, we can assume that each $\xi _i$ is symmetric, and hence any weak${}^*$-cluster point of $(\xi _i)$ is a mean on $(\mathcal{G},\mu )$ satisfying condition (5).

(5)$\Leftrightarrow$(6):
If $\bm{m}$ is a mean on $(\mathcal{G}, \mu )$ as in condition (5), then for each $F\in L^\infty (\mathcal{G},\mu ^1 )$, we can define a (countably additive!) finite, complex Borel measure $\mu _F$ on $\mathcal{G}^0$ by $\mu _F(A)\defeq \int 1_{\mathcal{G}_A}F\, d\bm{m}$, which is absolutely continuous with respect to $\mu$.
Countable additivity of $\mu _F$ follows from $\bm{m}$ being balanced.
Then the map $P \colon L^\infty (\mathcal{G}, \mu ^1 ) \rightarrow L^\infty (\mathcal{G}^0,\mu )$ defined by $P(F) \defeq d\mu _F/d\mu$ verifies condition (6).
Conversely, if $P$ is as in condition (6), then $\bm{m}(D)\defeq \int _{\mathcal{G}^0} P(1_D)\, d\mu$ defines a mean on $(\mathcal{G},\mu )$ as in condition (5).

It remains to prove the implication (5)$\Rightarrow$(2). 
For each $\eta \in L^1(\mathcal{G} ,\mu ^1 )$, we define $\check{\eta}$ by $\check{\eta}(\gamma  ) = \eta (\gamma ^{-1})$, and call $\eta$ symmetric if $\check{\eta}=\eta$.
In what follows, we denote by $L^1(\mathcal{G},\mu ^1 )_{+,1}$ the set of all non-negative unit vectors in $L^1(\mathcal{G},\mu ^1 )$.
The next lemma (with $D=\emptyset$) will complete the proof.

\begin{lem}\label{lem:weakstar}
Let $(\mathcal{G},\mu )$ be a discrete p.m.p.\ groupoid and let $\bm{m}$ be a diffuse, conjugation-invariant mean on $(\mathcal{G},\mu )$ which is symmetric and balanced.
Let $D\subset \mathcal{G}$ be a Borel subset with $\bm{m}(D)=0$.
Then $\bm{m}$ is the weak${}^*$-limit of some net $(\xi _i)$ of symmetric vectors in $L^1(\mathcal{G},\mu ^1 )_{+, 1}$ which vanish on $D$ and satisfy conditions (i)--(iv) of Definition \ref{def:groupoid}.
\end{lem}

\begin{proof}
Since $\bm{m}$ is symmetric, by replacing $D$ by $D\cup D^{-1}$, we may assume that $D$ is symmetric as well.
For each $\eta \in L^1(\mathcal{G} ,\mu ^1 )$, define $r_{\eta}, s_{\eta} \in L^1(\mathcal{G}^0,\mu )$ by $r_{\eta} (x) \defeq \sum _{\gamma \in \calg^x}\eta (\gamma )$ and $s_{\eta} (x) \defeq \sum _{\gamma \in \calg_x}\eta (\gamma )$.
If $\eta$ is a vector in $L^1(\mathcal{G},\mu ^1 )_{+, 1}$, then $r_{\eta}$ and $s_{\eta}$ are non-negative unit vectors in $L^1(\mathcal{G}^0,\mu )$, and if $\eta$ is symmetric, then $r_{\eta}=s_{\eta}$.

\begin{claim}\label{claim:transfin}
Let $\eta \in L^1(\mathcal{G},\mu ^1 )_{+, 1}$ be a symmetric vector which vanishes on $D$.
Then there exists a symmetric vector $\xi \in L ^1 (\mathcal{G},\mu ^1 )_{+, 1}$ which vanishes on $D$ and satisfies $s_{\xi} = 1_{\mathcal{G}^0}$ and $\| \eta - \xi \| _1 = \| s_{\eta}- 1_{\mathcal{G}^0} \| _1$.
\end{claim}

\begin{proof}
Let $\eta _0 \defeq \eta$.
We proceed by transfinite induction on countable ordinals $\alpha$ to define a symmetric vector $\eta _{\alpha}\in L^1 (\mathcal{G},\mu ^1 )_{+, 1}$ vanishing on $D$ and satisfying, for all $\beta  < \alpha$:
\begin{enumerate}
\item[(i)] $\| \eta _{\beta} - \eta _{\alpha} \| _1 = \| s_{\eta _{\beta }} -s_{\eta _{\alpha }}  \| _1$.
\item[(ii)] For almost every $x\in \mathcal{G}^0$, if $s_{\eta _{\beta}}(x) \leq 1$, then $s_{\eta _{\beta}}(x)\leq s_{\eta _{\alpha}}(x)\leq 1$.
\item[(iii)] For almost every $x\in \mathcal{G}^0$, if $s_{\eta _{\beta}}(x) \geq 1$, then $s_{\eta _{\beta}}(x)\geq s_{\eta _{\alpha}}(x)\geq 1$.
\item[(iv)] If $\| s_{\eta _{\beta}}- 1_{\mathcal{G}^0}\| _1 >0$, then $\| s_{\eta _{\alpha}} - 1_{\mathcal{G}^0}\| _1 < \| s_{\eta _{\beta}}- 1_{\mathcal{G}^0} \| _1$, and if $s_{\eta _{\beta}} = 1_{\mathcal{G}^0}$, then $\eta _{\alpha}=\eta _{\beta}$.
\end{enumerate}
Assume that $\{ \eta _{\beta } \} _{\beta <\alpha}$ has already been defined, and we show how to define $\eta _{\alpha}$.
If $\alpha$ is a limit ordinal, say $\alpha = \sup _{n\in \N}\beta _n$, where $\beta _1<\beta _2<\cdots$, then, by the induction hypothesis (namely, properties (ii) and (iii)), the sequence $(s_{\eta _{\beta _n}})_{n\in \N}$ is Cauchy in $L^1(\mathcal{G}^0,\mu )$.
Hence property (i) implies that the sequence $(\eta _{\beta _n} ) _{n\in \N}$ is Cauchy in $L^1(\mathcal{G},\mu ^1 )$, so we define $\eta _{\alpha}$ to be its limit in $L^1(\mathcal{G},\mu ^1)$.
If $\alpha$ is a successor, then we define $\eta _{\alpha}$ from $\eta _{\alpha -1}$ as follows:
If $s_{\eta _{\alpha - 1}} = 1_{\mathcal{G}^0}$, then we put $\eta _{\alpha} = \eta _{\alpha -1}$.
Otherwise, i.e., if $s_{\eta _{\alpha - 1}} \neq 1_{\mathcal{G}^0}$, then for some $\ve >0$, both the sets
\[A _0 \defeq  \{ \, x\in {\mathcal{G}^0} \mid s_{\eta _{\alpha -1}}(x)< 1-\ve \, \} \quad \text{and} \quad A_1 \defeq  \{ \, x\in {\mathcal{G}^0} \mid s_{\eta _{\alpha -1}}(x)>1+\ve \, \}\]
have positive measure.
For each $i\in \{ 0,1 \}$, we have $\bm{m}(\mathcal{G}_{A_i}\setminus D) =\bm{m}(\calg_{A_i})=\mu (A_i)>0$, and hence $\mu^1(\calg_{A_i}\setminus D)>0$. 
We may find symmetric Borel subsets $C_i\subset \mathcal{G}_{A_i}\setminus D$ with $\mu ^1 (C_0)=\mu ^1  (C_1)>0$ and $|C_i \cap \calg_x| \leq 1$ for all $x\in {\mathcal{G}^0}$ (indeed, by restricting $\calg$ to a $\calg$-invariant Borel subset of $\calg^0$, we may assume without loss of generality that $(\calg^0, \mu)$ either is atomless or consists of atoms with the same measure, and in both cases, we can find such $C_0$ and $C_1$).
Then the function $\eta _{\alpha}\defeq  \eta _{\alpha -1}+\ve (1_{C_0} - 1_{C_1})$ has the required properties, and the induction is complete.

By property (iv), there is some countable ordinal $\alpha _0$ such that $s_{\eta _{\alpha _0}} = 1_{\mathcal{G}^0}$, so letting $\xi \defeq  \eta _{\alpha _0}$ works.\end{proof}

Return to the proof of Lemma \ref{lem:weakstar}.
By the Hahn-Banach theorem, the set $L^1(\mathcal{G},\mu ^1 )_{+,1}$ is weak${}^*$-dense in the set of all means on $(\mathcal{G},\mu )$.
Since $\bm{m}$ is symmetric and $\bm{m}(D)=0$, $\bm{m}$ belongs to the weak${}^*$-closure of the set of all symmetric vectors in $L^1(\mathcal{G},\mu ^1 )_{+,1}$ that vanish on $D$.
Let $(\eta _i )$ be a net weak${}^*$-converging to $\bm{m}$ and consisting of symmetric vectors in $L^1(\mathcal{G},\mu ^1 )_{+,1}$ that vanish on $D$.
Then $s_{\eta _i}$ converges weakly to $1_{\mathcal{G}^0}$ in $L^1({\mathcal{G}^0},\mu )$, and $\eta _i ^{\phi} - \eta _i$ converges weakly to $0$ in $L^1({\mathcal{G}^0},\mu )$ for every $\phi \in [\mathcal{G}]$.
Thus, by the Hahn-Banach theorem, after taking convex sums, we may assume without loss of generality that $\| s_{\eta _i}- 1_{\mathcal{G}^0} \|_1 \rightarrow 0$ and $\| \eta _i ^{\phi } - \eta _i \| _1 \rightarrow 0$ for every $\phi \in [\mathcal{G}]$.
Applying Claim \ref{claim:transfin} to each $\eta _i$, we obtain the required net $(\xi _i )$.
\end{proof}


\subsection{Permanence of inner amenability}\label{subsec-perma}

We discuss permanence of inner amenability under the following constructions: inflations, restrictions, finite-index inclusions, measure-preserving extensions, ergodic decompositions, and inverse limits.

\subsubsection{Inflations and restrictions}

\begin{prop}\label{prop:inflate}
Let $(\mathcal{G},\mu )$ be an ergodic discrete p.m.p.\ groupoid and let $A\subset \calg^0$ be a Borel subset with positive measure.
Then $(\mathcal{G},\mu )$ is inner amenable if and only if $(\mathcal{G}_A,\mu _A )$ is inner amenable, where $\mu_A$ is the normalized restriction of $\mu$ to $A$.
\end{prop}

\begin{proof}
As seen in Remark \ref{rem:locsec}, if $(\mathcal{G},\mu )$ is inner amenable, then so is $(\mathcal{G}_A,\mu _A )$.
Conversely, assume $(\mathcal{G}_A, \mu _A)$ is inner amenable and let $\bm{m}_A$ be a mean on $(\mathcal{G}_A,\mu _A)$ as in condition (5) of Theorem \ref{thm:equiv}.
After shrinking $A$, we may assume that $\mu (A) = 1/n$ for some $n\in \N$.
Since $(\mathcal{G},\mu )$ is ergodic, we can find some $\phi \in [\mathcal{G}]$ with $\{ \phi^i A \}_{i=0}^{n-1}$ partitioning $\mathcal{G}^0$ and $(\phi ^n)^o(x) =x$ for every $x\in  \mathcal{G}^0$, where $\phi^i\in [\calg]$ is the $i$-th iterate of $\phi$.
Let $A_i \defeq  \phi ^i A$ for $i\in \{ 0, 1,\ldots, n-1\}$.
Define a mean $\bm{m}$ on $(\mathcal{G},\mu )$ by
\[\bm{m}(D) \defeq \frac{1}{n}\sum _{i=0}^{n-1} \bm{m}_A((\mathcal{G}_{A_i}\cap D)^{\phi ^i}).\]
Then it is clear that $\bm{m}$ is balanced and diffuse.
Fix $\psi \in [\mathcal{G}]$ and a Borel subset $D\subset \mathcal{G}$ toward verifying conjugation-invariance of $\bm{m}$.
Since $\bm{m}$ is balanced, we have
\begin{align*}
\bm{m}(D) &= \sum _{i, j=0}^{n-1} \bm{m} (\mathcal{G}_{A_i\cap \psi A_j} \cap D)\quad \text{and} \\ 
\bm{m} ( D^{\psi}) &= \sum _{i, j=0}^{n-1}\bm{m} (\mathcal{G}_{\psi^{-1}A_i\cap A_j}\cap D^{\psi}) = \sum _{i,j=0}^{n-1} \bm{m} ((\mathcal{G}_{A_i\cap \psi A_j} \cap D)^{\psi}) ,
\end{align*}
so it suffices to show that if $D\subset \mathcal{G}_{A_i\cap \psi A_j}$, then $\bm{m}(D)= \bm{m}(D^{\psi })$.
Let $\chi$ be the local section of $\calg$ given as the restriction of $\phi^{-i} \psi \phi^j$ with domain $\phi^{-j}(\psi^{-1}A_i\cap A_j )\subset A$.
Then the range of $\chi$ is $\phi^{-i}(A_i \cap \psi A_j)\subset A$, so $\chi$ is a local section of $\mathcal{G}_A$.
Since $D\subset \mathcal{G}_{A_i\cap \psi A_j}$, we have $D^{\phi ^i}\subset \mathcal{G}_{r(\chi)}$ and
\begin{align*}
\bm{m} (D) &= \bm{m}_A (D^{\phi ^i })/n = \bm{m}_A (D^{\phi ^i \chi })/n = \bm{m}_A(D^{\psi  \phi ^j})/n = \bm{m}(D^{\psi} ).
\end{align*}
Thus $\bm{m}$ is conjugation-invariant, and by Theorem \ref{thm:equiv}, $(\mathcal{G},\mu )$ is inner amenable.
\end{proof}


\subsubsection{Finite-index inclusions}

Let $(\calg, \mu)$ be a discrete p.m.p.\ groupoid and let $\calh$ be a Borel subgroupoid of $\calg$.
For each $x\in \calg^0$, we have the equivalence relation on $\calg^x$ such that two elements $\gamma, \delta \in \calg^x$ are equivalent if and only if $\gamma^{-1} \delta \in \calh$.
The function assigning to each $x\in \calg^0$ the number of equivalence classes in $\calg^0$ is Borel and $\calg$-invariant, and hence constant on a conull set if $(\calg, \mu)$ is ergodic.
If $(\calg, \mu)$ is ergodic, this constant value is called the \textit{index} of $\calh$ in $\calg$.
This definition extends the index of a subrelation of a discrete p.m.p.\ equivalence relation given in \cite[Section 1]{fsz}.

\begin{prop}\label{prop-finite-index}
Let $(\calg, \mu)$ be an ergodic discrete p.m.p.\ groupoid and let $\calh$ be a finite-index Borel subgroupoid of $\calg$.
If $(\calh, \mu)$ is inner amenable, then $(\calg, \mu)$ is inner amenable.
\end{prop}

\begin{proof}
Assume first that $(\calh, \mu)$ is ergodic.
By assumption, we have a balanced, diffuse, conjugation-invariant mean $\bm{m}_0$ on $(\calh, \mu)$.
By setting $\bm{m}_0(\calg \setminus \calh)=0$, we regard $\bm{m}_0$ as a mean on $(\calg, \mu)$.
Then $\bm{m}_0$ is a balanced mean on $(\calg, \mu)$, and is conjugation-invariant under $\calh$, i.e., we have $\bm{m}_0(D^\phi)=\bm{m}_0(D)$ for all Borel subsets $D\subset \calg$ and all $\phi \in [\calh]$.

For $\gamma \in \calg$, we set $\gamma \calh \defeq \{ \, \gamma \delta \in \calg \mid \delta \in \calh^{s(\gamma )}\, \}$.
Let $N$ be the index of $\calh$ in $\calg$.
Since $(\calh, \mu)$ is ergodic, we may choose $\psi_1,\ldots, \psi_N\in [\calg]$, as in \cite[Lemmas 1.1 and 1.3]{fsz}, such that for all $x\in \calg^0$, the sets $(\psi_i)^x\calh$ with $i=1, \ldots, N$ partition $\calg^x$.
We define a mean $\bm{m}$ on $(\calg, \mu)$ by $\bm{m}(D)\defeq N^{-1}\sum_{i=1}^N\bm{m}_0(D^{\psi_i})$.
Pick a Borel subset $D\subset \calg$ and $\phi \in [\calg]$.
Let $A_{ij}$ be the Borel subset of all $x\in \calg^0$ such that $\phi_x (\psi_i)^x\calh = (\psi_j)^y\calh$ with $y\defeq r(\phi_x)$.
For almost every $x\in \calg^0$, for every $i$, there exists a unique $j$ such that $\phi_x (\psi_i)^x\calh = (\psi_j)^y\calh$, and moreover this assignment $i\mapsto j$ is bijective.
Therefore $\calg^0=\bigsqcup_{j=1}^N A_{ij}=\bigsqcup_{i=1}^NA_{ij}$ for all $i$ and $j$.
We define
\[\theta_{ij}\defeq \{ \, [\phi_x(\psi_i)^x]^{-1}(\psi_j)^y\mid x\in A_{ij},\ y=r(\phi_x)\, \},\]
which is a local section of $\calh$ with $s(\theta_{ij})=\psi_j^{-1}\phi A_{ij}$ and $r(\theta_{ij})=\psi_i^{-1}A_{ij}$.
Then we have
\begin{align*}
\bm{m}(D^\phi)&=\frac{1}{N}\sum_{i=1}^N\bm{m}_0(D^{\phi \psi_i})=\frac{1}{N}\sum_{i, j=1}^N\bm{m}_0\bigl(D^{\phi \psi_i}\cap \calh_{\psi_i^{-1}A_{ij}}\bigr)\\
&=\frac{1}{N}\sum_{i, j=1}^N\bm{m}_0\Bigl(\bigl(D^{\phi \psi_i}\cap \calh_{\psi_i^{-1}A_{ij}}\bigr)^{\theta_{ij}}\Bigr)=\frac{1}{N}\sum_{i, j=1}^N\bm{m}_0\bigl(D^{\psi_j}\cap \calh_{\psi_j^{-1}\phi A_{ij}}\bigr)\\
&=\frac{1}{N}\sum_{j=1}^N \bm{m}_0(D^{\psi_j})=\bm{m}(D),
\end{align*}
where the second and fifth equations hold because $\bm{m}_0$ is balanced, and the third equation holds because $\bm{m}_0$ is conjugation-invariant under $\calh$.
The mean $\bm{m}$ is therefore conjugation-invariant under $\calg$.
Since $\bm{m}$ is diffuse, $(\calg, \mu)$ is inner amenable by Theorem \ref{thm:equiv}.

In general, since $(\calg, \mu)$ is ergodic and $\calh$ has finite index in $\calg$, there is some non-null $\calh$-invariant Borel subset $A\subset \calg ^0$ such that $(\calh _A, \mu_A)$ is ergodic, where $\mu_A$ is the normalized restriction of $\mu$ to $A$ (this follows from e.g., \cite[Lemma 2.1]{ha}). 
Since $(\calh, \mu)$ is inner amenable, so is $(\calh_A, \mu_A)$ by Remark \ref{rem:locsec}.
Since $\calh_A$ has finite index in $\calg_A$ (in fact, the index of $\calh_A$ in $\calg_A$ is not larger than the index of $\calh$ in $\calg$), it follows from the ergodic case proved above that $(\calg_A, \mu_A)$ is inner amenable, and hence $(\calg, \mu)$ is inner amenable by Proposition \ref{prop:inflate}.
\end{proof}

The converse of Proposition \ref{prop-finite-index} also holds (see Corollary \ref{cor-finite-index-cpt-ext}).


\subsubsection{Measure-preserving extensions}

Let $(\mathcal{G} ,\mu )$ be a discrete p.m.p.\ groupoid.
Let $(Z,\zeta )$ be a standard probability space and let $\alpha \colon \mathcal{G} \rightarrow \mathrm{Aut}(Z ,\zeta )$ be a cocycle.
The associated \textit{extension groupoid} $(\mathcal{G} ,\mu )\ltimes _{\alpha}(Z ,\zeta )= (\tilde{\mathcal{G}},\tilde{\mu} )$ is the discrete p.m.p.\ groupoid defined as follows:
The set of groupoid elements is $\tilde{\mathcal{G}} \defeq  \mathcal{G}\times Z$, with unit space $\tilde{\mathcal{G}}^0 \defeq  \mathcal{G}^0\times Z$ and measure $\tilde{\mu} \defeq  \mu \times \zeta$ on $\tilde{\calg}^0$.
The source and range maps are defined by $\tilde{s}(\gamma , z ) = (s(\gamma ), z)$ and $\tilde{r}(\gamma ,z )= (r(\gamma ), \alpha (\gamma )z )$, respectively, with groupoid operations defined by $(\gamma _1, \alpha (\gamma _0)z)(\gamma _0, z) = (\gamma _1\gamma _0 , z )$ and $(\gamma ,z )^{-1} = (\gamma ^{-1}, \alpha (\gamma )z )$.

\begin{prop}\label{prop:Ginn}
Let $(\mathcal{G},\mu )$ be a discrete p.m.p.\ groupoid, let $(Z, \zeta)$ be a standard probability space, and let $\alpha \colon \mathcal{G}\rightarrow \mathrm{Aut}(Z,\zeta )$ be a cocycle.
Suppose that the extension groupoid $(\mathcal{G} ,\mu )\ltimes _{\alpha}(Z ,\zeta )$ is inner amenable.
Then $(\mathcal{G},\mu )$ is inner amenable.

In particular, if a countable group $G$ admits a p.m.p.\ action $G\curvearrowright (Z,\zeta )$ such that the associated translation groupoid $G\ltimes (Z,\zeta )$ is inner amenable, then $G$ is inner amenable.
\end{prop}

\begin{proof}
Suppose that the groupoid $(\tilde{\mathcal{G}},\tilde{\mu} )\defeq  (\mathcal{G} ,\mu )\ltimes _{\alpha}(Z ,\zeta )$ is inner amenable, and let $\tilde{\bm{m}}$ be a mean on $(\tilde{\mathcal{G}},\tilde{\mu})$ as in condition (5) of Theorem \ref{thm:equiv}.
Then the mean $\bm{m}$ on $(\mathcal{G},\mu )$ defined by $\bm{m}(D)\defeq  \tilde{\bm{m}}(D\times Z )$ witnesses that $(\mathcal{G},\mu )$ is inner amenable.
\end{proof}

While the converse of Proposition \ref{prop:Ginn} does not hold in general (e.g., Corollary \ref{cor-ber}), it does hold for compact extensions (Corollary \ref{cor:cpctext}), and more generally for distal extensions (Corollary \ref{cor:distal}).


\subsubsection{Ergodic decompositions}

We refer to \cite{hahn} for the ergodic decomposition of discrete p.m.p.\ groupoids.

\begin{prop}\label{prop:ergdec}
Let $(\mathcal{G},\mu )$ be a discrete p.m.p.\ groupoid with ergodic decomposition map $\pi \colon (\mathcal{G}^0,\mu )\rightarrow  (Z,\zeta )$ and disintegration $(\mathcal{G},\mu ) = \int _Z (\mathcal{G}_z,\mu _z )\, d\zeta (z)$.
Then $(\mathcal{G},\mu )$ is inner amenable if and only if $(\mathcal{G}_z,\mu _z)$ is inner amenable for $\zeta$-almost every $z\in Z$.
\end{prop}

\begin{proof}
Assume first that $(\mathcal{G}_z,\mu _z)$ admits an inner amenability sequence for $\zeta$-almost every $z\in Z$.
The groupoid $(\mathcal{G}_z, \mu _z)$, being inner amenable, is aperiodic for $\zeta$-almost every $z\in Z$, and hence $(\mathcal{G}, \mu )$ is aperiodic as well.
Let $\mu _{\N}$ denote the counting measure on $\N$, the set of non-negative integers.
Since the source map $s\colon \mathcal{G}\rightarrow \mathcal{G}^0$ is a countably infinite-to-one Borel map, by the Lusin-Novikov Uniformization Theorem (\cite[Theorem 18.10]{kec-set}), we may find an isomorphism of measure spaces
\[
\varphi \colon (\N\times \mathcal{G}^0, \, \mu _{\N}\times \mu ) \rightarrow (\mathcal{G},\mu ^1)
 \]
satisfying $\varphi (0,x)=x\in \mathcal{G}^0$ and $s(\varphi (i,x)) = x$ for every $i\in \N$ and $\mu$-almost every $x\in \mathcal{G}^0$.
We may therefore assume without loss of generality that there is some standard probability space $(X,\mu _X )$ such that, as measure spaces we have $(\mathcal{G},\mu ^1 )=(\N\times X, \mu _{\N}\times \mu _X )$ and $(\mathcal{G}^0 , \mu ) =  ( \{ 0 \} \times X, \delta _0 \times \mu _X )$, with the source map $s\colon \mathcal{G}\rightarrow \mathcal{G}^0$ being given by $s(i,x)=(0,x)$ for $\mu^1$-almost every $(i,x) \in \mathcal{G}$.

Let $Z_0$ consist of all $z\in Z$ for which the measure $\mu _z$ on $\mathcal{G}^0_z$ is atomless, and for each integer $n\geq 1$, let $Z_n$ consist of all $z\in Z$ for which the measure $\mu _z$ is uniformly distributed on $n$ points.
Then $Z_0,Z_1,\dots$ partition $Z$, and $\pi ^{-1}(Z_0), \pi ^{-1}(Z_1),\dots $ partition $\mathcal{G}^0$ into $\mathcal{G}$-invariant sets, so it is enough to show that, for each $n$ with $\mu (\pi ^{-1}(Z_n))>0$, the groupoid $(\mathcal{G}_{\pi ^{-1}(Z_n) }, \mu _{\pi ^{-1}(Z_n)} )$ admits an inner amenability sequence.
We may therefore assume without loss of generality that $Z=Z_n$ for some $n$, and hence (arguing as in the proof of \cite[Theorem 3.18]{gla}) we may also assume that $(X,\mu _X )=(Y\times Z ,\nu \times \zeta )$ for some standard probability space $(Y,\nu )$, with $\pi \colon \mathcal{G}^0 \rightarrow Z$ being the projection map to the $Z$-coordinate, $\pi (0,(y,z)) = z$.
Then for each $z\in Z$, the measures $\mu _z$ on $\mathcal{G}_z^0= \{ 0 \} \times Y \times \{ z\} $ and $\mu ^1_z$ on $\mathcal{G}_z = \N \times Y \times \{ z \}$ are respectively given by $\mu _z = \delta _0 \times \nu \times \delta _z$ and $\mu ^1_z = \mu _{\N} \times \nu \times \delta _z$.

Fix $\ve >0$, along with a finite subset $\Phi \subset [\mathcal{G}]$, finite Borel partitions $\mathcal{B}$ and $\mathcal{C}$ of $Y$ and $Z$, respectively, and a Borel subset $D\subset \mathcal{G}$ with $\mu ^1 (D)<\infty$.
It is enough to find a symmetric, non-negative unit vector $\xi \in L^1(\mathcal{G},\mu ^1 )=L^1(\N \times Y\times Z , \mu _{\N} \times \nu \times \zeta )$ satisfying
\begin{enumerate}
\item[(1)] $| \| 1_{\calg_{\{ 0 \} \times B\times C}}\xi \| _1 - (\nu \times \zeta ) (B\times C) | < \ve$ for all $B\in \mathcal{B}$ and all $C\in \mathcal{C}$,
\item[(2)] $\| \xi ^{\phi } - \xi \| _1 < \ve$ for all $\phi \in \Phi$,
\item[(3)] $\| 1_D \xi \| _1 < \ve$, and
\item[(4)] $\sum _{i \in \N} \xi (i,(y,z)) =1$ for $(\nu\times \zeta )$-almost every $(y,z)\in Y\times Z$.
\end{enumerate}
For $z\in Z$ and $\xi \in L^1( \N \times Y , \mu _{\N } \times \nu )$, let $\xi ^{(z)} \in L^1(\mathcal{G}_z, \mu ^1 _z ) = L^1(\N \times Y \times \{ z \} , \mu _{\N}\times \nu \times \delta _z )$ be the vector given by $\xi ^{(z)} (i,y,z)\defeq \xi (i,y)$. 
For almost every $z\in Z$, the groupoid $(\mathcal{G}_z, \mu _z )$ is inner amenable and $\mu ^1_z (D)<\infty$, so we may find some $\xi \in L^1 (\N \times Y, \mu _{\N}\times \nu )$ such that $\xi ^{(z)}\in L^1(\mathcal{G}_z, \mu ^1 _z)$ is a symmetric, non-negative unit vector satisfying
\begin{enumerate}
\item[(1.$z$)] $| \| 1_{\calg_{\{ 0 \} \times B\times C}}\xi ^{(z)} \| _{L^1(\mu ^1_z )} -  \nu (B)1_C(z) | < \ve$ for all $B\in \mathcal{B}$ and all $C\in \mathcal{C}$,
\item[(2.$z$)] $\| (\xi ^{(z)} ) ^{\phi } - \xi ^{(z)} \| _{L^1(\mu ^1_z)} < \ve$ for all $\phi \in \Phi$,
\item[(3.$z$)] $\| 1_D \xi ^{(z)} \| _{L^1(\mu ^1 _z)} < \ve$, and
\item[(4.$z$)] $\sum _{i \in \N} \xi (i, y) =1$ for $\nu$-almost every $y\in Y$.
\end{enumerate}
Let $\Omega$ be the set of all such pairs $(z,\xi )$, i.e., all pairs $(z,\xi )\in Z\times L^1(\N \times Y ,\mu _{\N}\times \nu )$ such that $\xi ^{(z)}\in L^1(\mathcal{G}_z, \mu ^1 _z)$ is a symmetric, non-negative unit vector satisfying conditions (1.$z$)--(4.$z$).
Then $\Omega$ is a Borel subset of $Z\times L^1 (\N \times Y, \mu _{\N}\times \nu )$, where the reason (2.$z$) defines a Borel property is because all the groupoid operations are by assumption Borel and each $\phi \in \Phi$ is Borel.
By applying the Jankov-von Neumann Uniformization Theorem (\cite[Theorem 18.1]{kec-set}) and Lusin's theorem that analytic sets are universally measurable (\cite[Theorem 21.10]{kec-set}), after discarding a $\zeta$-null set from $Z$, we may find a Borel map $Z\rightarrow L^1(\N\times Y, \mu _{\N}\times \nu )$, $z\mapsto \xi _z$, such that $(z,\xi _z ) \in \Omega$ for almost every $z\in Z$. Define $\xi \in L^1(\calg,\mu ^1 )$ by $\xi (i,y,z)\defeq \xi _z (i,y)$.
Then $\xi$ is a symmetric, non-negative unit vector in $L^1(\calg,\mu ^1 )$ satisfying conditions (1)--(4).

Conversely, assume that $(\xi _n)_{n\in \N}$ is an inner amenability sequence for $(\mathcal{G},\mu )$.
By properties (i) and (iv) of Definition \ref{def:groupoid}, for every Borel subset $B\subset \mathcal{G}^0$, we have
\begin{align*}
&\int _Z \biggl| \mu _z (B) - \int _{(\mathcal{G}_z)_B} \xi _n \, d\mu ^1_z\biggr| \, d\zeta(z) = \int _Z \biggl| \int _B \biggl( 1 - \sum _{\gamma \in (\calg_B)_x} \xi _n (\gamma )\biggr) \, d\mu _z(x) \biggr| \, d\zeta(z) \\
&=\int _Z \int _B\biggl( 1 - \sum _{\gamma \in (\calg_B)_x} \xi _n (\gamma )\biggr)\, d\mu _z(x) \, d\zeta(z) = \mu (B) - \int _{\mathcal{G}_B}\xi _n \, d\mu ^1 \rightarrow 0.
\end{align*}
Likewise, for every $\phi \in [\mathcal{G}]$, we have $\int _{Z} \| \xi _n ^{\phi} - \xi _n \| _{L^1(\mu ^1 _z)} \, d\zeta(z) \rightarrow 0$, and for every bounded Borel subset $D\subset \mathcal{G}$, we have $\int _Z \| 1_D \xi _n \| _{L^1(\mu ^1 _z)}\, d\zeta(z) \rightarrow 0$.
Therefore, by separability of $L^1(\mathcal{G}^0,\mu )$ and of $[\mathcal{G}]$, we can find a single subsequence $(\xi _{n_i})$ such that for $\zeta$-almost every $z\in Z$, $(\xi _{n_i})$ is an inner amenability sequence for $(\mathcal{G}_z,\mu _z)$.
\end{proof}


\subsubsection{Inverse limits}

Let $(\mathcal{G}_1, \mu _1 )$ be a discrete p.m.p.\ groupoid.
A \textit{locally bijective extension} of $(\mathcal{G}_1 ,\mu _1 )$ is a measure-preserving groupoid homomorphism $\varphi \colon (\mathcal{G}, \mu ) \rightarrow (\mathcal{G}_1,\mu _1 )$, from a discrete p.m.p.\ groupoid $(\mathcal{G}, \mu )$ to $(\mathcal{G}_1,\mu _1 )$, such that for almost every $x\in \mathcal{G}^0$, its restriction $\varphi \colon \calg_x \rightarrow (\calg_1)_{\varphi (x)}$ is bijective.
Clearly, compositions of locally bijective extensions are locally bijective.

Suppose that $I$ is a countable directed set and we have a directed family $(\varphi _{i,j} \colon (\mathcal{G}_j,\mu _j)\rightarrow (\mathcal{G}_i,\mu _i ) )_{i,j\in I , i<j}$ of locally bijective extensions of groupoids, that is, $\varphi_{i, j}$ is a measure-preserving groupoid homomorphism such that $\varphi _{i,j}\circ \varphi _{j,k}=\varphi _{i,k}$ whenever $i<j<k$.
Then the \textit{inverse limit} of this family is the discrete p.m.p.\ groupoid $(\mathcal{G},\mu )$ defined by
\begin{align*}
\mathcal{G} &\defeq \left\{  \left. (\gamma _i )_{i\in I} \in \prod _{i\in I}\mathcal{G}_i \, \right| \, \varphi _{i,j} (\gamma _j ) = \gamma _i \text{ for  all }i<j \, \right\}  
\end{align*}
and
\begin{align*}
\mathcal{G}^0 &\defeq \left\{ \left. (x_i )_{i\in I} \in \prod _{i\in I}\mathcal{G}_i^0 \, \right| \, \varphi _{i,j}(x_j ) = x_i \text{ for all }i <j \, \right\} ,
\end{align*}
with $(\mathcal{G}^0, \mu )$ the inverse limit of the measure spaces $(\mathcal{G}_i^0,\mu _i^0 )$, and with source and range maps defined by $s((\gamma _i)_{i\in I}) = (s(\gamma _i))_{i\in I}$ and $r((\gamma _i)_{i\in I}) = (r(\gamma _i))_{i\in I}$, respectively.
For each $i\in I$, the projection map $\varphi _i \colon (\mathcal{G},\mu )\rightarrow (\mathcal{G}_i,\mu _i )$ is a locally bijective extension of groupoids, and if $i<j$, then $\varphi _{i,j}\circ \varphi _j = \varphi _i$.

\begin{prop}\label{prop:invlim}
Let $(\varphi _{i,j} \colon (\mathcal{G}_j,\mu _j)\rightarrow (\mathcal{G}_i,\mu _i ) )_{i,j\in I , i<j}$ be a countable directed family of locally bijective extensions of groupoids, and let $(\mathcal{G},\mu )$ be its inverse limit.
If each of the groupoids $(\mathcal{G}_i, \mu _i )$ is inner amenable, then $(\mathcal{G},\mu )$ is inner amenable.
\end{prop}

\begin{proof}
Since $(\mathcal{G}_i, \mu _i )$ is inner amenable, we may find a positive linear map $P_i\colon L^{\infty}(\mathcal{G}_i,\mu ^1 _i ) \rightarrow L^{\infty}(\mathcal{G}_i^0,\mu _i )$ as in condition (6) of Theorem \ref{thm:equiv}, so that the mean $\bm{n}_i$ on $(\calg_i, \mu_i)$ defined by $\bm{n}_i(D) \defeq \int _{\mathcal{G}_i^0} P_i(1_D) \, d\mu _i$ is a balanced, diffuse, conjugation-invariant mean on $(\mathcal{G}_i,\mu _i)$.
Let $\mu = \int _{\mathcal{G}_i^0} \mu _i^z \, d\mu _i(z)$ be the disintegration of $\mu$ via $\varphi _i$.
For $\mu_i^1$-almost every $\delta \in \mathcal{G}_i$, there is the bijection $\gamma _{\varphi _i, \delta} \colon \varphi _i^{-1}(s(\delta ))\rightarrow \varphi _i^{-1}(\delta )$ that sends each $x\in \varphi _i^{-1}(s(\delta ))$ to the unique element $\gamma = \gamma _{\varphi _i ,\delta}(x)$ in $\calg_x$ with $\varphi _i (\gamma ) = \delta$.
We then obtain the conditional expectation $E_i\colon L^{\infty}(\mathcal{G},\mu ^1 ) \rightarrow L^{\infty}(\mathcal{G}_i,\mu ^1 _i )$ given by
\[
E_i(F)(\delta )= \int _{\varphi _i^{-1}(s(\delta ))}F(\gamma _{\varphi _i,\delta}(x))\, d\mu _i ^{s(\delta )}(x)
\]
for $F\in L^\infty(\calg, \mu^1)$ and $\delta \in \calg_i$.
We define $\bm{m}_i$ as the mean on $(\mathcal{G} ,\mu )$ given by $\bm{m}_i(D) \defeq \int _{\mathcal{G}_i^0} P_i(E_i(1_D))\, d\mu _i$, which projects via $\varphi _i$ to $\bm{n}_i$.
Any weak${}^*$-cluster point of the net $(\bm{m}_i )_{i\in I}$ is then a balanced, diffuse, conjugation-invariant mean on $(\mathcal{G},\mu )$, and hence $(\mathcal{G},\mu )$ is inner amenable by Theorem \ref{thm:equiv}.
\end{proof}


\subsection{Central sequences in the full group}\label{subsec-schmidt}

Let $\mathcal{R}$ be an ergodic discrete p.m.p.\ equivalence relation on a standard probability space $(X,\mu )$.
A sequence $(T_n)$ of elements of $[\calr]$ is called \textit{central} in $[\calr]$ if it asymptotically commutes with every element of $[\calr]$, i.e., for every $S\in [\mathcal{R}]$, we have $\mu (\{ \, x\in X \mid ST_nx\neq  T_nSx \, \} )\to 0$.
A central sequence $(T_n)$ in $[\calr]$ is called \textit{trivial} if $\mu (\{ \, x\in X \mid T_nx = x \, \} ) \to 1$.
A sequence $(A_n)$ of Borel subsets of $X$ is called \textit{asymptotically invariant} for $\calr$ if $\mu(TA_n\bigtriangleup A_n)\to 0$ for every $T\in [\calr]$.

\begin{rem}\label{rem-central}
This remark is analogous to Remark \ref{rem:suffice}. 
Let $(T_n)$ be a sequence in $[\calr]$.
If $(T_n)$ is central in $[\calr]$, then $\mu(T_nA\bigtriangleup A)\to 0$ for every Borel subset $A\subset X$.
Indeed for each Borel subset $A\subset X$, if we pick $S\in [\calr]$ such that $S$ is the identity on $A$ and $Sx\neq x$ for all $x\in X\setminus A$, then the set $T_n^{-1}A\bigtriangleup A$ is contained in the set $\{ \, x\in X \mid ST_nx\neq  T_nSx \, \}$.
Conversely, if $(T_n)$ satisfies $\mu(T_nA\bigtriangleup A)\to 0$ for every Borel subset $A\subset X$, then for the sequence $(T_n)$ to be central in $[\calr]$, it is sufficient that there is some countable subgroup $G$ of $[\calr]$ generating $\calr$, such that $\mu (\{ \, x\in X \mid ST_nx\neq  T_nSx \, \} )\to 0$ for all $S\in G$ (see \cite[Remark 3.3]{js} and see also \cite[Proposition 6.2]{kec} for a more general statement).
\end{rem}

\begin{lem}\label{lem-ai}
Let $(A_n)$ be an asymptotically invariant sequence for $\calr$ with $\mu(A_n)\to r$ for some number $r$.
Then for every Borel subset $B\subset X$, we have $\mu(A_n\cap B)\to r\mu(B)$.
\end{lem}

\begin{proof}
This is observed in the proof of \cite[Lemma 2.3]{js}.
Since it will frequently be applied in the sequel, we give its proof here for the reader's convenience.
It is enough to show that the convergence holds for some subsequence of $(A_n)$.
Passing to a subsequence, we may assume that $1_{A_n}$  converges to some $f\in L^\infty(X, \mu)$ in the weak${}^*$-topology.
Since $(A_n)$ is asymptotically invariant for $\calr$, $f$ is invariant under $\calr$.
By ergodicity of $\calr$, $f$ is constant, and since $\mu (A_n)\to r$, that constant must be $r$.
Thus for every Borel subset $B\subset X$, we have $\mu(A_n\cap B)=\int 1_{A_n}1_B\, d\mu \to \int r1_B \, d\mu =r\mu(B)$.
\end{proof}

\begin{prop}\label{prop:iagroupoid}
Let $\mathcal{R}$ be an ergodic discrete p.m.p.\ equivalence relation on a standard probability space $(X,\mu )$, and suppose that $[\mathcal{R}]$ admits a non-trivial central sequence.
Then $\mathcal{R}$ is inner amenable.
\end{prop}

\begin{proof}
Let $(T_n)_{n\in \N}$ be a non-trivial central sequence in $[\calr]$.
For each $n$, we set $A_n \defeq  \{ \, x\in X \mid T_nx \neq x \, \}$.
After passing to a subsequence, we may assume that $\mu (A_n )$ converges to some $r>0$.
Since $(T_n)_{n\in \N}$ is central, the sequence $(A_n)_{n\in \N}$ is asymptotically invariant for $\calr$.
By Lemma \ref{lem-ai}, $\mu (A_n \cap A)\to r\mu (A)$ for every Borel subset $A\subset X$.
Let $\omega$ be a non-principal ultrafilter on $\N$ and define a mean $\bm{m}$ on $(\mathcal{R}, \mu)$ by
\[
\bm{m}(D) \defeq \lim _{n\to \omega} \frac{\mu ( \{ \, x\in A_n \mid (T_nx, x)\in D \, \} )}{\mu (A_n)}
\]
for a Borel subset $D\subset \mathcal{R}$.
Given a Borel subset $A\subset X$, since $\mu (T_nA\bigtriangleup A)\to 0$, we have
\[
\bm{m}(\mathcal{R}_A) = \lim _{n\to\omega}\mu (A_n)^{-1} \mu (T_n^{-1}A\cap A \cap A_n )= r^{-1}\mu (A)r = \mu (A).
\]
This shows that $\bm{m}$ is balanced.
For $S\in [\mathcal{R}]$ and a Borel subset $D\subset \mathcal{R}$, we have
\begin{align*}
|\bm{m}(D^S) - \bm{m}(D)| 
&\leq \lim _{n\to \omega} \mu (A_n)^{-1} (\mu ( SA_n\bigtriangleup A_n ) + \mu (\{ \, x\in X \mid ST_nx\neq T_nSx \, \} ) )  = 0
\end{align*}
since $(T_n)_{n\in \N}$ is central in $[\mathcal{R}]$ and $(A_n)_{n\in \N}$ is asymptotically invariant with $\lim _n \mu (A_n)=r>0$.
Thus $\bm{m}$ is conjugation-invariant.
By definition, we have $\bm{m} ( \{ \, (x,x) \mid x\in X \, \} ) =0$ and hence $\bm{m}$ is diffuse by Lemma \ref{lem:diffuse}.
By Theorem \ref{thm:equiv}, $\calr$ is inner amenable.
\end{proof}

\begin{rem}\label{rem-iagroupoid}
In Lemma \ref{lem-c} (whose proof is based on Lemma \ref{lem-patch-c}), under the assumption that there exists a non-trivial central sequence in $[\calr]$, we construct a non-trivial central sequence $(T_n)$ in $[\calr]$ which further satisfies $T_nx\neq x$ for all $n$ and all $x\in X$.
Once such $(T_n)$ is obtained, the indicator function of the graph of $T_n$ forms an inner amenability sequence for $(\calr, \mu)$ (and this gives another proof of Proposition \ref{prop:iagroupoid}).
Indeed, condition (i) of Definition \ref{def:groupoid} follows from the condition $\mu(T_nA\bigtriangleup A)\to 0$ for all Borel subsets $A\subset X$, condition (ii) follows from the condition $\mu(\{ \, x\in X\mid ST_nx\neq T_nSx\, \})\to 0$ for all $S\in [\calr]$, condition (iii) follows from Lemma \ref{lem-s} together with the condition $T_nx\neq x$ for all $n$ and all $x\in X$, and condition (iv) holds since $T_n$ is an automorphism of $(X, \mu)$.
\end{rem}

Schmidt raises the following problem, which remains open.

\begin{question}[\ci{Problem 4.6}{sch-prob}] \label{question:Schmidt}
Does every countable inner amenable group $G$ admit a free ergodic p.m.p.\ action $G\c (X,\mu )$ such that the full group $[\mathcal{R}(G\c (X,\mu ))]$ has a non-trivial central sequence?
\end{question}

We say that an ergodic discrete p.m.p.\ equivalence relation is \textit{Schmidt} if its full group has a non-trivial central sequence.
We say that a free ergodic p.m.p.\ action of a countable group is \textit{Schmidt} if the associated orbit equivalence relation is Schmidt, and say that a countable group has the \textit{Schmidt property} if it admits a free ergodic p.m.p.\ action which is Schmidt.
Question \ref{question:Schmidt} turns out to have an affirmative answer when $G$ is linear (\cite[Theorem 15]{td}).
In general though, there is much more evidence for an affirmative answer to the following question:

\begin{question}\label{qu:oia}
Does every countable inner amenable group admit a free ergodic p.m.p.\ action whose orbit equivalence relation is inner amenable?
\end{question}

We call a countable group \textit{orbitally inner amenable} if it admits an action as in Question \ref{qu:oia}.
Observe that by Proposition \ref{prop:iagroupoid}, every group with the Schmidt property is orbitally inner amenable, and by Proposition \ref{prop:Ginn}, every orbitally inner amenable group is inner amenable.
In Corollary \ref{cor:cpctext}, we will show that every residually finite, inner amenable group is orbitally inner amenable.
See Propositions \ref{prop:resfinex} and \ref{prop:wreath} for other examples of orbitally inner amenable groups.


\subsection{Property Gamma}\label{subsec-gamma}

Let $M$ be a $\mathrm{II}_1$ factor with the faithful normal trace $\tau$.
Let $L^2(M)$ be the Hilbert space obtained by completing $M$ with respect to the norm $\Vert x\Vert_2=\tau(x^*x)^{1/2}$.
We say that $M$ has \textit{property Gamma} if there exists a sequence $(u_n)$ of unitaries of $M$ such that $\tau(u_n)=0$ and $\Vert [x, u_n]\Vert_2\to 0$ for every $x\in M$.
For a countable ICC group $G$, if the group factor $LG$ has property Gamma, then $G$ is inner amenable (\cite{effros}), but the converse is not true (\cite{vaes}).

Choda \cite[Theorem (ii)]{choda} shows that for a free ergodic p.m.p.\ action $G\c (X, \mu)$ of a countable group $G$, if the associated factor $G\ltimes L^\infty(X)$ has property Gamma and the action $G\c (X, \mu)$ is further strongly ergodic, then $G$ is inner amenable.
Under the same assumption, we prove the stronger assertion that the translation groupoid $G\ltimes (X, \mu)$ is inner amenable.
Recall that a p.m.p.\ action $G\c (X, \mu)$ is called \textit{strongly ergodic} if every asymptotically invariant sequence $(A_n)$ for the action (i.e.,  sequence of Borel subsets $A_n\subset X$ with $\mu(gA_n\bigtriangleup A_n)\to 0$ for all $g\in G$) satisfies $\mu(A_n)(1-\mu(A_n))\to 0$.
We note that strong ergodicity is an invariant under orbit equivalence (\cite[Proposition 2.1]{sch-asymp}).

\begin{prop}\label{prop-gamma-ia}
Let $\calr$ be an ergodic discrete p.m.p.\ equivalence relation on a standard probability space $(X, \mu)$.
Suppose that the factor $M$ associated to $\calr$ has property Gamma, and let $(u_n)$ be a sequence of unitaries of $M$ such that $\tau(u_n)=0$ and $\Vert [x, u_n]\Vert_2\to 0$ for every $x\in M$.
Choose a family $\{ \phi_k\}_{k\in \N}$ of local sections of $\calr$ such that
\[\calr =\bigsqcup_{k\in \N} \{ \, (\phi_k(x), x)\mid x\in \mathrm{dom}(\phi_k)\, \}\]
and $\phi_1=\mathrm{id}$ with $\mathrm{dom}(\phi_1)=X$, where $\mathrm{dom}(\phi_k)$ denotes the domain of $\phi_k$.
Expand $u_n$ into $u_n=\sum_ku_{\phi_k}f_k^n$ with $f_k^n\in L^\infty(X)$ supported on $\mathrm{dom}(\phi_k)$, where for a local section $\phi$ of $\calr$, we let $u_{\phi}$ denote the associated partial isometry of $M$.
Define $\xi_n\in L^1(\calr, \mu^1)$ by
\[\xi_n(\phi_k(x), x)\defeq |f_k^n(x)|^2\]
for $k\in \N$ and $x\in \mathrm{dom}(\phi_k)$.
Then
\begin{enumerate}
\item[(i)] each $\xi_n$ is a non-negative unit vector of $L^1(\calr, \mu^1)$, we have $\Vert \xi_n^\phi -\xi_n\Vert_1\to 0$ for every $\phi \in [\calr]$, and we have $\sum_{y\in [x]_\calr}\xi_n(y, x)=1=\sum_{y\in [x]_\calr}\xi_n(x, y)$ for $\mu$-almost every $x\in X$, where $[x]_\calr$ is the equivalence class of $x$ in $\calr$.
\item[(ii)] If $\calr$ is strongly ergodic, then $(\xi_n)$ is an inner amenability sequence for $\calr$.
\end{enumerate}
\end{prop}

\begin{proof}
Since $u_n$ is a unitary, we have $\sum_k\Vert f_k^n\Vert_2^2=1$ and hence $\xi_n$ is a non-negative unit vector of $L^1(\calr, \mu^1)$.
Pick $\phi \in [\calr]$, and we verify $\Vert \xi_n^\phi -\xi_n\Vert_1\to 0$.
For $k, l\in \N$, set
\[D_k^l\defeq \{ \, x\in X\mid \phi(x)\in \mathrm{dom}(\phi_k),\, x\in \mathrm{dom}(\phi_l)\ \text{and}\ \phi^{-1}\phi_k\phi(x)=\phi_l(x)\, \}.\]
Then $D_1^1=X$ and we have the Borel partitions, $\phi^{-1}(\mathrm{dom}(\phi_k))=\bigsqcup_l D_k^l$ and $\mathrm{dom}(\phi_l)=\bigsqcup_kD_k^l$.
We also have
\[u_\phi^*u_nu_\phi =\sum_ku_{\phi^{-1}\phi_k\phi}(\phi^{-1}\cdot f_k^n)=\sum_k\sum_l u_{\phi_l}1_{D_k^l}(\phi^{-1}\cdot f_k^n),\]
where we set $\phi^{-1}\cdot f_k^n=f_k^n\circ \phi$, and thus
\begin{align*}
\Vert u_\phi^*u_nu_\phi-u_n\Vert_2^2&=\sum_l \biggl\Vert \sum_k 1_{D_k^l}(\phi^{-1}\cdot f_k^n)-f_l^n\biggr\Vert_2^2=\sum_{k, l}\int_{D_k^l}|f_k^n\circ \phi -f_l^n|^2\, d\mu.
\end{align*}
For all $x\in D_k^l$, we have
\[\xi_n^\phi(\phi_l(x), x)=\xi_n(\phi \phi_l(x), \phi(x))=\xi_n(\phi_k\phi(x), \phi(x))=|f_k^n(\phi(x))|^2.\]
We therefore have
\begin{align*}
\Vert \xi_n^\phi-\xi_n\Vert_1&=\sum_{k, l}\int_{D_k^l}|\xi_n^\phi(\phi_l(x), x)-\xi_n(\phi_l(x), x)|\, d\mu(x)=\sum_{k, l}\int_{D_k^l}|\, |f_k^n\circ \phi|^2-|f_l^n|^2 |\, d\mu\\
&\leq \Biggl( \sum_{k, l}\int_{D_k^l}( |f_k^n\circ \phi|-|f_l^n|)^2\, d\mu \Biggr)^{1/2}\Biggl( \sum_{k, l}\int_{D_k^l}(|f_k^n\circ \phi|+|f_l^n|)^2\, d\mu \Biggr)^{1/2}\\
&\leq \Vert u_\phi^*u_nu_\phi -u_n\Vert_2 \Biggl( \sum_{k, l}\int_{D_k^l}2(|f_k^n\circ \phi|^2+|f_l^n|^2)\, d\mu \Biggr)^{1/2}\\
&= 2\Vert u_\phi^*u_nu_\phi -u_n\Vert_2 \to 0.
\end{align*}

We verify the equation in assertion (i).
For a Borel subset $A\subset X$, let $\Delta A\defeq \{ \, (x, x)\mid x\in A\, \}$ be the diagonal set.
Then its indicator function $1_{\Delta A}$ is a vector in $L^2(\calr, \mu^1)$.
Since $u_n$ is a unitary, we have $\Vert u_n1_{\Delta A}\Vert_2^2=\Vert 1_{\Delta A}\Vert_2^2=\mu(A)$.
By definition of the operator $u_\phi$ on $L^2(\calr, \mu^1)$ for a local section $\phi$ of $\calr$, we have $\Vert u_n1_{\Delta A}\Vert_2^2=\int_A \sum_{y\in [x]_\calr}\xi_n(y, x)\, d\mu(x)$.
This is equal to $\mu(A)$ for every Borel subset $A\subset X$, and therefore we obtain $\sum_{y\in [x]_\calr}\xi_n(y, x)=1$ for $\mu$-almost every $x\in X$.
The other equation follows if $u_n$ is replaced by $u_n^*$.
Assertion (i) follows.

Suppose that $\calr$ is strongly ergodic. The space $(X,\mu )$ is atomless (since $M$ has property Gamma), so strong ergodicity implies that $\calr$ is not amenable, and hence $(\xi_n)$ is balanced by Lemma \ref{lem:bal}.
Suppose toward a contradiction that $(\xi_n)$ is not asymptotically diffuse, i.e., there is some Borel subset $D\subset \calr$ such that $\mu^1(D)<\infty$ and $\Vert 1_D\xi_n\Vert_1\not\to 0$.
Then by Lemma \ref{lem:diffuse}, after passing to a subsequence of $(\xi_n)$, the $\xi_n$-measure of the diagonal in $\calr$ is uniformly positive, and hence $\Vert f_1^n\Vert_2$ is uniformly positive.
It follows from $\tau(u_n)=0$ that $f_1^n$ belongs to $L^2_0(X)$, the orthogonal complement of the constants in $L^2(X)$.
The sequence $(f_1^n)$ in $L^2_0(X)$ is asymptotically invariant for $\calr$, and $f_1^n$ further belongs to $L^\infty(X)$ with $\Vert f_1^n\Vert_\infty \leq 1$.
As seen in Remark \ref{rem-proj} below, the existence of such $(f_1^n)$ implies that $\calr$ is not strongly ergodic, a contradiction.
Thus $(\xi_n)$ is an inner amenability sequence for $\calr$.
Assertion (ii) follows.
\end{proof}

\begin{rem}\label{rem-proj}
In the end of the proof of Proposition \ref{prop-gamma-ia}, we used the following fact:
Let $\calr$ be an ergodic discrete p.m.p.\ equivalence relation on a standard probability space $(X, \mu)$.
If $\calr$ admits an asymptotically invariant sequence of unit vectors in $L_0^2(X)$ which are also vectors in $L^\infty(X)$ with bounded $L^\infty$-norm, then one can find a sequence $(q_n)$ of projections in $L^\infty(X)$ which is asymptotically invariant as a sequence in $L^2(X)$ with $\Vert q_n\Vert_2$ uniformly positive, and hence $\calr$ is not strongly ergodic.
This is proved by considering the ultrapower $M^\omega$ of the factor $M$ associated with $\calr$, where $\omega$ is a non-principal ultrafilter on $\N$.
For the reader's convenience, we give a proof of this fact since we could not find a reference proving this statement explicitly (although it is probably well known among experts).
The following proof is based on the argument in \cite[Theorem 3.2]{wright}.

Let $\ell^\infty(M)$ be the von Neumann algebra of bounded sequences of elements of $M$.
Let $\cali$ be the two-sided ideal of $\ell^\infty(M)$ consisting of sequences $(x_n)_{n\in \N}$ such that $\lim_\omega \tau(x_n^*x_n)=0$, where $\tau$ is the faithful normal trace on $M$.
We can make the quotient algebra $M^\omega \defeq \ell^\infty(M)/\cali$ into a von Neumann algebra with the trace $\tau_\omega$ defined by $\tau_\omega((x_n)+\cali )\defeq \lim_\omega \tau(x_n)$ (we refer to \cite[Sections A.4 and A.5]{ss} for basic facts on $M^\omega$).
The algebra $M$ is embedded into $M^\omega$ via the map $x\mapsto (x, x,\cdots)+\cali$.
We set $A\defeq L^\infty(X)$.
The algebra $A^\omega$ is then naturally identified with the von Neumann subalgebra of $M^\omega$.
We will focus on the abelian von Neumann subalgebra $A^\omega \cap M'$ of $M^\omega$.

By the assumption that $\calr$ admits an asymptotically invariant sequence of unit vectors in $L_0^2(X)$ which are also vectors in $L^\infty(X)$ with bounded $L^\infty$-norm, the algebra $A^\omega \cap M'$ is not isomorphic to $\C 1$, and therefore one can find a projection $p$ in $A^\omega \cap M'$ such that $p\neq 0, 1$.
Pick $x_n\in A$ with $\Vert x_n\Vert \leq 1$ and $(x_n)+\cali =p$.
Since $p$ is a projection, after replacing $x_n$ by $x_n^* x_n$, we may assume that $0\leq x_n\leq 1$.
We set $x\defeq (x_n)\in \ell^\infty(A)$.
Then $0\leq x\leq 1$ and $x^2-x\in \cali$.
Let $B$ be the von Neumann subalgebra of $\ell^\infty(A)$ generated by $x$.
Then $B$ is identified with the algebra of continuous functions on its Gelfand spectrum, denoted by $\Sigma$.
Put $J\defeq B\cap \cali$ and let $S$ be the closed subset of $\Sigma$ corresponding to the ideal $J$.
In other words, we identify $J$ with the subalgebra of all continuous functions on $\Sigma$ that vanish on $S$.

Let $q\in B$ be the spectral projection of $x$ corresponding to the interval $[1/2, 1]$. 
Then $q=xu$ for some $u\in B$ and $\Vert x(1-q)\Vert \leq 1/2$.
We have $xq=x^2u\equiv xu=q$ modulo $J$ and therefore $xq\equiv q$ modulo $J$.

We claim that $x-q\in J$.
If this is proved, then $p=q+\cali$, and the sequence $q\in \ell^\infty(A)$ consists of projections of $A$.
Since $p$ commutes with every element of $M$ and hence of $[\calr]$, the desired sequence is obtained as some subsequence of $q$.
We prove that $x(s)=q(s)$ for each $s\in S$, which implies $x-q\in J$.
Since $x^2-x\in J$, we have $x(s)=0$ or $1$.
If $x(s)=0$, then $q(s)=x(s)q(s)=0$.
Next suppose $x(s)=1$.
Since $q$ is a projection, we have $q(s)=0$ or $1$.
If $q(s)=0$, then we would have $1=|x(s)|=|x(s)-x(s)q(s)|\leq \Vert x(1-q)\Vert \leq 1/2$, a contradiction.
Thus $q(s)=1$.
The claim follows.

We refer to \cite[Theorem A.5.3]{ss} for other kinds of results on lifting an element of $M^\omega$ to a sequence of $\ell^\infty(M)$.
We note that while the proof of \cite[Theorem A.5.3]{ss} is based on the ingenious inequality in \cite[Lemma A.5.2]{ss}, the argument we presented above does not rely on that inequality.  
\end{rem}

The following is a simplification of Proposition \ref{prop-gamma-ia} (ii):

\begin{cor}\label{cor-gamma}
Let $\calr$ be a strongly ergodic, discrete p.m.p.\ equivalence relation.
If the von Neumann algebra associated to $\calr$ has property Gamma, then $\calr$ is inner amenable.
\end{cor}

\begin{ex}\label{ex-vaes}
The converse of Corollary \ref{cor-gamma} is not true.
A counterexample is obtained via the Vaes group (\cite{vaes}), which is defined as follows.
Let $(p_n)_{n=0}^\infty$ be a sequence of mutually distinct prime numbers.
We set
\[H_n\defeq (\Z /p_n\Z)^3,\quad K\defeq \bigoplus_{n=0}^\infty H_n \quad \text{and}\quad \Lambda \defeq \mathit{SL}_3(\Z).\]
Let $\Lambda$ act on $H_n$ by automorphisms and act on $K$ diagonally.
For a non-negative integer $N$, we define the subgroup $K_N\defeq \bigoplus_{n=N}^\infty H_n$ of $K$.
We set $G_0\defeq \Lambda \ltimes K$ and inductively define the amalgamated free product
\[G_{N+1}\defeq G_N\ast_{K_N}(K_N\times \Z).\]
Let $G_N$ include in $G_{N+1}$ as the first factor subgroup, and let $G$ be the union $G\defeq \bigcup_NG_N$.
The group $G$ is called the Vaes group, and remarkably while it is ICC and inner amenable, the group factor $LG$ does not have property Gamma (\cite{vaes}).
We note that every $g\in G_{N+1}$ normalizes each $H_n$ with $n\geq N$.

Let $Z\defeq \prod_{n=0}^\infty H_n$ be the compact group, equip $Z$ with the normalized Haar measure, and regard $K$ as a subgroup of $Z$ naturally.
Let $K$ act on $Z$ by translation, and co-induce the action $G\c X\defeq \prod_{G/K}Z$.
After fixing a section $s\colon G/K\to G$ of the quotient map, this action $G\c X$ is defined by $(gf)(b)=k^{-1}f(g^{-1}b)$ for $f\in X$, $g\in G$ and $b\in G/K$, where the element $k\in K$ is determined by $s(g^{-1}b)k=g^{-1}s(b)$.
We have the probability measure $\mu$ on $X$ given by the product of the Haar measure on $Z$, and let $\calr$ be the orbit equivalence relation of the action $G\c (X, \mu)$.

We show that $\calr$ is inner amenable.
Define $\xi_n\in L^1(\calr, \mu^1)$ by $\xi_n(gx, x)\defeq 1_{H_n}(g)/|H_n|$ for $g\in G$ and $x\in X$.
We claim that $(\xi_n)$ is an inner amenability sequence for $\calr$.
For every $g\in G$, if $n$ is large enough, then $g$ normalizes $H_n$ and hence $\xi_n^g=\xi_n$.
Therefore by Remark \ref{rem:suffice}, it suffices to show that $\sup_{h\in H_n}\mu(hA\bigtriangleup A)\to 0$ for all Borel subsets $A\subset X$.
For all $h\in H_n$, $f\in X$ and $b\in G/K$, if $n$ is large enough, depending on $b$ and being independent of $f$, then $hb=b$ and the action of $h$ is given by $(hf)(b)=(s(b)^{-1}hs(b))f(b)$.
The element $s(b)^{-1}hs(b)$ belongs to $H_n$ and does not change the coordinates $H_k$ in $Z=\prod_{k=0}^\infty H_k$ with $k<n$.
Therefore $\sup_{h\in H_n}\mu(hA\bigtriangleup A)\to 0$ if $A$ is a cyrindrical subset of $X=\prod_{G/K}Z=\prod_{G/K}\prod_{k=0}^\infty H_k$, and the claim follows.

We next show that the von Neumann algebra $M$ associated to $\calr$ does not have property Gamma.
Suppose toward a contradiction that $M$ has property Gamma, and let $(u_n)$ be a sequence of unitaries of $M$ such that $\tau(u_n)=0$ and $\Vert [x, u_n]\Vert_2\to 0$ for every $x\in M$.
The action $G\c (X, \mu)$ has stable spectral gap (see the beginning of Subsection \ref{subsec-ssg} for the definition) because its restriction to $\Lambda =\mathit{SL}_3(\Z)$ is mixing.
Therefore if $P\colon L^2(M)\to \ell^2(G)\otimes \C 1$ denotes the orthogonal projection, where $L^2(M)$ is naturally identified with $\ell^2(G)\otimes L^2(X)$, then $\Vert P(u_n)-u_n\Vert_2\to 0$ and hence $\Vert P(u_n)\Vert_2\to 1$.
Since $P$ is $G$-equivariant, where $G$ acts on $M$ by conjugation, the sequence $(P(u_n))$ asymptotically commutes with every element of $G$.
The restriction of $P$ to $M$ is the conditional expectation onto the factor $LG$, and hence the operator norm of $P(u_n)$ is at most $1$. 
We also have $\tau(P(u_n))=\tau(u_n)=0$, and it turns out from \cite[Corollary 3.8]{connes-ap} (or \cite[Lemma A.7.3]{ss}) that $LG$ has property Gamma.
This contradicts the result of Vaes \cite{vaes}.
\end{ex}


\section{Compact extensions and inner amenability}\label{sec-cpt-ia}

As observed by Giordano-de la Harpe \cite{gdlh}, if a countable group $G$ is inner amenable, then every finite index subgroup $H$ of $G$ is inner amenable as well.
We can rephrase their argument as follows: Let $\bm{m}$ be a diffuse, conjugation-invariant mean on $G$.
We define a mean $\check{\bm{m}}$ on $G$ by $\check{\bm{m}}(D)\defeq \bm{m}(D^{-1})$ for a subset $D\subset G$.
Let $\check{\bm{m}} \ast \bm{m}$ be the convolution defined by $(\check{\bm{m}}\ast \bm{m})(D)\defeq \int_G \bm{m}(g^{-1}D)\, d\check{\bm{m}}(g)$ for a subset $D\subset G$, which is a diffuse, conjugation-invariant mean on $G$.
Since $H$ is of finite index in $G$, we have $\bm{m}(g_0H)>0$ for some $g_0\in G$.
Then
\[(\check{\bm{m}} \ast \bm{m})(H)=\int_G\bm{m}(gH)\, d\bm{m}(g)\geq \int_{g_0H}\bm{m}(gH)\, d\bm{m}(g)=\bm{m}(g_0H)^2>0.\]
Thus the normalized restriction of $\check{\bm{m}} \ast \bm{m}$ to $H$ is a diffuse, conjugation-invariant mean on $H$, and $H$ is inner amenable.

In this section, we generalize a version of this convolution argument to show that inner amenability is preserved under compact extensions of ergodic discrete p.m.p.\ groupoids.

Let $(\calg, \mu)$ be a discrete p.m.p.\ groupoid.
Let $\eta ,\xi \in L^1(\mathcal{G},\mu ^1 )$.
The \textit{convolution} of $\eta$ and $\xi$, denoted $\eta \ast \xi$, is defined by
\[
(\eta \ast \xi )(\gamma ) \defeq \sum _{\delta \in  \calg^{s(\gamma )}}\eta (\gamma \delta ) \xi (\delta ^{-1}) = \sum _{\delta \in  \calg^{r(\gamma )}} \eta (\delta )\xi (\delta ^{-1} \gamma ) = \sum _{\substack{\delta _1, \delta _0 \in \calg \\ \delta _1\delta _0 = \gamma}} \eta (\delta _1)\xi (\delta _0)  .
\]
For $\xi \in L^1(\mathcal{G},\mu ^1 )$, we define $\check{\xi}\in L^1(\mathcal{G},\mu ^1 )$ by $\check{\xi}(\gamma ) \defeq \xi (\gamma ^{-1})$.
We then have $\Vert \check{\xi}\Vert_1=\Vert \xi \Vert_1$ and $(\eta \ast \xi)^\vee =\check{\xi} \ast \check{\eta}$.

\begin{lem}\label{lem-conv}
Let $\eta, \xi \in L^1(\calg, \mu^1)$.
Then
\begin{enumerate}
\item[(i)] if $\eta$ is non-negative and there is some $c>0$ such that $\sum_{\gamma \in \calg_x}\eta(\gamma)\leq c$ and $\sum_{\gamma \in \calg^x}\eta(\gamma)\leq c$ for almost every $x\in \calg^0$, then $\Vert \eta \ast \xi \Vert_1=\Vert \xi \ast \eta \Vert_1\leq c\Vert \eta \Vert_1$.
\item[(ii)] For every $\phi \in [\calg]$, we have $(\eta \ast \xi)^\phi =\eta^\phi \ast \xi^\phi$.
\end{enumerate}
\end{lem}

\begin{proof}
Assertion (i) follows from:
\begin{align*}
\Vert \eta \ast \xi \Vert_1&\leq \int_{\calg^0}\sum_{\gamma \in \calg_x}\sum_{\delta \in \calg^x}\eta(\gamma \delta )|\xi(\delta^{-1})|\, d\mu(x)=\int_{\calg^0}\sum_{\delta \in \calg^x}\sum_{\gamma \in \calg_{s(\delta)}}\eta(\gamma)|\xi(\delta^{-1})|\, d\mu(x)\\
&\leq c\int_{\calg^0}\sum_{\delta \in \calg^x}|\xi(\delta^{-1})|\, d\mu(x)=c\Vert \xi \Vert_1
\end{align*}
and $\Vert \xi \ast \eta \Vert_1=\Vert (\xi \ast \eta)^\vee \Vert_1=\Vert \check{\eta} \ast \check{\xi} \Vert_1\leq c\Vert \check{\xi}\Vert_1$.
Putting $\psi =\phi^{-1}$, we obtain assertion (ii) from:
\begin{align*}
(\eta \ast \xi)^\phi(\gamma)&=(\eta \ast \xi)(\gamma^\psi)=\sum_{\delta \in \calg^{\phi^o(s(\gamma))}}\eta(\gamma^\psi \delta)\xi(\delta^{-1})=\sum_{\delta \in (\calg^{\phi^o(s(\gamma))})^\phi}\eta((\gamma \delta)^\psi)\xi((\delta^{-1})^\psi)\\
&=\sum_{\delta \in \calg^{s(\gamma)}}\eta^\phi(\gamma \delta)\xi^\phi(\delta^{-1})=(\eta^\phi \ast \xi^\phi)(\gamma),
\end{align*}
where we use $(\gamma \delta)^\psi =\gamma^\psi \delta^\psi$ and $(\delta^{-1})^\psi =(\delta^\psi)^{-1}$ in the third equation.
\end{proof}

\begin{lem}\label{lem:conv}
Let $(\eta _n)_{n\in \N}$ and $(\xi _n)_{n\in \N}$ be inner amenability sequences for $(\mathcal{G},\mu )$.
Then there exists a sequence $m_1<m_2<\cdots$ of positive integers such that $(\eta _n \ast \check{\xi} _{m_n})_{n\in \N}$ is also an inner amenability sequence for $(\mathcal{G},\mu )$.
\end{lem}

\begin{proof}
We first show that all sequences of the form $(\eta _n \ast \check{\xi} _{m_n})_{n\in \N}$ satisfy conditions (i), (ii) and (iv) of Definition \ref{def:groupoid}.
Condition (iv) follows from direct computation.
Condition (ii) follows from
\begin{align*}
\Vert (\eta_n\ast \check{\xi}_{m_n})^\phi -\eta_n\ast \check{\xi}_{m_n}\Vert_1\leq \Vert \eta_n^\phi -\eta_n\Vert_1+\Vert \check{\xi}_{m_n}^\phi -\check{\xi}_{m_n}\Vert_1,
\end{align*}
where Lemma \ref{lem-conv} applies.
To check condition (i), we set $\eta_A\defeq 1_{\calg_A}\eta$ for $\eta \in L^1(\calg, \mu^1)$ and a Borel subset $A\subset \calg^0$.
We also set $A^c\defeq \calg^0\setminus A$.
For $\ve >0$ and a Borel subset $A\subset \calg^0$, let $E_{\ve, A}$ be the set of non-negative unit vectors $\eta \in L^1(\calg, \mu^1)$ such that $\sum_{\gamma \in \calg_x}\eta(\gamma)=1=\sum_{\gamma \in \calg^x}\eta(\gamma)$ for almost every $x\in \calg^0$, $|\, \Vert \eta_A \Vert_1-\mu(A)\, |<\ve$, and $|\, \Vert \eta_{A^c} \Vert_1-\mu(A^c)\, |<\ve$.

Let $\eta, \xi \in E_{\ve, A}$.
We show that $|\, \Vert 1_{\calg_A}(\eta \ast \xi)\Vert_1-\mu(A)\, |<5\ve +5\ve^{1/2}$.
This is enough to imply condition (i) for all sequences of the form $(\eta _n \ast \check{\xi} _{m_n})_{n\in \N}$.
Since $\eta, \xi \in E_{\ve, A}$, we have $\Vert \eta -\eta_A-\eta_{A^c}\Vert_1<2\ve$ and the similar inequality for $\xi$.
Therefore
\[\eta \ast \xi \approx_{4\ve} (\eta_A+\eta_{A^c})\ast (\xi_A+\xi_{A^c})=\eta_A\ast \xi_A+\eta_{A^c}\ast \xi_{A^c},\]
where $f\approx_c g$ means $\Vert f-g\Vert_1<c$ for $f, g\in L^1(\calg, \mu^1)$.
Then $1_{\calg_A}(\eta \ast \xi)\approx_{4\ve}\eta_A\ast \xi_A$.
We also have $\Vert \eta_A\Vert_1=\int_A\sum_{\gamma \in \calg_x\cap \, \calg_A}\eta(\gamma)\, d\mu(x)$, and since this is more than $\mu(A)-\ve$ and the integrand in the right hand side is non-negative and at most $1$ almost everywhere, there exists a Borel subset $B\subset A$ such that $\mu(A\setminus B)<\ve^{1/2}$ and $\sum_{\gamma \in \calg_x\cap \, \calg_A}\eta(\gamma)>1-\ve^{1/2}$ for all $x\in B$.
Then $\Vert \xi_A-\xi_B\Vert_1\leq 2\mu(A\setminus B)<2\ve^{1/2}$, and 
\begin{align*}
\Vert \eta_A\ast \xi_A\Vert_1&\approx_{2\ve^{1/2}}\Vert \eta_A\ast \xi_B\Vert_1=\int_{\calg^0}\sum_{\gamma, \delta \in \calg_x}\eta_A(\gamma)\xi_B(\delta^{-1})\, d\mu(x)\\ 
& =\int_B\sum_{\gamma \in \calg_x\cap \, \calg_A}\eta(\gamma)\sum_{\delta \in \calg_x\cap \, \calg_B}\xi(\delta^{-1})\, d\mu(x) 
\approx_{\ve^{1/2}}\int_B\sum_{\delta \in \calg_x\cap \, \calg_B}\xi(\delta^{-1})\, d\mu(x)\\
&=\Vert \xi_B\Vert_1\approx_{2\ve^{1/2}}\Vert \xi_A\Vert_1\approx_\ve \mu(A),
\end{align*}
where $a\approx_cb$ means $|a-b|<c$ for $a, b\in \R$.
Thus $\Vert 1_{\calg_A}(\eta \ast \xi)\Vert_1\approx_{5\ve +5\ve^{1/2}}\mu(A)$.

For condition (iii), we will need to choose $m_n$ more carefully.
Let $D_1\subset D_2\subset \cdots$ be a sequence of bounded Borel subsets of $\mathcal{G}$ with $\mathcal{G}=\bigcup _n D_n$.
For each $n$, let $F_n \colon \mathcal{G}\rightarrow \R$ be defined by $F_n (\delta ) \defeq \sum _{\gamma \in D_n \cap \, \calg_{r(\delta )}}\eta _n (\gamma \delta )$.
Then $\| 1_{D_n} (\eta _n \ast \check{\xi} _m) \| _1 = \| F_n \xi _m \| _1$ and condition (iv) of $(\eta _n)$ implies that $\| F_n \| _1 = \mu ^1 (D_n) < \infty$.
Therefore, condition (iii) of $(\xi _n)$ implies that for all large enough $m$, we have $\| 1_{D_n} (\eta _n \ast \check{\xi} _m) \| _1 = \| F_n \xi _m \| _1 < 1/n$.
Thus, by choosing a sufficiently fast growing sequence of positive integers, $m_1<m_2<\cdots$, we can ensure that for every Borel subset $D\subset  \mathcal{G}$ with $\mu ^1 (D)<\infty$, we have $\| 1_{D}(\eta _n \ast \check{\xi} _{m_n})\| _1 \rightarrow 0$.
\end{proof}

\begin{lem}\label{lem-pigeon}
Let $(X, \mu)$ be a standard probability space, and let $(C_n)$ be a sequence of Borel subsets of $X$ having uniformly positive measure.
Then after passing to a subsequence, there is some $r>0$ with $\mu(C_n \cap C_m)>r$ for all $n$ and $m$.
\end{lem}

\begin{proof}
By assumption, there is $c>0$ such that $\mu(C_n)>c$ for all $n$.
After passing to a subsequence, we may assume that $1_{C_n}$ converges to some $f\in L^{\infty}(X,\mu )$ in the weak${}^*$-topology.
Then $f\geq 0$, and $\int f\, d\mu\geq c >0$, so we may find some $r>0$ with $\int f^2\, d\mu > r$.
Since $\int 1_{C_n}f\, d\mu \to \int f^2\, d\mu >r$, we may assume after passing to a subsequence that $\int 1_{C_n}f \, d\mu > r$ for all $n$.
It follows that for all $n$, as $m\to \infty$, we have $\int 1_{C_n}1_{C_m}\, d\mu \to \int 1_{C_n}f\, d\mu >r$, and hence $\mu (C_n\cap C_m ) >r$ for all large enough $m$. We may therefore inductively find $n_1<n_2<\cdots$ with $\mu (C_{n_i}\cap C_{n_j} ) >r$ for all $i<j$.
\end{proof}

\begin{thm}\label{thm:cpctseq}
Let $(\mathcal{G}, \mu )$ be an ergodic discrete p.m.p.\ groupoid which is inner amenable.
Let $(Z,\zeta )$ be a standard probability space, let $\alpha \colon \mathcal{G}\rightarrow \mathrm{Aut}(Z,\zeta )$ be a cocycle, and assume that the image $\alpha (\mathcal{G})$ is contained in a compact subgroup $K$ of $\mathrm{Aut}(Z,\zeta )$.
Then for every decreasing sequence $V_1\supset V_2\supset \cdots$ of neighborhoods of the identity in $K$, there exists an inner amenability sequence $(\xi _n)$ for $(\mathcal{G},\mu )$ such that for every $n$, the function $\xi _n$ is supported on $\alpha ^{-1}(V_n)$. 
\end{thm}

\begin{proof}
Fix a bi-invariant metric on $K$. 
For each $\ve >0$, let $V_{\ve}$ denote the open $\ve$-ball about the identity in $K$.
By Lemma \ref{lem:weakstar}, it is enough to show that for all $\ve >0$, we can find a mean $\bm{m}$ on $(\mathcal{G},\mu )$ as in condition (5) of Theorem \ref{thm:equiv} such that $\bm{m}(\alpha ^{-1}(V_{\ve}))=1$.

Toward this goal, fix $\ve >0$ and pick $0<\ve _2< \ve _1<\ve$ such that $V_{\ve _2}^2\subset V_{\ve _1}$. 
Since $K$ is compact, we may find $c_1,\dots , c_N\in K$ such that $K=\bigcup _{i=1}^NV_{\ve _2}c_i$.
Let $(\eta _n )_{n\in \N}$ be an inner amenability sequence for $(\mathcal{G},\mu )$.
For each non-negative unit vector $\eta \in L^1(\mathcal{G},\mu ^1 )$, let $\nu _{\eta}$ be the probability measure on $K$ given by $\nu _{\eta}(B) = \int _{\alpha ^{-1}(B)}\eta \, d\mu ^1$ for a Borel subset $B\subset K$.
After passing to a subsequence of $(\eta _n )$, we may assume that for some $i\in \{ 1,\ldots, N\}$ and $r>0$, we have $\inf _{n\in \N} \nu _{\eta _n} (V_{\ve  _2}c_i )>r$.
We put $c= c_i$ and define a function $f_n$ on $\calg^0$ by
\[f_n(x) \defeq \sum _{\delta \in \alpha ^{-1}(V_{\ve _2}c)_x}\eta _n (\delta ).\]
Then $\int _{\mathcal{G}^0} f_n \, d\mu = \nu _{\eta _n} (V_{\ve _2}c)>r$ and $0\leq f_n \leq 1$, and hence the sets $C_n$ for $n\in \N$ defined by $C_n \defeq \{ \, x\in \mathcal{G}^0 \mid f_n(x) > r^2 \, \}$ have uniformly positive measure.
By Lemma \ref{lem-pigeon}, after passing to a subsequence of $(\eta _n )$, we may assume without loss of generality that the sets $C_n\cap C_m$ for $n,m\in \N$ have uniformly positive measure, and $\inf _{n,m\in \N} \int _{\mathcal{G}^0} f_nf_m\, d\mu > r_0$ for some $r_0 >0$.
Therefore
\begin{align*}
&\nu _{\eta _n\ast \check{\eta} _m}(V_{\ve _1})  = \int _{\mathcal{G}^0} \sum _{\substack{\delta _1, \delta _0 \in \calg_x \\ \alpha (\delta _1\delta _0^{-1})\in V_{\ve _1} }}\eta _n (\delta _1)\eta _m (\delta _0) \, d\mu(x)  \\
&\geq \int  _{\mathcal{G}^0} \sum _{\delta _1\in \alpha ^{-1}(V_{\ve _2}c)_x}\eta _n (\delta _1) \sum _{\delta _0 \in \alpha ^{-1}(V_{\ve _2}c)_x} \eta _m (\delta _0 ) \, d\mu(x) = \int _{\mathcal{G}^0} f_n(x)f_m(x)\, d\mu(x) > r_0 .
\end{align*}
Thus, by Lemma \ref{lem:conv}, by choosing an appropriate subsequence $m_1<m_2<\cdots$, we obtain an inner amenability sequence $\xi_n \defeq  \eta _n \ast \check{\eta}_{m_n}$ satisfying $\inf _{n\in \N}\nu _{\xi_n}(V_{\ve _1}) > r_0>0$.

We may therefore assume without loss of generality that our original sequence $(\eta _n)$ already satisfies $\inf _{n\in \N}\nu _{\eta _n}(V_{\ve _1})>r_0>0$.
Since $V_{\ve _1}$ is symmetric, after replacing $\eta _n$ by $(\eta _n + \check{\eta}_n )/2$, we may assume that each $\eta _n$ is symmetric as well.
We may also assume, after passing to a subsequence, that the sequence $(\nu _{\eta _n} )$ converges to some probability measure $\nu _{\infty}$ in the compact space of Borel probability measures on $K$.
Then for all $\ve _0$ with $\ve _1 <\ve _0 < \ve$, we have $\nu _{\infty}(V_{\ve_0}) \geq \limsup _n \nu _{\eta _n} (\overline{V_{\ve _1}})\geq r_0 > 0$.
Since, as $\ve _0$ varies, the boundaries $\partial V_{\ve _0}$ are pairwise disjoint, we may find some $\ve _0$ with $\ve _1<\ve _0 <\ve$ such that $\nu _{\infty}(\partial V_{\ve _0}) = 0$.
Let $U\defeq  V_{\ve _0}$ so that $U\subset V_{\ve}$, $\nu _{\infty}(U)>0$ and $\nu _{\infty}(\partial U ) =0$.

Since $\nu _{\infty}(\partial U )=0$, it follows that $\nu _{\eta _n}(U)\rightarrow \nu _{\infty}(U)$.
Let $\omega$ be a non-principal ultrafilter on $\N$ and let $\bm{m}_1$ be the weak${}^*$-limit $\bm{m}_1 =\lim _{n\rightarrow  \omega} \eta _n$ in $L^{\infty}(\mathcal{G} ,\mu ^1  )^*$, so that $\bm{m}_1$ is a mean on $(\mathcal{G},\mu )$ satisfying condition (5) of Theorem \ref{thm:equiv} with $\bm{m}_1(\alpha ^{-1}(U))=\lim _{n\rightarrow \omega} \nu _{\eta _n}(U) =\nu _{\infty}(U)> 0$.

\begin{claim}\label{claim:alphaU}
Let $\phi \in [\mathcal{G}]$.
Then $\bm{m}_1(\alpha ^{-1}(U)\setminus \alpha ^{-1}(U)^{\phi} ) = 0$. 
\end{claim}

\begin{proof}
Let $W$ be an open neighborhood of the identity in $K$.
Let $W_1$ be a symmetric open neighborhood of the identity in $K$ with $W_1^2\subset W$, and let $b_1,\dots , b_M\in K$ be such that $K=\bigcup _{i=1}^M W_1b_i$.
For each $i\in \{ 1,\ldots, M\}$, we set $Y_i \defeq \{ \, x\in  \mathcal{G}^0 \mid \alpha (\phi_x) \in W_1b_i \, \}$ so that $\mathcal{G}^0=\bigcup _{i=1}^M Y_i$.
Then $\bm{m}_1(\bigcup_{i=1}^M \mathcal{G}_{Y_i}) =1$.
If $\gamma \in (\bigcup_{i=1}^M \mathcal{G}_{Y_i}) \cap (\alpha ^{-1}(U)\setminus \alpha ^{-1}(U)^{\phi})$, then there is some $i\in \{ 1,\ldots, M\}$ such that $\alpha (\phi_{r(\gamma )} ), \alpha (\phi_{s(\gamma )}) \in W_1b_i$, and hence
\begin{align*}
\alpha (\gamma ^{[\phi ^{-1}]})&= \alpha (\phi_{r(\gamma )} )\alpha (\gamma ) \alpha (\phi_{s(\gamma )})^{-1} \\
&= \alpha (\phi_{r(\gamma )}) \alpha (\phi_{s(\gamma )})^{-1} \alpha (\phi_{s(\gamma )})\alpha (\gamma ) \alpha (\phi_{s(\gamma )})^{-1} \in W_1b_ib_i^{-1}W_1^{-1}U \subset WU ,
\end{align*}
and so $\gamma ^{[\phi ^{-1}]} \in \alpha ^{-1}(WU  \setminus U)$.
This shows that
\[\Biggl(\bigcup _{i=1}^M \mathcal{G}_{Y_i}\Biggr) \cap (\alpha ^{-1}(U)\setminus \alpha ^{-1}(U)^{\phi})\subset \alpha ^{-1}(WU\setminus U) ^{\phi },\]
and therefore
\[
\bm{m}_1 (\alpha ^{-1}(U)\setminus \alpha ^{-1}(U)^{\phi} ) \leq \bm{m}_1 (\alpha ^{-1}(\overline{WU}\setminus U)) = \lim _{n\rightarrow \omega}\nu _{\eta _n}(\overline{WU}\setminus U)\leq \nu _{\infty}(\overline{WU}\setminus U) .
\]
Since $\nu _{\infty}(\partial U ) = 0$, we can make $\nu _{\infty}(\overline{WU}\setminus U )$ as small as we like by choosing an appropriate neighborhood $W$ of the identity in $K$.
This proves the claim.
\end{proof}

Consider now the (countably additive, finite Borel) measure $\mu _U$ on $\mathcal{G}^0$ given by $\mu _U(A)\defeq \bm{m}_1(\mathcal{G}_A \cap \alpha ^{-1}(U))$.
The measure $\mu _U$ is absolutely continuous with respect to $\mu$, and for every $\phi \in [\mathcal{G}]$ and every Borel subset $A\subset \mathcal{G}^0$, by Claim \ref{claim:alphaU}, we have
\begin{align*}
\mu _U (\phi^{-1}A)&= \bm{m}_1(\mathcal{G}_{\phi^{-1}A}\cap \alpha ^{-1}(U))= \bm{m}_1((\mathcal{G}_A \cap \alpha ^{-1}(U))^{\phi}) = \bm{m}_1(\mathcal{G}_A \cap \alpha ^{-1}(U))\\
& = \mu _U (A).
\end{align*}
Therefore, the Radon-Nikodym derivative $d\mu _U/d\mu$ must be constant by ergodicity of $(\mathcal{G},\mu )$, and hence $\bm{m}_1(\mathcal{G}_A\cap \alpha ^{-1}(U))= \mu (A)\bm{m}_1(\alpha ^{-1}(U))$ for every Borel subset $A\subset \mathcal{G}^0$.
Define the mean $\bm{m}$ on $(\mathcal{G},\mu )$ by $\bm{m}(D) \defeq \bm{m}_1(D\cap \alpha ^{-1}(U))/\bm{m}_1(\alpha ^{-1}(U))$.
It is now clear that $\bm{m}$ is a mean on $(\mathcal{G},\mu )$ satisfying condition (5) of Theorem \ref{thm:equiv}, and moreover $\bm{m}(\alpha ^{-1}(V_{\ve}))\geq  \bm{m}(\alpha ^{-1}(U))=1$.
\end{proof}

\begin{cor}\label{cor-finite-index-cpt-ext}
Let $(\mathcal{G}, \mu )$ be an ergodic discrete p.m.p.\ groupoid, and let $\calh$ be a finite-index Borel subgroupoid of $\calg$.
If $(\calg, \mu)$ is inner amenable, then $(\calh, \mu)$ is inner amenable.
\end{cor}

\begin{proof}
Let $N$ be the index of $\calh$ in $\calg$ and let $\tau \colon \calg \to \Sigma$ be the index cocycle, where $\Sigma$ is the symmetric group of $N$ letters.
This cocycle is constructed in \cite[Section 1]{fsz} when $\calg$ is principal, and it is similarly defined for general $\calg$ as follows:
We choose Borel maps $\psi_1,\ldots, \psi_N\colon \calg^0\to \calg$ such that for almost every $x\in \calg^0$, $\psi_1(x)=x$, $\psi_i(x)\in \calg^x$, and the sets $\{ \, \psi_i(x)h\mid h\in \calh^{s(\psi_i(x))}\, \}$ with $i=1,\ldots, N$ partition $\calg^x$.
For $\gamma \in \calg$ with $x\defeq s(\gamma)$ and $y\defeq r(\gamma)$, we define the permutation $\tau(\gamma)\in \Sigma$ so that $\gamma \psi_i(x)=\psi_{\tau(\gamma)(i)}(y)h$ for some $h\in \calh$.
Then $\calh$ is equal to the inverse image under $\tau$ of the subgroup $\{
\, \sigma \in \Sigma \mid \sigma(1)=1\, \}$.
The corollary thus follows from Theorem \ref{thm:cpctseq}.
\end{proof}

\begin{cor}\label{cor:cpctext}
Let $(\mathcal{G}, \mu )$ be an ergodic discrete p.m.p.\ groupoid which is inner amenable.
Let $(Z,\zeta )$ be a standard probability space, let $\alpha \colon \mathcal{G}\rightarrow \mathrm{Aut}(Z,\zeta )$ be a cocycle, and assume that the image $\alpha (\mathcal{G})$ is contained in a compact subgroup of $\mathrm{Aut}(Z,\zeta )$.
Then the extension groupoid $(\mathcal{G} ,\mu )\ltimes _{\alpha}(Z,\zeta )$ is inner amenable.

In particular, if $G$ is a countable inner amenable group which is a subgroup of a compact group $K$, $L$ is a closed subgroup of $K$, and we let $G$ act on $(K/L, \mu)$ by left multiplication, where $\mu$ is the $K$-invariant probability measure on $K/L$, then the associated translation groupoid $G\ltimes (K/L, \mu)$ is inner amenable.
\end{cor}

\begin{proof}
Let $K$ be a compact subgroup of $\aut(Z, \zeta)$ such that $\alpha(\calg)\subset K$, and let $V_1\supset V_2\supset \cdots$ be a decreasing sequence of neighborhoods of the identity in $K$ with $\bigcap_n V_n=\{ e\}$.
By Theorem \ref{thm:cpctseq}, we have an inner amenability sequence $(\xi_n)$ for $(\mathcal{G},\mu )$ such that for all $n$, $\xi_n$ is supported on $\alpha^{-1}(V_n)$.
We show that its lift $(\eta _n )$ defined by $\eta _n (\gamma , z )\defeq  \xi _n (\gamma )$ is an inner amenability sequence for the extension groupoid $(\tilde{\calg}, \tilde{\mu})\defeq (\mathcal{G}  ,\mu  )\ltimes _{\alpha}(Z,\zeta )$.
For each $\phi \in [\calg]$, its lift $\psi \in [\tilde{\calg}]$ is defined by $\psi_{(x, z)}\defeq (\phi_x, z)$, and then $\Vert \eta_n^{\psi} -\eta_n\Vert_1=\Vert \xi_n^\phi -\xi_n\Vert_1\to 0$.
By Remark \ref{rem:suffice}, once $(\eta_n)$ is shown to be balanced, it satisfies $\Vert \eta_n^\phi -\eta_n\Vert_1\to 0$ for all $\phi \in [\tilde{\calg}]$.
Conditions (iii) and (iv) of Definition \ref{def:groupoid} for $(\eta_n)$ follow from those for $(\xi_n)$, and thus the conclusion follows.

To see that $(\eta_n)$ is balanced, it suffices to verify that for all Borel subsets $A\subset \calg^0$ and $B\subset Z$, we have $\Vert 1_{\tilde{\calg}_{A\times B}}\eta_n\Vert_1\to \mu(A)\zeta(B)$.
Set $b_n\defeq \inf \{ \, \zeta(B\cap T^{-1}B)\mid T\in V_n\, \}$.
Then $b_n\nearrow \zeta(B)$ since the function $T\mapsto \zeta(B\cap T^{-1}B)$ is continuous on $\aut(Z, \zeta)$.
We have
\begin{align*}
\Vert 1_{\tilde{\calg}_{A\times B}}\eta_n\Vert_1&=\int_{A\times B}\sum_{\substack{\gamma \, \in \calg_x\cap \, \calg_A \\ \alpha(\gamma)z\in B}} \eta_n(\gamma, z)\,  d\tilde{\mu}(x, z)\\
&=\int_{A}\sum_{\gamma \in \calg_x\cap \, \calg_A} \zeta(B\cap \alpha(\gamma)^{-1}B)\, \xi_n(\gamma)\, d\mu(x),
\end{align*}
and thus $b_n\Vert 1_{\calg_A}\xi_n\Vert_1\leq \Vert 1_{\tilde{\calg}_{A\times B}}\eta_n\Vert_1\leq \zeta(B)\Vert 1_{\calg_A}\xi_n\Vert_1$.
Since $\Vert 1_{\calg_A}\xi_n\Vert_1\to \mu(A)$, the desired convergence follows.
\end{proof}

By Corollary \ref{cor:cpctext}, if $G$ is a countable, residually finite, inner amenable group, then the translation groupoid associated with any profinite free action of $G$ is inner amenable and therefore $G$ is orbitally inner amenable.

\begin{cor}\label{cor:distal}
Let $G$ be a countable inner amenable group.
Let $G\curvearrowright (X,\mu )$ be an ergodic p.m.p.\ action of $G$ which is measure distal.
Then the translation groupoid $G \ltimes (X,\mu )$ is inner amenable.
\end{cor}

\begin{proof}
We proceed by transfinite induction on the length of the distal tower associated to the action $G \curvearrowright (X,\mu )$.
At successor stages we apply Corollary \ref{cor:cpctext}, and at limit stages the translation groupoid $G\ltimes (X,\mu )$ is the inverse limit of the translation groupoids of the tower, so Proposition \ref{prop:invlim} applies.
\end{proof}



\section{Compact extensions and central sequences}\label{sec-cpt-cs}

Following the previous section, we investigate existence of a central sequence in the full group and existence of a stability sequence under compact extensions.
Main results of this section are collected in Subsections \ref{subsec-main} and \ref{subsec-main2}.
Throughout this section, let $(X, \mu)$ be a standard probability space and let $\calb$ be the measure algebra of $\mu$.

\subsection{Stability sequences}

Let $\calr$ be an ergodic discrete p.m.p.\ equivalence relation on $(X, \mu)$.
Let $(T_n, A_n)_{n\in \N}$ be a sequence of pairs of $T_n\in [\calr]$ and $A_n\in \calb$.
We call $(T_n, A_n)_{n\in \N}$ a \textit{stability sequence} for $\calr$ if the following three conditions hold:
\begin{enumerate}
\item[(1)] For all $B\in \calb$, we have $\mu(T_nB\bigtriangleup B)\to 0$.
\item[(2)] For all $g\in [\calr]$, we have $\mu (\{ gT_n\neq T_ng \})\to 0$.
\item[(3)] The sequence $(A_n)$ is asymptotically invariant for $\calr$, $T_n^2=\text{id}$ and $T_nA_n\bigtriangleup A_n=X$ for all $n\in \N$, and $T_nT_m=T_mT_n$ and $T_nA_m=A_m$ for all distinct $n, m\in \N$.
\end{enumerate}
We call $\calr$ \textit{stable} if $\calr$ is isomorphic to the direct product $\calr \times \calr_0$, where $\calr_0$ is the ergodic p.m.p.\ aperiodic hyperfinite equivalence relation.
By \cite[Theorem 3.4]{js}, $\calr$ is stable if and only if it admits a stability sequence.
The theorem also says that $\calr$ is stable if it admits a sequence $(T_n, A_n)$ satisfying conditions (1) and (2) and the following condition weaker than condition (3):
\begin{enumerate}
\item[$(4)$] The sequence $(A_n)$ is asymptotically invariant for $\calr$, and $\mu(T_nA_n\setminus A_n)$ is uniformly positive.
\end{enumerate}
We call a sequence $(T_n, A_n)$ satisfying conditions (1), (2) and (4) a \textit{pre-stability sequence} for $\calr$.


\subsection{Preliminary lemmas}

Throughout this subsection, let $\calr$ be an ergodic discrete p.m.p.\ equivalence relation on $(X, \mu)$.

\begin{lem}\label{lem-ai-cover}
If $(D_n)$ is an asymptotically invariant sequence for $\calr$ such that $\mu(D_n)$ is uniformly positive, then $\bigcup_nD_n=X$.
\end{lem}

\begin{proof}
Passing to a subsequence, we may assume that $\mu(D_n)\to d$ for some $d>0$ and also $\mu(D_n)> d/2$ for all $n$.
By Lemma \ref{lem-ai}, passing to a subsequence further, we may assume that for all $n$, the two values $\mu(\bigcap_{k=1}^nD_k^c)$ and $\mu(D_1^c)\cdots \mu(D_n^c)$ are close.
The latter value is less than $(1-d/2)^n$ and hence $\mu(\bigcap_nD_n^c)=0$.
\end{proof}

\begin{lem}\label{lem-comp}
For all $A, A', B, B'\in \calb$, the following inequality holds:
\[\mu(A\setminus B)\leq 2\mu(A\bigtriangleup A')+\mu(B\bigtriangleup B')+\mu(A'\setminus B').\]
\end{lem}

\begin{proof}
Let $\Vert \cdot \Vert_1$ denote the norm on $L^1(X, \mu)$.
The inequality follows from:
\begin{align*}
\mu(A\setminus B)&=\Vert 1_A-1_A1_B\Vert_1 \\
&\leq \Vert 1_A -1_{A'}\Vert_1 +\Vert 1_{A'}-1_{A'}1_{B'}\Vert_1 +\Vert (1_{A'}-1_A)1_{B'}\Vert_1 +\Vert 1_A (1_{B'}-1_B)\Vert_1\\
&\leq \mu(A\bigtriangleup A')+\mu(A'\setminus B')+\mu(A\bigtriangleup A')+\mu(B\bigtriangleup B').\qedhere
\end{align*}
\end{proof}

\begin{lem}\label{lem-patch}
Let $(T_n, D_n)_{n\in \N}$ be a sequence of a pair of $T_n\in [\calr]$ and $D_n\in \calb$ such that
\begin{itemize}
\item $(T_n)$ is a central sequence in $[\calr]$ and $T_n^2=\mathrm{id}$ for every $n$,
\item $(D_n)$ is an asymptotically invariant sequence for $\calr$ such that $\mu(D_n)$ is uniformly positive, and
\item $T_nD_n=D_n$ for every $n$.
\end{itemize}
Then for every $\ve >0$ and every finite subset $Q\subset [\calr]$, we can find an $S\in [\calr]$ such that
\begin{enumerate}
\item[(i)] the map $S$ is obtained by patching together pieces of the restrictions $T_n|_{D_n}$, $n\in \N$, along with a piece of the identity map such that the latter piece is small.
More precisely: for all $x\in X$ outside a subset of measure less than $\ve$, there exists an $n\in \N$ with $x\in D_n$ and $Sx=T_nx$, and for every point $y$ in the excluded subset, we have $Sy=y$;
\item[(ii)] $\mu(\{ gS\neq Sg \})<\ve$ for every $g\in Q$.
\end{enumerate}
If $(A_n)$ is further an asymptotically invariant sequence for $\calr$ with $\mu(T_nA_n\bigtriangleup A_n)=1$ for every $n$, then after replacing $A_n$ by $X\setminus A_n$ if necessary, we may assume that $\mu(A_n\cap D_n)\geq \mu(D_n)/2$ for every $n$, and we can find a $Z\in \calb$ such that
\begin{enumerate}
\item[(iii)] $\mu(SZ\setminus Z)>1/10$, and
\item[(iv)] $\mu(gZ\bigtriangleup Z)<\ve$ for every $g\in Q$.
\end{enumerate}
\end{lem}

As a result of the former assertion of the lemma, varying $\ve$ and $Q$, we obtain a central sequence $(S_m)$ in $[\calr]$ such that for each $m$, the map $S_m$ is obtained by patching together pieces of the maps $T_n|_{D_n}$, $n\in \N$, along with a piece of the identity map such that the latter piece is small.
The central sequence $(S_m)$ is hence non-trivial as long as $T_nx\neq x$ for all $n$ and all $x\in X$.
Under the assumption in the latter assertion of the lemma, we further obtain a pre-stability sequence $(S_m, Z_m)$ for $\calr$ with $\mu(S_mZ_m\setminus Z_m)>1/10$ for all $m$.
In the proof of Theorem \ref{thm-main} below, it will be significant that this lower bound ``1/10" can be taken independently of the uniform lower bound of $\mu(D_n)$.

\begin{proof}[Proof of Lemma \ref{lem-patch}]
Passing to a subsequence of $(T_n, D_n)$, we may assume that the following three conditions hold:
\begin{enumerate}
\item[(1)] $\sum_n\mu(gD_n\bigtriangleup D_n)<\ve$ for every $g\in Q$.
\item[(2)] $\sum_n\mu(\{ gT_n\neq T_ng \} )<\ve$ for every $g\in Q$.
\item[(3)] $\sum_n \sum_{k<n}\mu(T_nD_k\bigtriangleup D_k)<\ve$.
\end{enumerate}
These conditions follow from the sequence $(D_n)$ being asymptotically invariant for $\calr$, and the sequence $(T_n)$ being central in $[\calr]$.
Under the assumption in the latter assertion of the lemma, we may further assume that $\mu(D_n)\to d_1$ and $\mu(A_n\cap D_n)\to d_2$ for some $d_1>0$ and $d_2>0$ and that the following three conditions hold:
\begin{enumerate}
\item[(4)] For every $n$, setting $C_n\defeq \bigcup_{k<n}D_k$, we have $\mu(D_n)\mu(C_n^c)\geq (2/3)\mu(D_n\setminus C_n)$.
\item[(5)] For every $n$, setting $E_n\defeq (A_n\cap D_n)\setminus C_n$, we have $\mu(E_n)\geq (2/3)\mu(A_n\cap D_n)\mu(C_n^c)$.
\item[(6)] $\sum_n \mu(gE_n\bigtriangleup E_n)<\ve$ for every $g\in Q$.
\end{enumerate}
Indeed, condition (4) is obtained as follows:
If $D_1,\ldots, D_{n-1}$ are chosen, then by Lemma \ref{lem-ai}, we have $\mu(D_m\cap C_n^c)\to d_1\mu(C_n^c)$ as $m\to \infty$.
For all large $m$, $\mu(D_m)$ and $d_1$ are close and hence $\mu(D_m\cap C_n^c)=\mu(D_m\setminus C_n)$ and $\mu(D_m)\mu(C_n^c)$ are close.
Condition (4) therefore holds after relabeling $D_m$ for a sufficiently large $m$ as $D_n$.
Condition (5) is similarly obtained from Lemma \ref{lem-ai} and the convergence $\mu(A_n\cap D_n)\to d_2$.
Condition (6) is obtained from asymptotic invariance of the sequences $(D_n)$ and $(A_n)$.

We set
\[Y_1\defeq D_1,\ Y_n\defeq D_n\setminus (C_n\cup T_n^{-1}C_n)\ \ \text{for}\ \ n\geq 2,\ \ \text{and} \ \ Y\defeq \bigcup_{n=1}^\infty Y_n.\]
Note that the last union is a disjoint union.
For each $n$, we have $T_nY_n=Y_n$ because $T_n$ is an involution and $T_nD_n=D_n$.
The inclusion $Y_n\subset D_n\setminus C_n$ holds.
By condition (3), we have $\sum_n \mu(T_nC_n\bigtriangleup C_n)<\ve$.
Therefore $\sum_n \mu((D_n\setminus C_n)\setminus Y_n)<\ve$, and $\mu(\bigcup_n(D_n\setminus C_n)\setminus Y)<\ve$.
By the definition of $C_n$, the equation $\bigcup_n (D_n\setminus C_n)=\bigcup_nD_n$ holds, and this is equal to $X$ by Lemma \ref{lem-ai-cover}.
It follows that
\begin{enumerate}
\item[(7)] $\mu(X\setminus Y)<\ve$.
\end{enumerate}

We pick $g\in Q$ and estimate $\sum_n\mu(gY_n\bigtriangleup Y_n)$.
Pick $y\in Y_n\setminus gY_n$.
Since $g^{-1}y\not \in Y_n$, either $g^{-1}y\not \in D_n$ or $g^{-1}y\in C_n\cup T_n^{-1}C_n$.
In the former case, we have $y\in D_n\setminus gD_n$.
In the latter case, we have
\begin{align*}
y & \in (g(C_n\cup T_n^{-1}C_n)\setminus (C_n\cup T_n^{-1}C_n))\cap Y_n\\
& \subset \Biggl( \, \bigcup_{k<n}(gD_k\setminus D_k)\cap Y_n\Biggr)\cup \Biggl(\, \bigcup_{k<n}(gT_n^{-1}D_k\setminus T_n^{-1}D_k)\cap Y_n\Biggr). 
\end{align*}
Let $N$ be a positive integer.
We have
\begin{align*}
&\sum_{n=1}^N\mu(Y_n\setminus gY_n)\leq \sum_{n=1}^N\mu(D_n\setminus gD_n)\\
&\hspace{5em}+\sum_{n=1}^N\sum_{k=1}^{n-1}\mu((gD_k\setminus D_k)\cap Y_n)+\sum_{n=1}^N\sum_{k=1}^{n-1}\mu((gT_n^{-1}D_k\setminus T_n^{-1}D_k)\cap Y_n).
\end{align*}
By condition (1), in the right hand side, the first term is less than $\ve$, and the second term is at most
\[\sum_{n=1}^N\sum_{k=1}^{N-1}\mu((gD_k\setminus D_k)\cap Y_n)\leq \sum_{k=1}^{N-1}\mu(gD_k\setminus D_k)<\ve.\]
The third term is
\begin{align*}
&\sum_{n=1}^N\sum_{k=1}^{n-1}\mu((gT_n^{-1}D_k\cap Y_n)\setminus (T_n^{-1}D_k\cap Y_n))\\
& \leq \sum_{n=1}^N\sum_{k=1}^{n-1}(\mu((gD_k\cap Y_n)\setminus (D_k\cap Y_n))+3\mu(T_n^{-1}D_k\bigtriangleup D_k))\\
& <\sum_{n=1}^N\sum_{k=1}^{n-1} \mu((gD_k\setminus D_k)\cap Y_n)+3\ve <4\ve,
\end{align*}
where Lemma \ref{lem-comp} and condition (3) apply in the first and second inequalities, respectively.
It follows that $\sum_{n=1}^N\mu(Y_n\setminus gY_n)<6\ve$ and therefore
\begin{enumerate}
\item[(8)] $\sum_n \mu(gY_n\bigtriangleup Y_n)< 12\ve$ for every $g\in Q$.
\end{enumerate}

We define $S\in [\calr]$ as follows:
For each $n$, we set $S=T_n$ on $Y_n$ and define $S$ on $X\setminus Y$ to be the identity map.
This map $S$ is an automorphism of $X$ because $T_n$ preserves $Y_n$, and by condition (7), it satisfies condition (i).
To check condition (ii), pick $g\in Q$.
We have the inclusions
\begin{align*}
\{ gS\neq Sg\} &\subset \bigcup_n \, (\{ gS\neq Sg\} \cap Y_n)\cup (X\setminus Y),\\
\{ gS\neq Sg\} \cap Y_n &\subset (\{ gS\neq Sg\} \cap (Y_n\cap g^{-1}Y_n))\cup (Y_n\setminus g^{-1}Y_n)\ \text{and}\\
\{ gS\neq Sg\} \cap (Y_n\cap g^{-1}Y_n)&\subset \{ gT_n\neq T_ng\} \cap (Y_n\cap g^{-1}Y_n).
\end{align*}
It follows from conditions (2), (8) and (7) that
\[\mu(\{ gS\neq Sg\})\leq \sum_n(\mu(\{ gT_n\neq T_ng\})+\mu(Y_n\setminus g^{-1}Y_n))+\mu(X\setminus Y)<\ve +6\ve +\ve.\]
Condition (ii) follows after scaling $\ve$ appropriately.
This completes the proof of the former assertion of the lemma.

We prove the latter assertion of the lemma.
By conditions (4) and (5) and the inequality $\mu(A_n\cap D_n)\geq \mu(D_n)/2$, for all $n$, we have
\begin{enumerate}
\item[(9)] $\mu(E_n)\geq (2/9)\mu(D_n\setminus C_n)$.
\end{enumerate}
We set
\[Z\defeq \bigcup_n \, (A_n\cap Y_n).\]
We check condition (iii).
Since the set $E_n$ is defined as $E_n=(A_n\cap D_n)\setminus C_n$, we have the inclusions $A_n\cap Y_n\subset E_n$ and $E_n\setminus (A_n\cap Y_n)\subset T_n^{-1}C_n\setminus C_n$.
It follows that
\[0\leq \sum_n\mu(E_n)-\mu(Z)= \sum_n\mu(E_n\setminus (A_n\cap Y_n))\leq \sum_n\sum_{k<n}\mu(T_n^{-1}D_k\setminus D_k)<\ve,\]
where the last inequality follows from condition (3).
On the other hand, by condition (9), we have $\sum_n\mu(E_n)\geq (2/9)\sum_n\mu(D_n\setminus C_n)$, and by condition (7),
\[\sum_n\mu(D_n\setminus C_n)\geq \sum_n\mu(Y_n)=\mu(Y)>1-\ve.\]
It follows that $\sum_n\mu(E_n)>2/9-2\ve /9$ and $\mu(Z)>\sum_n\mu(E_n)-\ve>2/9-11\ve/9$.
The sets $SZ$ and $Z$ are disjoint because $T_nA_n$ and $A_n$ are disjoint and $S$ is equal to $T_n$ on $Y_n$.
We therefore have $\mu(SZ\setminus Z)=\mu(SZ)=\mu(Z)>1/10$, where the last inequality holds if $\ve$ is taken to be small enough.
Condition (iii) follows.

Finally we check condition (iv).
Pick $g\in Q$.
By Lemma \ref{lem-comp},
\[\mu(g(A_n\cap Y_n)\setminus (A_n\cap Y_n))\leq 3\mu(E_n\bigtriangleup (A_n\cap Y_n))+\mu(gE_n\setminus E_n).\]
Summing over $n$, we obtain
\begin{align*}
\mu(gZ\setminus Z)&\leq \sum_n\mu(g(A_n\cap Y_n)\setminus (A_n\cap Y_n))\leq \sum_n(3\mu(E_n\bigtriangleup (A_n\cap Y_n))+\mu(gE_n\setminus E_n))\\
&<3\ve +\ve,
\end{align*}
where the last inequality follows from the inclusion $A_n\cap Y_n\subset E_n$, the inequality shown in the previous paragraph, and condition (6).
Condition (iv) follows.
\end{proof}

Without assuming that $T_n$ is an involution and that $D_n$ is invariant under $T_n$, we prove the following lemma in which conclusion (i) is slightly milder than that in Lemma \ref{lem-patch}.
This will be used in the proof of Lemma \ref{lem-c} and Theorem \ref{thm-main-c}.

\begin{lem}\label{lem-patch-c}
Let $(T_n, D_n)_{n\in \N}$ be a sequence of a pair of $T_n\in [\calr]$ and $D_n\in \calb$ such that
\begin{itemize}
\item $(T_n)$ is a central sequence in $[\calr]$,
\item $(D_n)$ is an asymptotically invariant sequence for $\calr$ such that $\mu(D_n)$ is uniformly positive, and
\item $\mu(T_nD_n\bigtriangleup D_n)\to 0$ as $n\to \infty$.
\end{itemize}
Then for every $\ve >0$ and every finite subset $Q\subset [\calr]$, we can find an $S\in [\calr]$ such that
\begin{enumerate}
\item[(i)] a large piece of the map $S$ is obtained by patching together pieces of the restrictions $T_n|_{D_n}$, $n\in \N$. More precisely, for every $x\in X$ outside a subset of measure less than $\ve$, there exists an $n\in \N$ with $x\in D_n$ and $Sx=T_nx$, and
\item[(ii)] $\mu(\{ gS\neq Sg \})<\ve$ for every $g\in Q$.
\end{enumerate}
\end{lem}

\begin{proof}
As in the proof of Lemma \ref{lem-patch}, after passing to a subsequence of $(T_n, D_n)$, we may assume that conditions (1)--(3) in that proof hold.
By the third assumption in the present lemma, we may further assume that
\begin{enumerate}
\item[(10)] $\sum_n\mu(T_nD_n\bigtriangleup D_n)<\ve$.
\end{enumerate}
We set $C_n\defeq \bigcup_{k<n}D_k$, $Y_1\defeq D_1$, $Y_n\defeq D_n\setminus (C_n\cup T_n^{-1}C_n)$ for $n\geq 2$, and $Y\defeq \bigcup_{n=1}^\infty Y_n$ in the same way.
As in the proof of Lemma \ref{lem-patch}, condition (3) implies $\mu(X\setminus Y)<\ve$, and conditions (1) and (3) imply $\sum_n(gY_n\bigtriangleup Y_n)< 12\ve$ for every $g\in Q$.
We set
\[Y_1'\defeq D_1\cap T_1^{-1}D_1,\ Y_n'\defeq (D_n\cap T_n^{-1}D_n)\setminus (C_n\cup T_n^{-1}C_n)\ \text{for}\ n\geq 2,\ \text{and} \ Y'\defeq \bigcup_{n=1}^\infty Y_n'.\]
For all $n$, the inclusions $Y_n'\subset Y_n$ and $Y_n\setminus Y_n'\subset D_n\setminus T_n^{-1}D_n$ hold, and hence
\[\sum_n\mu(Y_n\setminus Y_n')\leq \sum_n\mu(D_n\setminus T_n^{-1}D_n)<\ve /2\]
by condition (10).
Hence $\mu(Y\setminus Y')<\ve /2$ and $\mu(X\setminus Y')<2\ve$.
For all $g\in Q$, we have
\[\sum_n\mu(gY_n'\bigtriangleup Y_n')<\sum_n(\mu(gY_n'\bigtriangleup gY_n)+\mu(gY_n\bigtriangleup Y_n)+\mu(Y_n\bigtriangleup Y_n'))<\ve +12\ve +\ve.\]

We define $S\in [\calr]$ as follows:
We first define it on $Y'$ so that $S=T_n$ on $Y_n'$ for every $n$.
This map $S\colon Y'\to X$ is injective.
Indeed, if $k<n$ and $T_kx=T_ny$ with $x\in Y_k'$ and $y\in Y_n'$, then it follows from $T_kY_k'\subset D_k$ that $T_kx\in D_k$ and thus $y\in T_n^{-1}D_k\subset T_n^{-1}C_n$.
This contradicts $y\in Y_n'$.
It turns out that the measures of the sets $Y'$ and $SY'$ are equal.
Pick a Borel isomorphism between $X\setminus Y'$ and $X\setminus SY'$ which is a local section of $\calr$.
We define the map $S$ on $X\setminus Y'$ to be that isomorphism.
The obtained map $S\colon X\to X$ is then an automorphism of $X$ and belongs to $[\calr]$.
Condition (i) then follows.
Along the proof in Lemma \ref{lem-patch}, condition (ii) also follows from the estimates for $Y_n'$ and $Y'$ obtained in the previous paragraph, after scaling $\ve$ appropriately.
\end{proof}

As a simple application of the last lemma, we obtain the following:

\begin{lem}\label{lem-c}
If $\calr$ is Schmidt, then there exists a central sequence $(T_n)$ in $[\calr]$ such that $T_nx\neq x$ for all $n$ and all $x\in X$.
\end{lem}

\begin{proof}
Let $(T_n)$ be a non-trivial central sequence in $[\calr]$.
Set $D_n\defeq \{ \, x\in X\mid T_nx\neq x\, \}$.
The measure $\mu(D_n)$ is uniformly positive, and $T_nD_n=D_n$ for all $n$.
We claim that $(D_n)$ is an asymptotically invariant sequence for $\calr$.
Indeed, for every $g\in [\calr]$, if $n$ is large, then for all $x\in X$ outside a subset of small measure, we have $gT_nx=T_ngx$, and if furthermore $x \in X\setminus D_n$, then $gx=T_ngx$, that is, $gx\in X\setminus D_n$.
The sequence $(X\setminus D_n)$ is therefore asymptotically invariant for $\calr$, and the claim follows.

Pick $\ve >0$ and a finite subset $Q\subset [\calr]$.
By Lemma \ref{lem-patch-c}, we can find an $S\in [\calr]$ with $\mu(\{ gS\neq Sg\})<\ve$ for all $g\in Q$ and $Sx\neq x$ for all $x\in Y$, where $Y$ is a Borel subset of $X$ with $\mu(X\setminus Y)<\ve$.
There exists an isomorphism from $X\setminus Y$ onto $S(X\setminus Y)$ which fixes no point and is a local section of $\calr$.
We define $R\in [\calr]$ as the map equal to $S$ on $Y$ and equal to that isomorphism on $X\setminus Y$.
It turns out that $\mu(\{ gR\neq Rg\})<3\ve$ for all $g\in Q$ and $Rx\neq x$ for all $x\in X$.
\end{proof}

The next lemma will be used in the proof of Theorem \ref{thm-main-c}.

\begin{lem}\label{lem-s}
Let $(T_n)$ be a sequence in $[\calr]$ such that $\mu(T_nA\bigtriangleup A)\to 0$ for all Borel subsets $A\subset X$.
Then $\mu(\{ \, x\in X\mid T_nx=Sx\neq x\, \})\to 0$ for all $S\in [\calr]$.
\end{lem}

\begin{proof}
Suppose that the conclusion fails for some $S\in [\calr]$. Set $Y\defeq \{ \, x\in X\mid Sx\neq x\, \}$. Passing to a subsequence, we may assume that $\mu(\{ \, x\in Y\mid T_nx=Sx \, \}) \rightarrow s$ for some $s>0$. Fix a non-principal ultrafilter $\omega$ on $\N$ and let $\bm{m}$ be the mean on $\calr$ defined by
\[
\bm{m}(D)\defeq \lim _{n\rightarrow \omega} \mu ( \{ \, x\in X \mid (T_n x, x) \in D \, \} ) .
\]
The mean $\bm{m}$ is balanced since $\bm{m}(\calr _A) = \lim _{n\to\omega} \mu (T_n ^{-1}A \cap A ) = \mu (A)$. Therefore, by Lemma \ref{lem:diffuse} we have $\bm{m}( \{ \, (Sx,x) \mid x\in Y \, \} ) =0$. Hence
\[
0 = \bm{m}( \{ \, (Sx,x) \mid x\in Y \, \} ) = \lim _{n\to\omega} \mu ( \{ \, x\in Y \mid T_nx = Sx \, \} ) = s,
\]
a contradiction.
\end{proof}


\subsection{Stability under compact extensions}\label{subsec-main}


\begin{thm}\label{thm-main}
Let $\calr$ be an ergodic discrete p.m.p.\ equivalence relation on $(X, \mu)$ and let $\alpha \colon \calr \to K$ be a cocycle into a compact group $K$.
If $\calr$ is stable, then for every decreasing sequence $V_1\supset V_2\supset \cdots$ of open neighborhoods of the identity in $K$, there exists a pre-stability sequence $(T_n, A_n)$ for $\calr$ such that $\alpha(T_n x, x)\in V_n$ for every $n$ and almost every $x\in X$.
\end{thm}

\begin{proof}
Throughout this proof, for $T\in [\calr]$ and $x\in X$, we denote $\alpha(Tx, x)$ by $\alpha(T, x)$ for the ease of symbols.
Fix a bi-invariant metric on $K$.
For $\ve >0$, let $V_\ve$ denote the open $\ve$-ball about the identity in $K$.
Pick $\ve >0$.
We will find a pre-stability sequence $(T_n, A_n)$ for $\calr$ such that $\alpha(T_n, x)\in V_\ve$ for all $n$ and all $x\in X$.

Let $(T_n, A_n)_{n\in \N}$ be a stability sequence for $\calr$.
Choose $0<\ve_2<\ve_1<\ve$ with $V_{\ve_2}^2\subset V_{\ve_1}$.
Since $K$ is compact, there are finitely many $c_1,\ldots, c_N\in K$ with $K=\bigcup_{i=1}^NV_{\ve_2}c_i$.
Passing to a subsequence of $(T_n, A_n)_{n\in \N}$, we can find $i\in \{ 1,\ldots, N\}$ such that the set
\[C_n\defeq \{ \, x\in X\mid \alpha(T_n, x)\in V_{\ve_2}c_i\, \}\]
has uniformly positive measure.
We put $c=c_i$.
By Lemma \ref{lem-pigeon}, passing to a subsequence, we may assume that the sets $C_n\cap C_m$ with $n, m\in \N$ have uniformly positive measure.
If $x\in C_n\cap C_m$, then
\[\alpha(T_mT_n^{-1}, T_nx)=\alpha(T_m, x)\alpha(T_n, x)^{-1}\in V_{\ve_2}c(V_{\ve_2}c)^{-1}=V_{\ve_2}^2\subset V_{\ve_1},\]
and hence
\[T_n(C_n\cap C_m)\subset \{ \, x\in X\mid \alpha(T_mT_n^{-1}, x)\in V_{\ve_1}\, \}.\]
It follows that the set in the right hand side has uniformly positive measure.
Therefore, after replacing the pair $(T_n, A_n)$ into the pair $(T_{2n+1}T_{2n}^{-1}, A_{2n+1})$, we may assume that the stability sequence $(T_n, A_n)_{n\in \N}$ is such that the set
\[E_n\defeq \{ \, x\in X\mid \alpha(T_n, x)\in V_{\ve_1}\, \}.\]
has uniformly positive measure.
We should note that for all distinct $n, m\in \N$, $T_mT_n^{-1}$ is involutive and satisfies $(T_mT_n^{-1})A_m\bigtriangleup A_m=X$ thanks to the condition, $T_nT_m=T_mT_n$ and $T_nA_m=A_m$, that are required for a stability sequence, while not being required for a pre-stability sequence.

Let $\nu_n$ be the probability measure on $K$ given by the image of $\mu$ under the map $X\to K$, $x\mapsto \alpha(T_n, x)$.
Let $\nu_\infty$ be any weak${}^*$-cluster point of $\nu_n$.
We may assume that $\nu_n$ converges to $\nu_\infty$ in the weak${}^*$-topology.
Since $\mu(E_n)$ is uniformly positive, we have $\nu_\infty(\overline{V_{\ve_1}})>0$.
For every $\ve_0$ with $\ve_1<\ve_0<\ve$, we have $\nu_\infty(V_{\ve_0})>0$.
Since, as $\ve_0$ varies, the boundaries $\partial V_{\ve_0}$ are mutually disjoint, we may find some $\ve_0$ with $\ve_1<\ve_0<\ve$ such that $\nu_\infty(\partial V_{\ve_0})=0$.
We set $U\defeq V_{\ve_0}$ and set
\[D_n\defeq \{ \, x\in X\mid \alpha(T_n, x)\in U\, \}.\]
Since $E_n\subset D_n$, the set $D_n$ has uniformly positive measure.
The equation $T_nD_n=D_n$ follows from $T_n$ being involutive and $U^{-1}=U$: for all $x\in D_n$,
\[\alpha(T_n, T_nx)=\alpha(T_n^{-1}, T_nx)=\alpha(T_n, x)^{-1}\in U^{-1}=U\]
and hence $T_nx\in D_n$.
Using $\nu_\infty(\partial U)=0$, we verify the following:

\begin{claim}\label{claim-d-ai}
The sequence $(D_n)$ is asymptotically invariant for $\calr$.
\end{claim}

\begin{proof}
Pick $g\in [\calr]$ and $\ve >0$.
Since $\nu_\infty(\partial U)=0$, there exists an open ball $W$ in $K$ about the identity with $\nu_\infty(\overline{WU}\setminus U)<\ve$.
We claim that for all sufficiently large $n$,
\[\mu(\{ \, x\in X\mid \alpha(g, T_nx)\alpha(g, x)^{-1}\not\in W\, \})<\ve.\]
Indeed, choosing an open ball $W_1$ in $K$ centered at the identity with $W_1^2\subset W$ and finitely many elements $b_1,\ldots, b_M\in K$ with $K=\bigcup_{i=1}^MW_1b_i$, we set
\[Y_i\defeq \{ \, x\in X\mid \alpha(g, x)\in W_1b_i\, \}\]
for $i\in \{ 1,\ldots, M\}$.
The set $X$ is covered by the sets $Y_1,\ldots, Y_M$.
If $n$ is sufficiently large, then for all $i$, we have $\mu(T_nY_i\bigtriangleup Y_i)<\ve /M$ and for all $x\in T_n^{-1}Y_i\cap Y_i$,
\[\alpha(g, T_nx)\alpha(g, x)^{-1}\in W_1b_i(W_1b_i)^{-1}=W_1^2\subset W.\]
The claim follows.

If $n$ is sufficiently large, then for all $x\in D_n$ outside a subset of small measure, we have $gT_nx=T_ngx$ and $\alpha(g, T_nx)\alpha(g, x)^{-1}\in W$, and therefore
\begin{align*}
\alpha(T_n, gx)&=\alpha(T_ng, x)\alpha(g, x)^{-1}=\alpha(gT_n, x)\alpha(g, x)^{-1}\\
&=\alpha(g, T_nx)\alpha(T_n, x)\alpha(g, x)^{-1}\\
&=\alpha(g, T_nx)\alpha(g, x)^{-1}\alpha(g, x)\alpha(T_n, x)\alpha(g, x)^{-1}\\
&\in W\alpha(g, x)U\alpha(g, x)^{-1}=WU.
\end{align*}
Since $\limsup_n \nu_n(\overline{WU}\setminus U)\leq \nu_\infty(\overline{WU}\setminus U)$ and the measure $\nu_\infty(\overline{WU}\setminus U)$ is small, for all $x\in D_n$ outside a subset of small measure, we have $\alpha(T_n, gx)\in U$, i.e., $gx\in D_n$.
It turns out that the measure of $D_n\setminus g^{-1}D_n$ is small if $n$ is sufficiently large.
\end{proof}

By Lemma \ref{lem-patch}, there exists a pre-stability sequence $(S_m, Z_m)$ for $\calr$ such that for every $m$, the map $S_m$ is obtained by patching together pieces of the restrictions $T_n|_{E_n}$, $n\in \N$, along with a piece of the identity map, and $\mu(S_mZ_m\setminus Z_m)>1/10$.
The former condition implies that for all $x\in X$, $\alpha(S_m, x)$ belongs to $U$ and hence to $V_\ve$.
This proves the claim in the beginning of the proof of Theorem \ref{thm-main}.

To obtain the conclusion of Theorem \ref{thm-main}, we vary $\ve$.
Let $V_1\supset V_2\supset \cdots$ be a decreasing sequence of open balls in $K$ about the identity.
We then obtain a pre-stability sequence $(R_n, Y_n)$ for $\calr$ such that $\alpha(R_n, x)\in V_n$ for all $n$ and all $x\in X$.
We note that $\mu(R_nY_n\setminus Y_n)$ is uniformly positive because $\mu(S_mZ_m\setminus Z_m)$ is uniformly positive independently of $V_n$.
This completes the proof of Theorem \ref{thm-main}.
\end{proof}

We recall fundamental facts on compact extensions of equivalence relations.
Let $\calr$ be an ergodic discrete p.m.p.\ equivalence relation on $(X, \mu)$.
Let $\alpha \colon \calr \to K$ be a cocycle into a compact group $K$, and let $L$ be a closed subgroup of $K$.
We equip the space $X\times K/L$ with the product measure of $\mu$ and the $K$-invariant probability measure on $K/L$.
We define the equivalence relation $\calr_{\alpha, L}$ on $X\times K/L$ so that for all $(y, x)\in \calr$ and $k\in K$, $(x, kL)$ and $(y, \alpha(y, x)kL)$ are equivalent.
Then $\calr_{\alpha, L}$ is a discrete p.m.p.\ equivalence relation.
We call the equivalence relation $\calr_{\alpha, L}$ obtained through this procedure a \textit{compact extension} of $\calr$.
When $L$ is trivial, the extension $\calr_{\alpha, L}$ is simply denoted by $\calr_\alpha$.

For every cocycle $\alpha \colon \calr \to K$ into a compact group $K$, there exist a closed subgroup $K_0$ of $K$ and a cocycle $\alpha_0\colon \calr \to K$ equivalent to $\alpha$ such that values of $\alpha_0$ are in $K_0$ and there is no cocycle equivalent to $\alpha_0$ with values in a proper closed subgroup of $K_0$.
The subgroup $K_0$ is uniquely determined up to conjugacy in $K$, and it is called the Mackey range of the cocycle $\alpha$.
The extension $\calr_{\alpha_0}$ is then ergodic.
We refer to \cite[Corollary 3.8]{z} for details.

\begin{cor}\label{cor-ext-s}
Let $\calr$ be an ergodic discrete p.m.p.\ equivalence relation on $(X, \mu)$.
Let $\alpha \colon \calr \to K$ be a cocycle into a compact group $K$ such that there is no cocycle equivalent to $\alpha$ with values in a proper closed subgroup of $K$.
Let $L$ be a closed subgroup of $K$.
If $\calr$ is stable, then the extension $\calr_{\alpha, L}$ is stable.
\end{cor}

\begin{proof}
Pick a decreasing sequence $V_1\supset V_2\supset \cdots$ of open neighborhoods of the identity in $K$ with $\bigcap_n V_n=\{ e\}$.
By Theorem \ref{thm-main}, we may find a pre-stability sequence $(T_n, A_n)$ for $\calr$ such that $\alpha(T_n, x)\in  V_n$ for all $n$ and all $x\in X$.
For each $T\in [\calr]$, its lift $\tilde{T}\in [\calr_{\alpha, L}]$ is associated by the formula $\tilde{T}(x, kL)=(Tx, \alpha(T, x)kL)$ for $x\in X$ and $k\in K$.
Then $\tilde{T}_n$ asymptotically commutes with all elements of $[\calr_{\alpha, L}]$ that are the lift of an element of $[\calr]$.
Since $V_n$ approaches to the identity, the sequence $(\tilde{T}_n)$ satisfies $\tilde{\mu}(\tilde{T}_n A\bigtriangleup A)\to 0$ for every Borel subset $A\subset X\times K/L$, where $\tilde{\mu}$ is the measure on $X\times K/L$.
By Remark \ref{rem-central}, $(\tilde{T}_n)$ asymptotically commutes with all elements of $[\calr_{\alpha, L}]$. 
Thus the sequence $(\tilde{T}_n, A_n\times K/L)$ is a pre-stability sequence for $\calr_{\alpha, L}$.
\end{proof}

\begin{rem}\label{rem-erg-comp}
Let $\calr$ be an ergodic discrete p.m.p.\ equivalence relation on $(X, \mu)$ and let $\alpha \colon \calr \to K$ be a cocycle into a compact group $K$.
Even if $\alpha$ does not satisfy the assumption in Corollary \ref{cor-ext-s}, we can show that if $\calr$ is stable, then almost every ergodic component of $\calr_{\alpha, L}$ is stable.
Indeed, take a closed subgroup $K_0$ and a cocycle $\beta$ equivalent to $\alpha$ such that there is no cocycle equivalent to $\beta$ with values in a proper closed subgroup of $K_0$.
Since $\alpha$ and $\beta$ are equivalent, we have an isomorphism between $\calr_{\alpha, L}$ and $\calr_{\beta, L}$.
Let $\beta_0\colon \calr \to K_0$ be the cocycle obtained by simply replacing the range of the cocycle $\beta \colon \calr \to K$ into $K_0$.
Let $\theta \colon X\times K/L\to K_0\backslash K/L$ be the projection of the second coordinate.
Then the map $\theta$ gives rise to the decomposition into ergodic components of $\calr_{\beta, L}$.
For almost every $c\in K$, we have an isomorphism between the ergodic component $(\calr_{\beta, L})|_{\theta^{-1}(c)}$ and the extension $\calr_{\beta_0, K_0 \, \cap \, cLc^{-1}}$.
If $\calr$ is stable, then $\calr_{\beta_0, K_0 \, \cap \, cLc^{-1}}$ is stable by Corollary \ref{cor-ext-s}, and therefore almost every ergodic component of $\calr_{\alpha, L}$ is stable.
\end{rem}

\begin{cor}\label{cor-td}
Let $\calr$ be an ergodic discrete p.m.p.\ equivalence relation on $(X, \mu)$.
Let $\cals$ be a finite index subrelation of $\calr$.
If $\calr$ is stable, then every ergodic component of $\cals$ is stable.
\end{cor}

\begin{proof}
Let $N$ be the index of $\cals$ in $\calr$ and let $\tau \colon \calr \to \Sigma$ be the index cocycle defined in the proof of Corollary \ref{cor-finite-index-cpt-ext}, where $\Sigma$ is the symmetric group of $N$ letters.
Then $\cals$ is equal to the inverse image under $\tau$ of the subgroup $\Sigma_1\defeq \{ \, \sigma \in \Sigma \mid \sigma(1)=1\, \}$, and the restriction of $\calr_\tau$ to $X\times \Sigma_1$ is identified with $\cals$.
The corollary now follows from Remark \ref{rem-erg-comp}.
\end{proof}

We say that a free ergodic p.m.p.\ action of a countable group is \textit{stable} if the associated orbit equivalence relation is stable, and say that a countable group is \textit{stable} if it admits a free ergodic p.m.p.\ action which is stable.
Combining a result from \cite{kida-sce}, we obtain the following:

\begin{cor}\label{cor-fi}
Let $G$ be a countable group.
Then
\begin{enumerate}
\item[(i)] for every finite index subgroup $H$ of $G$, $H$ is stable if and only if $G$ is stable.
\item[(ii)] For every finite normal subgroup $N$ of $G$, $G/N$ is stable if and only if $G$ is stable.
\end{enumerate}
\end{cor}

\begin{proof}
Let $H$ be a finite index subgroup of $G$.
If $H$ has a free ergodic p.m.p.\ action which is stable, then the action of $G$ induced from it is stable.
Conversely, if $G$ is stable, then $H$ is stable by Corollary \ref{cor-td}.
Assertion (i) follows.

To prove assertion (ii), let $N$ be a finite normal subgroup of $G$.
If $G$ has a free ergodic p.m.p.\ action $G\c (X, \mu)$ which is stable, then the action of the quotient, $G/N\c X/N$, is also stable and hence $G/N$ is stable.

Conversely, suppose that $G/N$ is stable.
Let $H$ be the centralizer of $N$ in $G$, which is of finite index in $G$ since $N$ is finite.
The quotient $H/(H\cap N)$ is a finite index subgroup of $G/N$ and is hence stable by assertion (i).
Since $H\cap N$ is central in $H$, by \cite[Corollary 1.2]{kida-sce}, $H$ is stable.
By assertion (i) again, $G$ is stable.
\end{proof}


\subsection{Being Schmidt under compact extensions}\label{subsec-main2}

Following Theorem \ref{thm-main}, we obtain a similar result for existence of a non-trivial central sequence in the full group, in place of a pre-stability sequence, while the conclusion of the following theorem is slightly weaker than that of Theorem \ref{thm-main}.

\begin{thm}\label{thm-main-c}
Let $\calr$ be an ergodic discrete p.m.p.\ equivalence relation on $(X, \mu)$ and let $\alpha \colon \calr \to K$ be a cocycle into a compact group $K$.
If $\calr$ is Schmidt, then for every decreasing sequence $V_1\supset V_2\supset \cdots$ of open neighborhoods of the identity in $K$, there exists a non-trivial central sequence $(T_n)$ in $[\calr]$ such that $\mu(\{ \, x\in X\mid \alpha(T_n, x)\not\in V_n\, \})\to 0$.
\end{thm}

\begin{proof}
We basically follow the proof of Theorem \ref{thm-main}, while not a few modifications will be performed.
Fix a bi-invariant metric on $K$.
For $\ve >0$, let $V_\ve$ denote the open $\ve$-ball about the identity in $K$.
Pick $\ve >0$ and a finite subset $Q\subset [\calr]$.
To prove the theorem, it suffices to find an $S\in [\calr]$ such that $\mu(\{ \, x\in X\mid Sx=x\, \})<\ve$, $\mu(\{ gS\neq Sg\})<\ve$ for all $g\in Q$, and $\mu(\{ \, x\in X\mid \alpha(S, x)\not\in V_\ve \, \})<\ve$.

Let $(T_n)_{n\in \N}$ be a non-trivial central sequence in $[\calr]$.
Choose $0<\ve_2<\ve_1<\ve$ with $V_{\ve_2}^2\subset V_{\ve_1}$.
Since $K$ is compact, there are finitely many $c_1,\ldots, c_N\in K$ with $K=\bigcup_{i=1}^NV_{\ve_2}c_i$.
Passing to a subsequence of $(T_n)$, we can find an $i\in \{ 1,\ldots, N\}$ such that the set
\[C_n\defeq \{ \, x\in X\mid T_nx\neq x,\, \alpha(T_n, x)\in V_{\ve_2}c_i\, \}\]
has uniformly positive measure.
We put $c=c_i$.
By Lemma \ref{lem-pigeon}, passing to a subsequence, we may assume that for some $r>0$, we have $\mu(C_n\cap C_m)>r$ for all $n, m\in \N$.
We have the inclusion
\[T_n(C_n\cap C_m)\subset \{ \, x\in X\mid T_n^{-1}x\neq x,\, \alpha(T_mT_n^{-1}, x)\in V_{\ve_1}\, \},\]
as shown in the proof of Theorem \ref{thm-main}, and hence the set in the right hand side has measure more than $r$.
By Lemma \ref{lem-s}, if we fix $n$ and let $m$ be large, then the measure of the set $\{ \, x\in X\mid T_m^{-1}x=T_n^{-1}x\neq x\, \}$ converges to $0$.
Therefore for every $n$, for all sufficiently large $m>n$, the set
\[\{ \, x\in X\mid T_mT_n^{-1}x\neq x,\, \alpha(T_mT_n^{-1}, x)\in V_{\ve_1}\, \}\]
has measure more than $r/2$.
As a result, relabeling $T_mT_n^{-1}$ as $S_n$, we obtain a central sequence $(S_n)_{n\in \N}$ in $[\calr]$ such that the set
\[E_n\defeq \{ \, x\in X\mid S_nx\neq x,\, \alpha(S_n, x)\in V_{\ve_1}\, \}\]
has uniformly positive measure.

For each $n$, let $\mu_n$ be the restriction of $\mu$ to the set $\{ \, x\in X\mid S_nx\neq x\, \}$, whose total measure is uniformly positive.
Let $\nu_n$ be the measure on $K$ given by the image of $\mu_n$ under the map $x\mapsto \alpha(S_n, x)$.
Let $\nu_\infty$ be any weak${}^*$-cluster point of $\nu_n$.
We may assume that $\nu_n$ converges to $\nu_\infty$ in the weak${}^*$-topology.
Since $\mu(E_n)$ is uniformly positive, we have $\nu_\infty(\overline{V_{\ve_1}})>0$.
For every $\ve_0$ with $\ve_1<\ve_0<\ve$, we have $\nu_\infty(V_{\ve_0})>0$.
Since, as $\ve_0$ varies, the boundaries $\partial V_{\ve_0}$ are mutually disjoint, we may find some $\ve_0$ with $\ve_1<\ve_0<\ve$ and $\nu_\infty(\partial V_{\ve_0})=0$.
We set $U\defeq V_{\ve_0}$ and set
\[D_n\defeq \{ \, x\in X\mid S_nx\neq x,\ \alpha(S_n, x)\in U\, \}.\]
Since $E_n\subset D_n$, the set $D_n$ has uniformly positive measure.
Using the sequence of the sets $\{ \, x\in X\mid S_nx\neq x\, \}$ being asymptotically invariant for $\calr$ and following the proof of Claim \ref{claim-d-ai}, we can verify that the sequence $(D_n)$ is asymptotically invariant for $\calr$.
If $\mu(S_nD_n\bigtriangleup D_n) \not\to 0$, then we obtain a pre-stability sequence for $\calr$ as some subsequence of $(S_n, D_n)$.
The proof then reduces to Theorem \ref{thm-main}.
Otherwise, i.e., if $\mu(S_nD_n\bigtriangleup D_n)\to 0$, then by Lemma \ref{lem-patch-c}, we can find an $S\in [\calr]$ such that a large part of the map $S$ is obtained by patching together pieces of the restrictions $S_n|_{D_n}$, $n\in \N$, and $\mu(\{ gS\neq Sg\})<\ve$ for all $g\in Q$.
The former condition says that for all $x\in X$ outside a subset of measure less than $\ve$, there is an $n$ such that $x\in D_n$ and $Sx=S_nx$.
For all such $x\in X$, we have $Sx\neq x$ and $\alpha(S, x)\in U\subset V_\ve$.
\end{proof}

\begin{cor}\label{cor-ext-c}
Let $\calr$ be an ergodic discrete p.m.p.\ equivalence relation on $(X, \mu)$.
Let $\alpha \colon \calr \to K$ be a cocycle into a compact group $K$ such that there is no cocycle equivalent to $\alpha$ with values in a proper closed subgroup of $K$.
Let $L$ be a closed subgroup of $K$.
If $\calr$ is Schmidt, then the extension $\calr_{\alpha, L}$ is Schmidt.
\end{cor}

\begin{proof}
Pick a decreasing sequence $V_1\supset V_2\supset \cdots$ of open neighborhoods of the identity in $K$ with $\bigcap_n V_n =\{ e\}$.
By Theorem \ref{thm-main-c}, we may find a non-trivial central sequence $(T_n)$ in $[\calr]$ such that $\mu(\{ \, x\in X\mid \alpha(T_n, x)\not\in V_n\, \})\to 0$.
As shown in the proof of Corollary \ref{cor-ext-s}, to each $T_n$, the lift $\tilde{T}_n\in [\calr_{\alpha, L}]$ is associated, and the sequence $(\tilde{T}_n)$ is then a non-trivial central sequence in $[\calr_{\alpha, L}]$.
\end{proof}

The proof of Corollary \ref{cor-td} also works to obtain the following:

\begin{cor}\label{cor-td-c}
Let $\calr$ be an ergodic discrete p.m.p.\ equivalence relation on $(X, \mu)$ and let $\cals$ be a finite index subrelation of $\calr$.
If $\calr$ is Schmidt, then every ergodic component of $\cals$ is Schmidt.
\end{cor}

We may ask whether the same result as Corollary \ref{cor-fi} holds for the Schmidt property of groups in place of stability.
Corollary \ref{cor-td-c} immediately implies that the Schmidt property passes from a group to each of its finite index subgroups.
However the following remains unsolved:

\begin{question}\label{q-finite-central}
Let $G$ be a countable group with a finite central subgroup $C$.
If $G/C$ has the Schmidt property, then does $G$ have the Schmidt property as well?
\end{question}

We note that if $C$ is an infinite central subgroup of $G$, then $G$ has the Schmidt property regardless of whether $G/C$ has the Schmidt property (see Example \ref{ex-infinite-center}).



\section{Results under spectral-gap assumptions}\label{sec-sg}

\subsection{Stable spectral gap}\label{subsec-ssg}

Recall that a unitary representation of a countable group $G$ has \textit{spectral gap} if it does not contain the trivial representation of $G$ weakly.
We say that a p.m.p.\ action $G\c (X, \mu)$ has {\it spectral gap} if the Koopman representation $G\c L^2_0(X)$ has spectral gap, and say that a p.m.p.\ action $G\c (X, \mu)$ has {\it stable spectral gap} if for every unitary representation $G\c \cal{H}$, the product representation $G\c L^2_0(X)\otimes \cal{H}$ has spectral gap.
If a free p.m.p.\ action $G\c (X, \mu)$ has stable spectral gap and the associated orbit equivalence relation is inner amenable, then we can find the following remarkable F\o lner sequence for the conjugating action of $G$:

\begin{prop}\label{prop-iasgap}
Let $G\c (X, \mu)$ be a free p.m.p.\ action with stable spectral gap, and let $\mathcal{R}$ be the associated orbit equivalence relation.
If $\mathcal{R}$ is inner amenable, then there exists a sequence $(F_n)_{n\in \N}$ of finite subsets of $G$ such that
\begin{enumerate}
\item[(i)] $|F_n^g \bigtriangleup F_n |/|F_n| \to 0$ for every $g\in G$,
\item[(ii)] $1_{F_n}(g) \to 0$ for every $g\in G$, and
\item[(iii)] $\sup _{g\in F_n}\mu (gA\bigtriangleup A) \to 0$ for every Borel subset $A\subset X$.
\end{enumerate}
Thus, if $(g_n)_{n\in \N}$ is a sequence in $G$ with $g_n\in F_n$ for all $n\in \N$, then $g_n$ converges to the identity in $\mathrm{Aut}(X,\mu)$.
In particular, the image of $G$ in $\mathrm{Aut}(X, \mu)$ is not discrete.
\end{prop}

Before the proof, we prepare the following:

\begin{lem}\label{lem-specgap}
Let $G\c (X, \mu)$ be a p.m.p.\ action with stable spectral gap.
We identify each $g\in G$ with the element of $[G\ltimes (X, \mu)]$ given by the section $\{ g\} \times X$.
Then for every $\ve >0$, there exist a finite subset $S\subset G$ and $\delta >0$ such that if $\xi \in L^1(G\times X)$ is any non-negative unit vector satisfying $\sup _{g\in S} \Vert \xi^{g} - \xi \Vert_1 < \delta$, then $\Vert \xi  -  P\xi \Vert _1 < \ve$, where $P\colon L^1(G\times X)\to L^1(G\times X)$ is the projection defined by integrating functions along $X$:
\[(P\xi)(g, x) \defeq \int_X\xi(g, t) \, d\mu(t)\]
for $\xi \in L^1(G\times X)$, $g\in G$ and $x\in X$.
\end{lem}

\begin{proof}
Let $G$ act on $G\times X$ by $g (h,x)=(h,x)^{g^{-1}} = (ghg^{-1},gx )$.
This action gives rise to the unitary representation $\pi \colon G\c L^2(G\times X)$ identified with the tensor product of the conjugation representation $G\c \ell ^2(G)$ and the Koopman representation $G\c L^2(X)$.
The projection $P$ is also defined on $L^2(G\times X)$ by the same formula, and it is exactly the orthogonal projection onto the subspace $\ell^2(G)\otimes \C 1$.

Let $\ve >0$ and choose $\ve _0>0$ so that $2\ve_0 ^{1/2}+\ve_0 <\ve$.
Since the representation $\pi \colon G\c \ell^2(G)\otimes L^2_0 (X)$ has spectral gap, there exist a finite subset $S\subset G$ and $\delta >0$ such that if $\eta \in L^2(G\times X)$ is any unit vector satisfying $\sup _{g\in S} \Vert \pi (g)\eta - \eta \Vert_2^2 <\delta$, then $\Vert \eta - P\eta \Vert _2 ^2 < \ve_0$.
Let $\xi \in L^1(G\times X)$ be any non-negative unit vector satisfying $\sup _{g\in S} \Vert \xi ^{g} - \xi \Vert _1 < \delta$, and let $\eta \defeq  \xi ^{1/2}$.
Then $\eta$ is a unit vector in $L^2(G\times X)$, and
\[\sup _{g\in S} \Vert \pi (g)\eta - \eta \Vert _2^2 \leq \sup _{g\in S} \Vert \xi ^{g} - \xi \Vert _1 < \delta,\]
where the first inequality follows from the inequality $|a-b|^2\leq |a^2-b^2|$ for all $a, b\geq 0$.
By our choice of $\delta$, we then have $\Vert \eta - P\eta \Vert _2 ^2 < \ve _0$.
It follows that $\Vert P\eta \Vert _2 ^2 > 1-\ve _0$ and hence
\begin{align*}
\Vert P\xi - (P\eta )^2 \Vert _1 &= \sum _{g\in G} \Biggl( \int_X \xi(g, x) \, d\mu(x) - \biggl(\int _X \xi(g, x)^{1/2} \, d\mu(x) \biggr)^2\Biggl)\notag \\
&= 1- \Vert P\eta \Vert _2 ^2 <\ve _0,\notag
\end{align*}
where we use Jensen's inequality in the first equation.
By the Cauchy-Schwarz inequality, we have
\begin{align*}
\Vert \xi - (P\eta ) ^2 \Vert _1 &= \Vert (\eta + P\eta ) (\eta - P\eta ) \Vert _1  \leq \Vert \eta + P\eta \Vert _2 \Vert \eta - P\eta \Vert _2 \leq 2\ve _0^{1/2},
\end{align*}
and therefore $\Vert \xi - P\xi \Vert _1 < 2 \ve _0^{1/2} + \ve _0 < \ve$.
\end{proof}

\begin{proof}[Proof of Proposition \ref{prop-iasgap}]
Let $\alpha \colon \mathcal{R}\to G$ be the cocycle defined by the equation $\alpha (y,x)x = y$.
Let $(\xi _n )_{n\in \N} $ be an inner amenability sequence for $\mathcal{R}$.
For each $n\in \N$, we define a function $p_n \in \ell ^1(G)$ by $p_n (g) \defeq \int _X\xi _n (gx,x) \, d\mu(x)$.
Then $p_n$ is a probability measure on $G$, and by Lemma \ref{lem-specgap}, we have $\Vert \xi_n-p_n\circ \alpha \Vert_1\to 0$.
Therefore, for every $g\in G$, we have $\Vert p_n^g - p_n \Vert _1 \to 0$, and $p_n (g) \to 0$.
For every Borel subset $A\subset X$, we have
\begin{align*}
\sum _{g\in G}p_n(g)\mu (g^{-1}A\cap A ) &= \int _{\mathcal{R}_A}p_n (\alpha (y,x)) \, d\mu ^1 (y,x)  \\
&= \int _{\mathcal{R}_A}\xi _n \, d\mu ^1 + \int _{\mathcal{R}_A}(p_n \circ \alpha  - \xi _n )\, d\mu ^1 \to \mu (A),
\end{align*}
where the last convergence follows because $(\xi _n)$ is balanced.
Therefore, for every $\ve >0$, we have $p_n (D_{A,\ve})\to 1$, where we set
\[D_{A,\ve} \defeq \{ \, g \in G \mid \mu (gA\bigtriangleup A ) < \ve \, \}.\]

Let $\{ A_i \}_{i\in \N}$ be a countable collection of sets which are dense in the measure algebra of $\mu$.
After passing to a subsequence of $(p_n)_{n\in \N}$, we may assume without loss of generality that $p_n (\bigcap _{i<n}D_{A_i,1/n}) > 1-1/n$ for every $n\in \N$.
Let $Q_1\subset Q_2\subset \cdots$ be an increasing exhaustion of $G$ by finite subsets.
Since $p_n(g)\to 0$ for every $g\in G$, after passing to a further subsequence, we can assume without loss of generality that $p_n (Q_n) < 1/n$ for every $n\in \N$.
Let $q_n$ be the normalized restriction of $p_n$ to the set $(\bigcap _{i<n}D_{A_i, 1/n}) \setminus Q_n$.
Then $\Vert q_n - p_n \Vert _1 \to 0$, so $\Vert q_n^g - q_n \Vert _1 \to 0$ for every $g\in G$.
Therefore, by the Namioka trick, we may find a sequence of finite subsets $F_n \subset (\bigcap _{i<n}D_{A_i,1/n})\setminus Q_n$ such that condition (i) of the proposition holds.
Condition (ii) holds since $F_n\cap  Q_n=\emptyset$.
Given any Borel subset $A\subset X$ and $\ve >0$, we can find some $i\in \N$ with $\mu (A\bigtriangleup A_i)<\ve /3$.
Then for all $n>i$ with $1/n <\ve /3$, for all $g\in F_n\subset  D_{A_i,1/n}$, we have $\mu (gA\bigtriangleup A ) \leq \mu (gA_i\bigtriangleup A_i ) +2\mu (A\bigtriangleup A_i ) <\ve$.
This shows that condition (iii) holds.
\end{proof}

\begin{cor}\label{cor-ber}
Every Bernoulli shift of a countable non-amenable group gives rise to an orbit equivalence relation which is not inner amenable.
\end{cor}

\begin{proof}
Let $G\c (X,\mu )$ be any Bernoulli shift of a non-amenable group $G$.
This action is mixing, so the image of $G$ in $\mathrm{Aut}(X,\mu )$ is discrete.
The action has stable spectral gap since $G$ is non-amenable (see Remark \ref{rem-ssg} below).
Thus the corollary follows from Proposition \ref{prop-iasgap}.
\end{proof}

\begin{rem}\label{rem-ssg}
It is widely known that the Bernoulli shift $G\c (X, \mu)$ of a non-amenable group $G$ has stable spectral gap.
We give a proof of this fact for the reader's convenience, as follows:
By \cite[Lemma 1]{j}, the Koopman representation $G\c L^2_0(X)$ is a direct sum of subrepresentations of the left regular representation $\lambda \colon G\c \ell^2(G)$.
Take an arbitrary unitary representation $G\c \calh$.
By Fell's absorption principle \cite[Corollary 1 to Lemma 4.2]{fell} (see \cite[Lemma 2.1]{ch} for a direct proof), the product representation $G\c L^2_0(X)\otimes \calh$ is equivalent to a subrepresentation of the direct sum of countably many copies of $\lambda$.
Thus it does not contain the trivial representation of $G$ weakly since $G$ is non-amenable.  
\end{rem}

We will use the following lemma and corollary, which impose constraints on central sequences in a full group, in constructing interesting examples in Sections \ref{sec-finite-index} and \ref{sec-ex}.

\begin{lem}\label{lem-specgap-product}
Let $G\c (X, \mu)$ and $G\c (Y, \nu)$ be p.m.p.\ actions and suppose that the action $G\c (X, \mu)$ has stable spectral gap.
We set $(Z, \zeta)\defeq (X\times Y, \mu \times \nu)$ and let $G$ act on $(Z, \zeta)$ diagonally.
For each $g\in G$, let $\phi_g\in [G\ltimes (Z, \zeta)]$ denote the section $\{ g\} \times Z$.
Then for every $\ve >0$, there exist a finite subset $S\subset G$ and $\delta >0$ such that if $\phi$ is any element of $[G\ltimes (Z, \zeta)]$ satisfying
\begin{equation}\label{eqn-phicomm}
\inf_{g\in S}\zeta(\{ \, z\in Z\mid (\phi  \phi_g)_z=(\phi_g \phi)_z\, \})>1-\delta,
\end{equation}
then there exist a Borel subset $Y'\subset Y$, its partition $Y'=\bigsqcup_{i=1}^m Y_i$ into finitely many Borel subsets, and $g_1,\ldots, g_m\in G$ such that $\nu(Y')>1-\ve$ and for every $y\in Y_i$, we have
\[\mu(\{ \, x\in X\mid \phi_{(x, y)}=(g_i, (x, y))\, \})>1-\ve.\]
\end{lem}

\begin{proof}
We may assume $\ve < 1/2$.
Let $\pi \colon G\c \ell^2(G)\otimes L^2_0 (X)\otimes L^2(Y)$ be the subrepresentation of the tensor product of the conjugation representation $G\c \ell ^2(G)$ with the Koopman representation $G\c L^2(Z)$.
Let $P\colon L^2(G\times Z)\to \ell^2(G)\otimes \C 1\otimes L^2(Y)$ be the orthogonal projection.
Since $\pi$ has spectral gap, there exist a finite subset $S\subset G$ and $\delta >0$ such that if $\eta \in L^2(G\times Z)$ is any unit vector satisfying $\inf _{g\in S} \mathrm{Re}\langle \pi (g) \eta , \eta \rangle >1-\delta$, then $\Vert \eta - P\eta \Vert _2 ^2 < \ve^2$.

Assume that $\phi \in [G\ltimes (Z, \zeta)]$ satisfies condition \eqref{eqn-phicomm}.
Then the indicator function $1_\phi$ of $\phi \subset G\ltimes Z$ is a unit vector in $L^2(G\times Z)$, and for every $g\in S$, we have
\[
\langle \pi (g)1_{\phi} , 1_{\phi} \rangle = \zeta (\{ \, z\in Z \mid (\phi  \phi _g )_z = (\phi _g  \phi )_z \, \} ) >1-\delta .
\]
By our choice of $S$ and $\delta$, we have $\Vert 1_{\phi} - P(1_{\phi}) \Vert _2 ^2 < \ve^2$ and hence $\Vert P(1_{\phi} ) \Vert _2^2 > 1-\ve^2$.
For $g\in G$ and $y\in Y$, we set $A_{g, y} \defeq \{ \, x\in X \mid \phi_{(x, y)} = (g, (x, y)) \, \}$, so that $P(1_{\phi})(g, (x, y)) = \mu (A_{g, y})$ for every $g\in G$ and almost every $(x, y)\in X\times Y$.
Then
\begin{align*}
1-\ve^2 & < \Vert P(1_{\phi})\Vert _2^2 = \int_Y \sum _{g\in G}\mu (A_{g, y})^2 \, d\nu(y)\leq \int_Y \biggl( \sup _{g\in  G}\mu (A_{g, y}) \biggr) \sum_{g\in G}\mu(A_{g, y})\, d\nu(y)\\
& = \int_Y  \sup_{g\in G} \mu (A_{g, y}) \, d\nu(y),
\end{align*}
and therefore there exists a Borel subset $Y'\subset Y$ such that $\nu(Y')>1-\ve$ and for almost every $y\in Y'$, we have $\sup_{g\in G}\mu(A_{g, y})>1-\ve$.
Since $\ve <1/2$, for almost every $y\in Y'$, there exists a unique $g\in G$ such that $\mu(A_{g, y})>1-\ve$.
If $Y'$ is replaced with its slightly smaller subset, then there exist finitely many $g_1,\ldots, g_m\in G$ and a Borel partition $Y'=\bigsqcup_{i=1}^mY_i$ such that for almost every $y\in Y_i$, we have $\mu(A_{g_i, y})>1-\ve$.
\end{proof}

\begin{cor}\label{cor:specgap}
Let $G\c (X, \mu)$ be a p.m.p.\ action with stable spectral gap.
Then for every $\ve >0$ and every finite subset $F\subset G$, there exist a finite subset $S\subset G$ and $\delta >0$ such that if $\phi$ is any element of $[G\ltimes (X, \mu)]$ satisfying
\[ 
\inf _{g\in S} \mu ( \{ \, x\in X \mid (\phi  \phi _g)_x = (\phi _g \phi )_x \, \} ) > 1-\delta,
\]
then there exists an element $g_0\in G$ which commutes with all elements of $F$ and satisfies $\mu (\{ \, x\in X \mid \phi_x=(g_0, x) \, \} ) > 1-\ve$.
\end{cor}

\begin{proof}
Pick $0<\ve <1/2$ and a finite subset $F\subset G$.
In the proof of Lemma \ref{lem-specgap-product}, suppose that $Y$ is a singleton.
We may assume that $\delta <1/4$ for the number $\delta$ obtained from the assumption that the action $G\c (X, \mu)$ has stable spectral gap.
We may also assume that the obtained finite subset $S\subset G$ contains $F$.
Following the proof of Lemma \ref{lem-specgap-product}, we obtain $1-\ve^2 < \Vert P(1_{\phi})\Vert _2^2\leq \sup _{g\in G}\mu (A_g)$, where $A_g \defeq \{ \, x\in X \mid \phi_x = (g, x) \, \}$ for $g\in G$.
Hence there is some $g_0\in G$ with $\mu (A_{g_0})>1-\ve^2$.
It remains to show that $g_0$ commutes with all elements of $S$ (and hence of $F$).
Fix $g\in S$.
Since $\mu (A_{g_0})>1-\ve^2 > 3/4$ and $\mu ( \{ \, x\in X \mid (\phi  \phi _g)_x = (\phi _g  \phi )_x \, \} ) > 1-\delta > 3/4$, the set
\[
g^{-1}A_{g_0}\cap A_{g_0}\cap \{ \, x\in X \mid (\phi  \phi _g) _x = (\phi _g  \phi )_x \, \}
\]
is non-null, so fix some element $x$ of this set.
Then $(g_0g, x) = (\phi \phi _g)_x = (\phi _g \phi )_x = (gg_0, x)$, and thus $g_0$ commutes with $g$.
\end{proof}


\subsection{Product actions}

In this subsection, we show that if a free p.m.p.\ action $G\c (X, \mu)$ satisfies a certain spectral gap property and a mixing property, then its product with an arbitrary ergodic p.m.p.\ action $G\c (Y, \nu)$ gives rise to an orbit equivalence relation which is not inner amenable, and moreover the associated von Neumann algebra does not have property Gamma if the action $G\c (Y, \nu)$ is strongly ergodic.
For the Bernoulli shift of a non-amenable group, Ioana proves that the associated von Neumann algebra does not have property Gamma (\cite[Lemma 2.3]{i}).

\begin{prop}\label{prop-diagonal-action}
Let $G\c (X, \mu)$ be a free, p.m.p., mildly mixing action of an infinite countable group $G$.
Suppose that either
\begin{enumerate}
\item[(1)] the action $G\c (X, \mu)$ has stable spectral gap, or
\item[(2)] there is an infinite subgroup $H$ of $G$ such that the pair $(G, H)$ has property (T).
\end{enumerate}
Let $G\c (Y, \nu)$ be an ergodic p.m.p.\ action and let $G$ act on $(X\times Y, \mu \times \nu)$ diagonally.
Then the translation groupoid $G\ltimes (X\times Y, \mu \times \nu)$ is not inner amenable.

If the action $G\c (Y, \nu)$ is further strongly ergodic, then the von Neumann algebra associated to the action $G\c (X\times Y, \mu \times \nu)$ does not have property Gamma.
\end{prop}

Recall that a p.m.p.\ action $G\c (X, \mu)$ is called \textit{mildly mixing} if for every Borel subset $A\subset X$ with $0<\mu(A)<1$, we have $\liminf_{g\to \infty}\mu(gA\bigtriangleup A)>0$ (\cite{sch-mild}).
Every mildly mixing action is weakly mixing, and hence the diagonal action $G\c (X\times Y, \mu \times \nu)$ in Proposition \ref{prop-diagonal-action} is ergodic.
Without assuming the mildly mixing condition, the conclusion of the proposition may fail (see Remark \ref{rem-product-group-action}).

We will actually prove Proposition \ref{prop-diagonal-action} in a slightly more general setting.

\begin{ass}\label{assum}
Let $G\c (X, \mu)$ be a free p.m.p.\ action of an infinite countable group $G$ satisfying the following condition:
\begin{enumerate}
\item[(A)] There exist a finite subset $K\subset G$ and a Borel subset $D\subset X$ such that
\[\inf_{g\in G\setminus K}\mu(gD\bigtriangleup D)>0.\]
\end{enumerate}
Let $G\c (Y, \nu)$ be a p.m.p.\ action and set $(Z, \zeta)\defeq (X\times Y, \mu \times \nu)$.
Let $G$ act on $(Z, \zeta)$ diagonally and suppose that the action $G \c (Z, \zeta)$ is ergodic.
Let $G$ act on $G\times Z$ by the formula $g(h, z)=(ghg^{-1}, gz)$.
We have the Koopman representation $\pi \colon G\c L^2(G\times Z)$ and suppose the following condition:
\begin{enumerate}
\item[(B)] The representation $G\c \ell^2(G)\otimes L^2_0(X)\otimes L^2(Y)$ given as the subrepresentation of $\pi$ has spectral gap.
\end{enumerate}
We fix notation.
Let $M$ be the von Neumann algebra associated to the action $G\c (Z, \zeta)$, with the faithful normal trace $\tau$.
Let $L^2(M)$ be the completion of $M$ with respect to the norm $\Vert x\Vert_2=\tau(x^*x)^{1/2}$, which is naturally identified with $\ell^2(G)\otimes L^2(X)\otimes L^2(Y)$.
Let $Q\colon L^2(M)\to \ell^2(\{ e\})\otimes \mathbb{C}1\otimes L^2(Y)$ be the orthogonal projection.
\end{ass}

\begin{rem}\label{rem-ass}
The diagonal action $G\c (X\times Y, \mu \times \nu)$ in Proposition \ref{prop-diagonal-action} satisfies conditions (A) and (B) in Assumption \ref{assum}.
Indeed, condition (A) holds since the action $G\c (X, \mu)$ is mildly mixing.
If condition (1) holds, then condition (B) obviously follows.
If condition (2) holds, then the restriction $H\c (X, \mu)$ is mildly mixing and hence weakly mixing (\cite[Section 2]{sch-mild}).
Since the representation $H\c L^2_0(X)$ is weakly mixing, there is no $H$-invariant unit vector in $\ell^2(G)\otimes L^2_0(X)\otimes L^2(Y)$, and condition (B) follows from property (T) of the pair $(G, H)$.
\end{rem}

The following is a remark due to Adrian Ioana on the first author's earlier note.

\begin{lem}[A. Ioana]\label{lem-p}
Under Assumption \ref{assum}, if $(\eta_n)$ is a sequence of unit vectors in $L^2(M)$ such that $\Vert \pi(g)\eta_n-\eta_n\Vert_2\to 0$ for every $g\in G$ and $\Vert \eta_n1_A-1_A\eta_n\Vert_2\to 0$ for every Borel subset $A\subset Z$, then $\Vert \eta_n-Q\eta_n\Vert_2\to 0$.
\end{lem}

Note that $\eta_n1_A$ is the vector in $L^2(M)$ obtained by multiplying $\eta_n$ by $1_A\in M$ from the right, and is identified with the pointwise product $1_{\calg_{Z, A}}\eta_n$ of the two functions $1_{\calg_{Z, A}}$ and $\eta_n$ on $G\times Z$, where $(\calg, \zeta)\defeq G\ltimes (Z, \zeta)$.
Similarly, the vector $1_A\eta_n\in L^2(M)$ is identified with the pointwise product $1_{\calg_{A, Z}}\eta_n$.

\begin{proof}[Proof of Lemma \ref{lem-p}]
We first find a Borel subset $E\subset X$ such that
\[\inf_{g\in G\setminus \{ e\}}\mu(gE\bigtriangleup E)>0.\]
Let $K_0$ be the subgroup of all $g\in G$ with $gD=D$, which is contained in $K$ and hence finite.
By condition (A), $c\defeq \inf_{g\in G\setminus K_0}\mu(gD\bigtriangleup D)$ is positive.
Since the action $G\c X$ is free, there exists a Borel subset $D_1\subset X\setminus D$ such that $\mu(D_1)<c/3$ and $\mu(gD_1\bigtriangleup D_1)>0$ for all $g\in K_0\setminus \{ e\}$.
The set $E\defeq D\cup D_1$ is a desired one.
Indeed, for every $g\in G\setminus K_0$, we have $\mu(gE\bigtriangleup E)\geq \mu(gD\bigtriangleup D)-2\mu(D_1)\geq c/3$, and for every $g\in K_0\setminus \{ e\}$, we have $\mu(gE\bigtriangleup E)=\mu(gD_1\bigtriangleup D_1)>0$.
We set $d\defeq \inf_{g\in G\setminus \{ e\}}\mu(gE\bigtriangleup E)>0$.

For $g\in G$, let $u_g$ be the unitary of $M$ associated to $g$.
The representation $\pi$ is given by $\pi(g)x=u_gxu_g^*$ for $x\in M$.
Let $P\colon L^2(M)\to \ell^2(G)\otimes \mathbb{C}1\otimes L^2(Y)$ be the orthogonal projection.
By condition (B), we have
\begin{align}\label{px}
\Vert \eta_n-P\eta_n\Vert_2\to 0.
\end{align}
For each $n$, write $P\eta_n=\sum_{g\in G}u_g(1\otimes b_{n, g})$, where $b_{n, g}\in L^2(Y)$.
Let $F\defeq E\times Y$ and let $1_F$ be the indicator function of $F$.
We have $\Vert \eta_n1_F-1_F\eta_n\Vert_2\to 0$, and hence condition (\ref{px}) implies that $\Vert P(\eta_n)1_F-1_FP(\eta_n)\Vert_2\to 0$.
We also have
\[\Vert P(\eta_n)1_F-1_FP(\eta_n)\Vert_2^2=\sum_{g\in G}\mu(E\bigtriangleup g^{-1}E)\Vert b_{n, g}\Vert_2^2\geq d\sum_{g\in G\setminus \{ e\}}\Vert b_{n, g}\Vert_2^2.\]
By the definition of $P$ and $Q$, it follows that $\sum_{g\in G\setminus \{ e\}}\Vert b_{n, g}\Vert_2^2=\Vert P\eta_n-Q\eta_n\Vert_2^2$ and hence $\Vert P\eta_n-Q\eta_n\Vert_2\to 0$.
By condition (\ref{px}) again, $\Vert \eta_n-Q\eta_n\Vert_2\to 0$.
\end{proof}

\begin{cor}\label{cor-c}
Under Assumption \ref{assum}, the following two assertions hold:
\begin{enumerate}
\item[(i)] The translation groupoid $G\ltimes (Z, \zeta)$ is not inner amenable.
\item[(ii)] (A. Ioana) If the action $G\c (Y, \nu)$ is strongly ergodic, then $M$ does not have property Gamma.
\end{enumerate}
\end{cor}

\begin{proof}
To prove assertion (i), suppose toward a contradiction that there exists an inner amenability sequence $(\xi_n)$ for the groupoid $G\ltimes (Z, \zeta)$.
Let $\eta_n\defeq \xi_n^{1/2}$.
Then $\eta_n$ is a unit vector in $L^2(G\times Z)$ and satisfies the assumption in Lemma \ref{lem-p}.
Indeed, for every $g\in G$, we have $\Vert \pi(g)\eta_n-\eta_n\Vert_2^2\leq \Vert (\xi_n)^{g^{-1}}-\xi_n\Vert_1\to 0$.
Set $(\calg, \zeta) \defeq G\ltimes (Z, \zeta)$.
For every Borel subset $A\subset Z$, we have $\Vert \eta_n1_A-1_A\eta_n\Vert_2^2\leq \Vert 1_{\calg_{Z, A}}\xi_n-1_{\calg_{A, Z}}\xi_n\Vert_1\to 0$ since $(\xi_n)$ is balanced.
By Lemma \ref{lem-p}, we have $\Vert \eta_n-Q\eta_n\Vert_2\to 0$.
The projection $Q$ is also defined on $L^1(G\times Z)$:
For $\xi \in L^1(G\times Z)$, $g\in G$, $x\in X$ and $y\in Y$, we set
\[(Q\xi)(g, x, y)\defeq 
\begin{cases}
\int_X\xi(e, t, y)\, d\mu(t) & \text{if} \ g=e,\\
0 & \text{if}\ g\neq e.
\end{cases}\]
As in the proof of Lemma \ref{lem-specgap}, by Jensen's inequality, we have
\begin{align*}
\Vert Q\xi_n-(Q\eta_n)^2\Vert_1&=\int_{X\times Y}\xi_n(e, x, y)\, d\mu(x) d\nu(y)-\Vert Q\eta_n\Vert_2^2\leq \Vert \eta_n\Vert_2^2-\Vert Q\eta_n\Vert_2^2 \to 0.
\end{align*}
By the Cauchy-Schwarz inequality, we have
\[\Vert \xi_n-(Q\eta_n)^2\Vert_1\leq \Vert \eta_n+Q\eta_n\Vert_2\Vert \eta_n-Q\eta_n\Vert_2\leq 2\Vert \eta_n-Q\eta_n\Vert_2\to 0.\]
Thus $\Vert \xi_n-Q\xi_n\Vert_1\to 0$, and $\xi_n$ will concentrate on the set $\{ e\} \times Z$.
This contradicts $(\xi_n)$ being diffuse.
Assertion (i) follows.

Suppose that the action $G\c (Y, \nu)$ is strongly ergodic.
If $M$ had property Gamma, then we would have a sequence $(u_n)$ of unitaries of $M$ such that $\tau(u_n)=0$ and $\Vert [x, u_n]\Vert_2\to 0$ for all $x\in M$.
Since $Q$ is $G$-equivariant, $(Q(u_n))$ is asymptotically $G$-invariant.
By strong ergodicity of the action $G\c (Y, \nu)$, we have $\Vert Q(u_n)\Vert_2=\Vert Q(u_n)-\tau(Q(u_n))\Vert_2\to 0$.
This contradicts Lemma \ref{lem-p}, and assertion (ii) follows.
\end{proof}

Proposition \ref{prop-diagonal-action} follows from Corollary \ref{cor-c} and Remark \ref{rem-ass}.


\section{Finite-index inclusions and central sequences}\label{sec-finite-index}

For a finite-index inclusion $\cals <\calr$ of ergodic discrete p.m.p.\ equivalence relations, in Corollaries \ref{cor-td} and \ref{cor-td-c}, we proved that if $\calr$ is stable or Schmidt, then so is $\cals$.
In this section, we discuss the converse.
In Subsection \ref{subsec-alg}, we give a sufficient condition for the converse to hold.
In Subsections \ref{subsec-counter-schmidt} and \ref{subsec-counter-stable}, we construct examples for which the converse does not hold.
Throughout this section, let $(X, \mu)$ be a standard probability space and let $\calb$ be the measure algebra of $\mu$.


\subsection{The action on the algebra of asymptotically invariant sequences}\label{subsec-alg}

Let $\cals<\calr$ be a finite-index inclusion of ergodic discrete p.m.p.\ equivalence relations on $(X, \mu)$.
By \cite[Theorem 2.11]{ha} or \cite[Theorem 2]{su}, we have an ergodic finite-index subrelation $\cals_0<\cals$ and a finite group $F$ acting on $\cals_0$ by automorphisms such that $\calr=\cals_0 \rtimes F$.
Under the assumption that $\cals$ is stable or Schmidt, since these properties pass to $\cals_0$, we may therefore assume that $\calr$ is written as $\calr =\cals \rtimes F$ for some finite group $F$ acting on $\cals$.

Fix a non-principal ultrafilter $\omega$ on $\N$ and form the ultraproduct $(\calb^\omega, \mu^\omega)$ of the measure algebra $(\calb, \mu)$.
The full group $[\cals]$ naturally acts on $\calb^\omega$, preserving $\mu^\omega$.
Let $\cala$ denote the fixed point algebra of this action.
The group $F$ also acts on $\calb^\omega$ and on $\cala$.

\begin{prop}\label{prop-f-faith}
Suppose that $F$ acts on $\cala$ faithfully.
Then
\begin{enumerate}
\item[(i)] if $\cals$ is stable, then $\calr$ is stable.
\item[(ii)] If $\cals$ is Schmidt, then $\calr$ is Schmidt.
\end{enumerate}
\end{prop}

\begin{proof}
We first find a non-zero $\bar{B}\in \cala$ such that $\alpha(\bar{B})\cap \bar{B}=0$ for all non-trivial $\alpha \in F$.
Such a $\bar{B}$ is obtained by applying the following repeatedly:
For every non-trivial $\alpha \in F$, if $\bar{C}\in \cala$ is non-zero, then there exists a non-zero $\bar{B}\in \cala$ such that $\bar{B}\subset \bar{C}$ and $\alpha(\bar{B})\cap \bar{B}=0$.
Although this is proved in \cite[Lemma 2.3]{ck}, we give a proof for completeness:
Otherwise we would have $\alpha (\bar{B})=\bar{B}$ for all $\bar{B}\in \cala$ with $\bar{B}\subset \bar{C}$, and since $\alpha$ acts on $\cala$ non-trivially, there exists a non-zero $\bar{D}\subset 1-\bar{C}$ such that $\alpha (\bar{D})\cap \bar{D}=0$.
Let $(C_n)_{n\in \N}$ and $(D_n)_{n\in \N}$ be sequences which represent $\bar{C}$ and $\bar{D}$, respectively.
By Lemma \ref{lem-ai}, there is a subsequence $(D_{k_n})$ of $(D_n)$ such that $\mu(D_n')$ is uniformly positive, where $D_n'\defeq C_n\cap D_{k_n}$.
Let $\bar{D}'\in \cala$ be represented by the sequence $(D_n')_{n\in \N}$.
Then $\bar{D}'$ is non-zero, but we have $\alpha(\bar{D}')=\bar{D}'$ since $\bar{D}'\subset \bar{C}$, and we have $\alpha(\bar{D}')\cap \bar{D}'=0$ since $\alpha (\bar{D})\cap \bar{D}=0$, a contradiction.

Let $(B_n)_{n\in \N}$ be a sequence representing $\bar{B}$ such that $\alpha(B_n)\cap B_n=\emptyset$ for all $n\in \N$ and all non-trivial $\alpha \in F$.
Take a decreasing sequence $\ve_n\searrow 0$ of positive numbers, a sequence $(E_k)_{k\in \N}$ of elements of $\calb$ which is dense in $\calb$, and a sequence $(g_k)_{k\in \N}$ of elements of $[\cals]$ which is dense in $[\cals]$.
Passing to a subsequence of $(B_n)$, we may assume that
\begin{enumerate}
\item[(1)] $\mu(g_k\alpha(B_n)\bigtriangleup \alpha(B_n))<\ve_n/|F|$ for every $k\leq n$ and every $\alpha \in F$.
\end{enumerate}
To prove assertion (i), assume that $\cals$ is stable, and let $(T_n, A_n)_{n\in \N}$ be a stability sequence for $\cals$.
Passing to a subsequence of $(T_n, A_n)$, we may assume that for every $n$, for all $k\leq n$ and all $\alpha \in F$, we have
\begin{enumerate}
\item[(2)] $\mu(T_nB_n\bigtriangleup B_n)<\ve_n/|F|$,
\item[(3)] $|\mu(A_n\cap B_n)-\mu(B_n)/2|<\ve_n$,
\item[(4)] $\mu(T_n(\alpha^{-1}(E_k\cap \alpha(B_n)))\bigtriangleup \alpha^{-1}(E_k\cap \alpha(B_n)))<\ve_n/|F|$, and
\item[(5)] $\mu(\{ \alpha^{-1}g_k\alpha T_n\neq T_n\alpha^{-1}g_k\alpha \})<\ve_n/|F|$.
\end{enumerate}
We define $S_n\in [\cals]$ so that for every $\alpha \in F$, we have $S_n=\alpha T_n\alpha^{-1}$ on $\alpha(B_n)$ outside the set $\alpha(B_n\setminus T_n^{-1}B_n)$, which has measure less than $\ve_n/(2|F|)$ by condition (2), and also have $S_n(\alpha(B_n))=\alpha(B_n)$, and moreover $S_n$ is the identity outside $\bigcup_{\alpha \in F}\alpha(B_n)$.
By condition (4), we have
\[\mu(S_n(E_k\cap \alpha(B_n))\bigtriangleup (E_k\cap \alpha(B_n)))<2\ve_n/|F|\quad \text{and}\quad \mu(S_nE_k\bigtriangleup E_k)<2\ve_n\]
for every $k\leq n$.
Therefore $\mu(S_nE\bigtriangleup E)\to 0$ for all $E\in \calb$.
By construction, we also have $\mu(\{ \alpha S_n\neq S_n\alpha \})<\ve_n$ for every $\alpha \in F$.

Fix $k\leq n$.
We claim that $\mu(\{ g_kS_n\neq S_ng_k\})<5\ve_n$.
This claim together with the facts proved in the last paragaraph implies that the sequence $(S_n)$ is central in $[\calr]$.
By the definition of $S_n$, for every $\alpha \in F$, we have $g_kS_n=g_k\alpha T_n\alpha^{-1}$ on $\alpha(B_n)$ outside
a subset of measure less than $\ve_n/|F|$.
By condition (5), outside a subset of measure less than $\ve_n/|F|$, we have $g_k\alpha T_n\alpha^{-1}=\alpha(\alpha^{-1}g_k\alpha)T_n\alpha^{-1}=\alpha T_n\alpha^{-1}g_k$, and the right hand side is equal to $S_ng_k$ on $\alpha(B_n)$ outside a subset of measure less than $2\ve_n/|F|$ by condition (1) and the definition of $S_n$.
As a result, $\mu(\{ g_kS_n\neq S_ng_k\} \cap \alpha(B_n))<4\ve_n/|F|$ for every $\alpha \in F$.
Our claim then follows from condition (1).

We set $C_n\defeq F(A_n\cap B_n)$.
The sequence $(A_n\cap B_n)$ is asymptotically invariant for $\cals$, and the sequence $(C_n)$ is thus asymptotically invariant for $\calr$.
Since $S_n$ preserves $\alpha(B_n)$ for each $\alpha \in F$, we have $S_n(A_n\cap B_n)\setminus (A_n\cap B_n)\subset S_nC_n\setminus C_n$.
Up to a subset of measure less than $\ve_n/|F|$, the left hand side is equal to $T_n(A_n\cap B_n)\setminus (A_n\cap B_n)$, which is equal to $T_n(A_n\cap B_n)$ because $T_nA_n$ is disjoint from $A_n$.
By condition (3), the measure of the set $T_n(A_n\cap B_n)$ is equal to $\mu(B_n)/2$ up to $\ve_n$.
Therefore $\mu(S_nC_n\setminus C_n)$ is uniformly positive, and $(S_n, C_n)$ is a pre-stability sequence for $\calr$.
Assertion (i) follows.

We prove assertion (ii) and assume that $\cals$ is Schmidt.
In the beginning of the proof of this proposition, we found a sequence $(B_n)_{n\in \N}$ of elements of $\calb$ representing a non-zero element of $\cala$ and satisfying $\alpha (B_n)\cap B_n=\emptyset$ for all $n$ and all non-trivial $\alpha \in F$.
Let $(T_n)$ be a central sequence in $[\cals]$ such that $T_nx\neq x$ for all $n$ and all $x\in X$.
Such a sequence exists by Lemma \ref{lem-c}.
Take a sequence $(E_k)_{k\in \N}$ of elements of $\calb$ which is dense in $\calb$.
Passing to a subsequence of $(T_n)$, we may assume that for every $n$, $T_n$ almost fixes the set $B_n$, and that for all $k\leq n$ and all $\alpha \in F$, $T_n$ almost fixes the set $\alpha^{-1}(E_k\cap \alpha(B_n))$.
Define $S_n\in [\cals]$ so that $S_n=\alpha T_n\alpha^{-1}$ on the set $\alpha(B_n)$ outside its some subset of small measure, $S_n$ preserves $\alpha(B_n)$ for each $\alpha \in F$, and $S_n$ is the identity outside $\bigcup_{\alpha \in F}\alpha(B_n)$.
In the same manner as in the proof of assertion (i), we can show that $(S_n)$ is central in $[\calr]$.
Since $S_n$ is equal to $T_n$ on most part of $B_n$, the measure of the set $\{ \, x\in X\mid S_nx\neq x\, \}$ is uniformly positive.
Assertion (ii) follows.
\end{proof}


\subsection{Being Schmidt is not preserved}\label{subsec-counter-schmidt}

The following proposition is a slight refinement of \cite[Theorem 29.10]{kec}:

\begin{prop}\label{prop-t}
Let $G$ be a countably infinite, discrete group with property (T) and let $G\c (X, \mu)$ a free ergodic p.m.p.\ action.
Suppose the following two conditions:
\begin{itemize}
\item For every non-trivial element of $G$, its conjugacy class in $G$ consists of at least two elements.
\item For every $g\in G$ whose conjugacy class in $G$ is finite, the centralizer of $g$ in $G$ acts on $(X, \mu)$ ergodically.
\end{itemize}
Then the action $G\c (X, \mu)$ is not Schmidt.
\end{prop}

\begin{proof}
Let $\calr \defeq \calr(G\c (X, \mu))$.
Suppose toward a contradiction that there exists a non-trivial central sequence $(T_n)$ in $[\calr]$.
By Lemma \ref{lem-c}, we may assume that $T_nx\neq x$ for all $n$ and all $x\in X$.
We set $G^*\defeq G\setminus \{ e\}$.
Let $G$ act on the space $G^*\times X$ by $g(h, x)=(ghg^{-1}, gx)$ for $g\in G$, $h\in G^*$ and $x\in X$.
We then have the unitary representation of $G$ on the Hilbert space $\mathcal{H}\defeq \ell^2(G^*)\otimes L^2(X)$.
Let $P\colon \mathcal{H}\to \mathcal{H}$ be the orthogonal projection onto the subspace of $G$-invariant vectors.
Let $\Omega$ denote the set of finite conjugacy classes in $G^*$.
Each vector $\xi \in \mathcal{H}$ is uniquely written as $\xi =\sum_{g\in G^*}\delta_g\otimes \xi_g$, where $\delta_g$ is the Dirac function on $g$ and $\xi_g\in L^2(X)$.
For $\omega \in \Omega$, we set $c_\omega(\xi)\defeq |\omega|^{-1}\sum_{h\in \omega}\mu(\xi_h)$ and identify it with the constant function on $X$ of that value.
By our second assumption, the equation $P\xi =\sum_{\omega \in \Omega}\sum_{g\in \omega}\delta_g\otimes c_\omega(\xi)$ holds.

\begin{claim}\label{claim-ineq}
Every element $T\in [\calr]$ with $Tx\neq x$ for all $x\in X$ gives rise to the unit vector of $\mathcal{H}$, $\xi_T=\sum_{g\in G^*}\delta_g\otimes \xi_{T, g}$, where $\xi_{T, g}$ is the indicator function of the set $\{ T=g\}$, and this vector satisfies the inequality $\Vert \xi_T-P\xi_T\Vert^2\geq \Vert P\xi_T\Vert^2$.
\end{claim}

\begin{proof}
For simplicity, we set $\xi \defeq \xi_T$ and $\xi_g\defeq \xi_{T, g}$.
For every $\omega \in \Omega$, we have $0\leq c_\omega(\xi)\leq |\omega|^{-1}\leq 1/2$ by the first assumption in Proposition \ref{prop-t}.
The desired inequality is obtained as follows:
\begin{align*}
\Vert \xi-P\xi\Vert^2&\geq \sum_{\omega \in \Omega}\sum_{g\in \omega}\Vert \xi_g-c_\omega(\xi)\Vert^2=\sum_{\omega \in \Omega}\sum_{g\in \omega}(\mu(\xi_g)-2c_\omega(\xi)\mu(\xi_g)+c_\omega(\xi)^2)\\
&=\sum_{\omega \in \Omega}|\omega|c_\omega(\xi)(1-c_\omega(\xi))\geq \sum_{\omega \in \Omega}|\omega|c_\omega(\xi)^2=\Vert P\xi \Vert^2.\qedhere
\end{align*}
\end{proof}

Let $\xi_n$ denote the unit vector to which the automorphism $T_n$ gives rise as in Claim \ref{claim-ineq}.
Since $(T_n)$ is a central sequence in $[\calr]$, the sequence $(\xi_n)$ is asymptotically $G$-invariant, that is, for every $g\in G$, we have $\Vert g\xi_n-\xi_n\Vert \to 0$ as $n\to \infty$.
It follows from property (T) of $G$ that $\Vert \xi_n-P\xi_n\Vert \to 0$.
This contradicts the inequality in Claim \ref{claim-ineq}.
\end{proof}

\begin{prop}\label{prop-t-c}
Let $G$ be a countable discrete group with property (T), and let $H$ be a finite index subgroup of $G$.
Let $C$ be an infinite central subgroup of $H$ which is normal in $G$.
Suppose that for every non-trivial element of $G$, its conjugacy class in $G$ either contains an element of $C$ and at least two elements, or is infinite.
Then there exists a free ergodic p.m.p.\ action $G\c (X, \mu)$ which is not Schmidt and such that the restriction $H\c (X, \mu)$ is ergodic and Schmidt.
\end{prop}

\begin{proof}
We embed $C$ into an abelian compact group $K$, equip $K$ with the normalized Haar measure, and let $C$ act on $K$ by translation.
Pick a section $s_0\colon H/C\to H$ of the quotient map.
Pick also representatives $g_1,\ldots, g_N$ of left cosets of $H$ in $G$, where $N$ is the index of $H$ in $G$.
We then have the section $s\colon G/C\to G$ defined by $s(g_ib)\defeq g_is_0(b)$ for $i\in \{ 1,\ldots, N\}$ and $b\in H/C$.
Let $G$ act on $X\defeq \prod_{G/C}K$ by the action co-induced from the action $C\c K$.
It is defined by $(gf)(b)=c^{-1}f(g^{-1}b)$ for $g\in G$, $f\in X$ and $b\in G/C$, where the element $c\in C$ is determined by the equation $s(g^{-1}b)c=g^{-1}s(b)$.
Let $\mu$ be the probability measure on $X$ given by the product of the Haar measure on $K$.
Since $H/C$ is infinite, the restriction $H\c (X, \mu)$ is ergodic.
Choose a sequence $(c_n)_{n\in \N}$ of non-trivial elements of $C$ such that for all $i\in \{ 1,\ldots, N\}$, the sequence $(g_i^{-1}c_ng_i)_{n\in \N}$ converges to the identity in $K$.
Then $(c_n)$ is central as a sequence in the full group $[\calr(H\c (X, \mu))]$.
Indeed, for all $b\in H/C$, $c\in C$ and $i\in \{ 1,\ldots, N\}$, we have
\[c^{-1}s(g_ib)=c^{-1}g_is_0(b)=g_ig_i^{-1}c^{-1}g_is_0(b)=g_is_0(b)g_i^{-1}c^{-1}g_i=s(g_ib)g_i^{-1}c^{-1}g_i,\]
where the third equation holds since $C$ is normal in $G$ and thus $g_i^{-1}c^{-1}g_i\in C$.
This implies that $c$ acts on the coordinate of $X$ whose index is $g_ib$ by adding $g_i^{-1}cg_i$.
Thus by Remark \ref{rem-central}, $(c_n)$ is central in the full group $[\calr(H\c (X, \mu))]$.
However, by Proposition \ref{prop-t}, the action $G\c (X, \mu)$ is not Schmidt.
\end{proof}

\begin{rem}\label{rem-h-g-schmidt}
With the notation of Proposition \ref{prop-t-c}, we set $\calr \defeq \calr(G\c (X, \mu))$, and set $M\defeq \bigcap_{g\in G}gHg^{-1}$, which is a finite index, normal subgroup of $G$.
Let $\alpha \colon \calr \to G/M$ be the cocycle defined by $\alpha(gx, x)=gM$ for $g\in G$ and $x\in X$.
Then the compact extension $\calr_{\alpha, H/M}$ (reviewed right before Corollary \ref{cor-ext-s}) is an equivalence relation on $X\times G/H$, and its restriction to the set $X\times \{ eH\}$ is identified with $\calr(H\c (X, \mu))$.
Proposition \ref{prop-t-c} says that $\calr_{\alpha, H/M}$ is Schmidt, while $\calr$ is not.
Thus, the converse of Corollary \ref{cor-ext-c} does not hold (modulo the existence of groups satisfying the assumption in Proposition \ref{prop-t-c}, which we will give in Example \ref{ex-matrix} below).

Since $\calr(H\c (X, \mu))$ is Schmidt and hence inner amenable, it follows from Proposition \ref{prop-finite-index} that $\calr$ is inner amenable.
Nevertheless, we can directly find an inner amenability sequence for $\calr$ from the proof of Proposition \ref{prop-t-c} as follows:
Choose the sequence $(c_n)_{n\in \N}$ in the proof so that the sequence $(g_i^{-1}g_jc_ng_j^{-1}g_i)_{n\in \N}$ also converges to the identity in $K$ for all $i, j\in \{ 1, \ldots, N\}$.
We set $F_n\defeq \{ \, g_jc_ng_j^{-1}\mid j=1,\ldots, N\, \}$ and define $f_n\in L^1(G\ltimes X, \mu^1)$ by $f_n(g, x)\defeq 1_{F_n}(g)/N$.
Then $f_n^g=f_n$ for all $g\in G$, and by the required convergence property of $(c_n)$, the sequence $(f_n)$ is balanced.
By Remark \ref{rem:suffice}, $(f_n)$ satisfies condition (ii) of Definition \ref{def:groupoid}.
Condition (iii) holds since the set $F_n$ diverges to infinity in $G$.
Condition (iv) follows from the definition of $f_n$.
\end{rem}

\begin{ex}\label{ex-matrix}
We give an example of groups $G$, $H$ satisfying the assumption in Proposition \ref{prop-t-c}.
We define $H$ as the subgroup of $\mathit{SL}_5(\Z)$ consisting of matrices of the form
\[\begin{pmatrix}
1 & \ast & \ast \\
0 & h & \ast \\
0 & 0 & 1
\end{pmatrix},\]
where $h$ runs through all elements of $\mathit{SL}_3(\Z)$.
This kind of group appears in \cite{cor}, and as mentioned in \cite[Proposition 2.7]{cor}, the group $H$ has property (T).
The center of $H$, denoted by $C$, consists of matrices whose diagonal entries are one and $(1, 5)$-entry is the only off-diagonal entry that is possibly non-zero.
We set
\[u\defeq \begin{pmatrix}
-1 & 0 & 0 \\
0 & I_3 & 0 \\
0 & 0 & 1
\end{pmatrix},\]
where $I_3$ is the $3$-by-$3$ identity matrix, and define $G$ as the subgroup of $\mathit{GL}_5(\Z)$ generated by $H$ and $u$.
Then $u$ normalizes $H$, and the index of $H$ in $G$ is $2$.

We claim that these groups $G$, $H$ satisfy the assumption of Proposition \ref{prop-t-c}, i.e., for every non-trivial element of $G$, its conjugacy class in $G$ either contains an element of $C$ and at least two elements, or is infinite.
For each non-trivial $c\in C$, we have $ucu^{-1}=c^{-1}$, and the conjugacy class of $c$ in $G$ consists of the two distinct elements, $c$ and $c^{-1}$.

We show that the other conjugacy classes except for $\{ e\}$ are infinite.
Pick $g, h\in H$ and write them as
\[g=\begin{pmatrix}
1 & g_{12} & g_{13} \\
0 & g_{22} & g_{23} \\
0 & 0 & 1
\end{pmatrix}, \qquad
h=\begin{pmatrix}
1 & h_{12} & h_{13} \\
0 & h_{22} & h_{23} \\
0 & 0 & 1
\end{pmatrix}.\]
We verify that the centralizer of $gu$ in $H$ is of infinite index in $H$.
This implies that the conjugacy class of $gu$ in $G$ is infinite.
Suppose toward a contradiction that the centralizer of $gu$ in $H$ is of finite index in $H$.
Then for every matrix $h$ such that $h_{22}$ belongs to some finite index subgroup of $\mathit{SL}_3(\Z)$ and the other $h_{ij}$ are zero, $gu$ and $h$ commute.
Since
\begin{align*}
guh&=\begin{pmatrix}
-1 & g_{12} & g_{13} \\
0 & g_{22} & g_{23} \\
0 & 0 & 1
\end{pmatrix}
\begin{pmatrix}
1 & 0 & 0 \\
0 & h_{22} & 0 \\
0 & 0 & 1
\end{pmatrix}
=
\begin{pmatrix}
-1 & g_{12}h_{22} & g_{13} \\
0 & g_{22}h_{22} & g_{23} \\
0 & 0 & 1
\end{pmatrix} \ \text{and}\\
hgu&=\begin{pmatrix}
1 & 0 & 0 \\
0 & h_{22} & 0 \\
0 & 0 & 1
\end{pmatrix}
\begin{pmatrix}
-1 & g_{12} & g_{13} \\
0 & g_{22} & g_{23} \\
0 & 0 & 1
\end{pmatrix}
=
\begin{pmatrix}
-1 & g_{12} & g_{13} \\
0 & h_{22}g_{22} & h_{22}g_{23} \\
0 & 0 & 1
\end{pmatrix},
\end{align*}
we have $g_{22}=I_3$, $g_{12}=0$, and $g_{23}=0$.
For every $h$ in some finite index subgroup of $C$, $gu$ and $h$ also commute.
The $(1, 5)$-entries of $guh$ and $hgu$ are $g_{13}-h_{13}$ and $g_{13}+h_{13}$, respectively.
This is impossible.
It turns out that the centralizer of $gu$ in $H$ is of infinite index in $H$.

Similarly, comparing $gh$ and $hg$ for the matrix $h$ such that $h_{22}$ belongs to some finite index subgroup of $\mathit{SL}_3(\Z)$ and the other $h_{ij}$ are zero, we can verify that if the centralizer of $g$ in $H$ is of finite index in $H$, then $g$ belongs to $C$.
Therefore the conjugacy class of every non-central element of $H$ in $G$ is infinite.
This completes the proof of $G$ and $H$ satisfying the assumption of Proposition \ref{prop-t-c}.
\end{ex}


\subsection{Stability is not preserved}\label{subsec-counter-stable}

We construct an ergodic p.m.p.\ equivalence relation $\calr$ and its ergodic finite-index subrelation $\cals$ such that $\cals$ is stable, but $\calr$ is not stable.
Let $X_0\defeq \prod_\N \Z/ 2\Z$ be the compact group equipped with the normalized Haar measure $\mu_0$.
Let $H_0\defeq \bigoplus_\N \Z /2\Z$ be the subgroup of $X_0$, let $H_0$ act on $X_0$ by translation, and set $\calr_0\defeq \calr(H_0\c (X_0, \mu_0))$.
Choose a non-amenable group $\Gamma$ arbitrarily.
We set $Y\defeq \prod_\Gamma X_0$ and equip $Y$ with the product measure of $\mu_0$, denoted by $\nu$.
Let $H_1\defeq \Gamma \times H_0$ act on $Y$ so that $\Gamma \times \{ e\}$ acts on $Y$ by the Bernoulli shift and $\{ e\} \times H_0$ acts on $Y$ diagonally.
We set $\calr_1\defeq \calr(H_1\c (Y, \nu))$, and set $\cals \defeq \calr_0\times \calr_1$, which is an equivalence relation on the product space $(Z, \zeta)\defeq (X_0\times Y, \mu_0\times \nu)$ and is also the orbit equivalence relation of the coordinatewise action $G\defeq H_0\times H_1\c (Z, \zeta)$.

We define an automorphism $\sigma$ of $(Z, \zeta)$ by
\[\sigma(x, (y_\gamma)_{\gamma \in \Gamma})\defeq (x, (xy_\gamma)_{\gamma \in \Gamma})\]
for $x\in X_0$ and $(y_\gamma)_\gamma \in Y$.
This $\sigma$ is involutive since every non-trivial element of $X_0$ has order $2$.
We claim that $\sigma$ is also an automorphism of $\cals$.
We first show that $\sigma$ is equivariant under the action of $H_1$, where $H_1$ is identified with the subgroup $\{ e\} \times H_1$ of $G$.
Indeed, for all $(x, (y_\gamma)_\gamma)\in Z$ and $(\delta, h)\in H_1=\Gamma \times H_0$, we have $(\delta, h)(y_\gamma)_\gamma =(hy_{\delta^{-1}\gamma})_\gamma$.
We also have $\sigma((\delta, h)(x, (y_\gamma)_\gamma))=(x, (xhy_{\delta^{-1}\gamma})_\gamma)=(\delta, h)\sigma(x, (y_\gamma)_\gamma)$, where the last equation holds since the group $X_0$ is commutative.
We therefore obtain the desired equivariance of $\sigma$ and hence $\sigma(\cali_0\times \calr_1)=\cali_0\times \calr_1$, where $\cali_0$ denotes the trivial equivalence relation on $(X_0, \mu_0)$.
We next show that $\sigma(\calr_0\times \cali_1)\subset \cals$, where $\cali_1$ denotes the trivial equivalence relation on $(Y, \nu)$.
For all $(x, (y_\gamma)_\gamma)\in Z$ and $h\in H_0$, we have $\sigma(x, (y_\gamma)_\gamma)=(x, (xy_\gamma)_\gamma)$, which is $\cals$-equivalent to $(hx, (hxy_\gamma)_\gamma)=\sigma(hx, (y_\gamma)_\gamma)$.
The claim follows.

Let $\calr \defeq \cals \rtimes \langle \sigma \rangle$ be the equivalence relation generated by $\cals$ and $\sigma$, which contains $\cals$ as a subrelation of index $2$.
The equivalence relation $\cals =\calr_0\times \calr_1$ is clearly stable.
We show that $\calr$ is not stable.
Otherwise, by Theorem \ref{thm-main}, there would exist a pre-stability sequence $(T_n, A_n)$ for $\calr$ with $T_n\in [\cals]$.
Since the action $G\c (Y, \nu)$ has stable spectral gap, the unitary representation $G\c L^2(X_0)\otimes L^2_0(Y)$ has spectral gap.
We may therefore assume that $A_n$ is of the form $A_n=\bar{A}_n\times Y$ for some Borel subset $\bar{A}_n\subset X_0$.
Furthermore, by Lemma \ref{lem-specgap-product}, passing to a subsequence of $(T_n, A_n)$, we may assume that for every $n$, there exist finitely many $g_1, \ldots, g_m\in G$, a Borel subset $X'\subset X_0$ and its Borel partition $X'=\bigsqcup_{i=1}^mX_i$ such that $\mu_0(X')>1-1/n$ and $\nu(A_{g_i, x})>1-1/n$ for all $x\in X_i$ and all $i\in \{ 1,\ldots, m\}$, where $A_{g, x}\defeq \{ \, y\in Y\mid T_n(x, y)=g(x, y)\, \}$ for $g\in G$ and $x\in X_0$. Passing to a subsequence of $(T_n, A_n)$, we may assume that $\zeta (W)>1-1/n$, where
\[W\defeq \{ \, z\in Z\mid (\sigma \circ T_n)(z)=(T_n\circ \sigma)(z)\, \}.\]
We set $A_g\defeq \{ \, z\in Z\mid T_n(z)=gz\, \}$ for $g\in G$.
Then for every $i\in \{ 1,\ldots, m\}$, we have $A_{g_i}\cap (X_i\times Y)=\bigsqcup_{x\in X_i}\{ x\} \times A_{g_i, x}$ and hence $\zeta(A_{g_i} \cap \sigma A_{g_i}\cap (X_i\times Y))>(1-2/n)\mu_0(X_i)$ since $\sigma$ preserves the set $\{ x\} \times Y$ for every $x\in X_0$.
Let $I$ be the set of all $i\in \{ 1,\ldots, m\}$ such that the set $W\cap A_{g_i}\cap \sigma A_{g_i} \cap (X_i\times Y)$ has positive measure.
Then
\[\zeta \bigl( \, \textstyle \bigsqcup_{i\in I}(A_{g_i} \cap \sigma A_{g_i}\cap (X_i\times Y))\bigr)>1-4/n.\]

If $i\in I$, then $g_i$ belongs to $\{ e\} \times H_1$.
Indeed, if we put $g_i= (h_i, (\gamma_i, h_i'))\in H_0\times H_1$ and pick a point $(x, (y_\gamma)_\gamma)$ from the set $W\cap A_{g_i} \cap \sigma A_{g_i}\cap (X_i\times Y)$, then
\[(\sigma\circ T_n)(x, (y_\gamma)_\gamma)=\sigma(h_ix, (h_i'y_{\gamma_i^{-1}\gamma})_\gamma )=(h_ix, (h_ixh_i'y_{\gamma_i^{-1}\gamma})_\gamma)\]
and
\[(T_n\circ \sigma)(x, (y_\gamma)_\gamma)=T_n(x, (xy_\gamma)_\gamma)=(h_ix, (h_i'xy_{\gamma_i^{-1}\gamma})_\gamma).\]
These two are equal, and therefore $h_i=e$.

Since $\{ e\} \times H_1$ acts on the first factor $X_0$ of $Z=X_0\times Y$ trivially and the set $A_n$ is of the form $A_n=\bar{A}_n\times Y$, the map $T_n$ sends $A_n\cap A_{g_i}$ into $A_n$ for every $i\in I$.
Hence we have $A_n\setminus T_n^{-1}A_n\subset Z\setminus \bigsqcup_{i\in I}(A_{g_i} \cap (X_i\times Y))$. 
The measure of the set in the right hand side is less than $4/n$, and this contradicts $(T_n, A_n)$ being a pre-stability sequence.

\begin{rem}
We note that $G$ and $\sigma$ generate the semi-direct product group $G\rtimes \langle \sigma \rangle$ such that $\sigma$ acts on $G$ by $\sigma(h, (\gamma, h'))=(h, (\gamma, hh'))$ for $h\in H_0$ and $(\gamma, h')\in H_1=\Gamma \times H_0$.
Then $\calr$ is the orbit equivalence relation generated by the action of $G\rtimes \langle \sigma \rangle$.
We also note that $\calr$ is not stable, but is Schmidt because the subgroup $\{ e\} \times (\{ e \} \times H_0)$ of $G$ is central in $G\rtimes \langle \sigma \rangle$, and therefore the elements $h_n\in H_0=\bigoplus_\N \Z /2\Z$ whose coordinate indexed by $n$ is $1$ and other coordinates are $0$ form a central sequence in $[\calr]$ by Remark \ref{rem-central}.
\end{rem}


\section{Miscellaneous examples}\label{sec-ex}

\subsection{Orbitally inner amenable groups}

Recall that a countable group is said to be orbitally inner amenable if it admits a free ergodic p.m.p.\ action whose orbit equivalence relation is inner amenable (Subsection \ref{subsec-schmidt}).
We provide several examples of such groups and free ergodic p.m.p.\ actions whose orbit equivalence relations are inner amenable, but not Schmidt (see also Remark \ref{rem-h-g-schmidt} for such an example).
In Corollary \ref{cor:cpctext}, we showed that every countable, residually finite, inner amenable group is orbitally inner amenable.
More generally, we obtain the following:

\begin{prop}\label{prop:resfinex}
Let $G$ be a countable group with normal subgroup $N$.
Assume that there exists a chain $N=N_0> N_1>\cdots$ of finite index subgroups of $N$ with $\bigcap _i N_i =\{ e\}$, and with each $N_i$ normal in $G$. 
Assume furthermore that there exists a diffuse, conjugation-invariant mean $\bm{m}$ on $G$ with $\bm{m}(N)=1$.
Then $G$ is orbitally inner amenable.
\end{prop}

\begin{proof}
Let $\check{\bm{m}}$ be the mean on $G$ defined by $\check{\bm{m}}(D)=\bm{m}(D^{-1})$ for $D\subset G$.
Then the convolution $\bm{n}_0\defeq \check{\bm{m}}\ast \bm{m}$ is also diffuse and conjugation-invariant with $\bm{n}_0(N)=1$.
For every finite index subgroup $L$ of $N$, by finite additivity, we have $\sum _{hL \in N/L}\bm{m}(hL) =1$ and hence
\[
\bm{n}_0(L) = \int _N \bm{m}(k^{-1}L) \, d\check{\bm{m}}(k) = \sum _{hL\in N/L}\int _{hL}\bm{m}(kL)\, d\bm{m}(k) = \sum _{hL \in N/L}\bm{m}(hL)^2 > 0.
\]
For every $i$, since $N_i$ is normal in $G$, the normalized restriction of $\bm{n}_0$ to $N_i$, denoted by $\bm{n}_i$, is a diffuse, conjugation-invariant mean on $G$ with $\bm{n}_i(N_i)=1$.
Let $\bm{n}$ be any weak${}^*$-cluster point of $(\bm{n}_i )$ in the space of means on $G$.
Then $\bm{n}$ is a diffuse, conjugation-invariant mean on $G$ with $\bm{n}(N_i)=1$ for all $i$.
Let $N\curvearrowright (K,\mu _K )$ be the profinite action associated to the chain $(N_i )$, i.e., the inverse limit of the finite actions $N\curvearrowright (N/N_i , \mu _{N/N_i})$ with each $\mu _{N/N_i}$ the normalized counting measure on $N/N_i$.
We naturally view $N$ as a subgroup of the profinite group $K$.
Observe that for every open neighborhood $U$ of the identity of $K$, we have $\bm{n}(N\cap U ) =1$.
Therefore, for every Borel subset $B\subset K$, we have
\begin{align}\label{eqn:Kint0}
\int _{N } \mu _K (hB\bigtriangleup B ) \, d\bm{n}(h) = 0.
\end{align}
Let $G\curvearrowright (X,\mu )$ be the action co-induced from the action $N\curvearrowright (K,\mu _K)$.
This is a free ergodic p.m.p.\ action of $G$.
We set $\mathcal{R}\defeq \calr(G\c (X, \mu))$.
Since $N$ is normal in $G$, the action $N\c (X,\mu )$ is isomorphic to the product of countably many copies of the action $N\c (K,\mu _K)$.
Therefore, \eqref{eqn:Kint0} implies that for every Borel subset $A\subset X$, we have
\begin{align}\label{eqn:Xint0}
\int _{N } \mu  (hA\bigtriangleup A ) \, d\bm{n}(h) = 0.
\end{align}
Define a mean $\bm{n}_{\mathcal{R}}$ on $(\mathcal{R}, \mu)$ by
\[
\bm{n}_{\mathcal{R}}(D) \defeq \int _{N} \mu ( \{ \, x\in  X \mid (hx,x) \in D \, \} ) \, d\bm{n}(h)
\]
for a Borel subset $D\subset \mathcal{R}$.
The mean $\bm{n}_{\mathcal{R}}$ is diffuse since $\bm{n}$ is diffuse, and $\bm{n}_{\mathcal{R}}$ is balanced by \eqref{eqn:Xint0}.
Since $\bm{n}$ is invariant under conjugation by $G$, and $\mu$ is $G$-invariant, we have $\bm{n}_{\mathcal{R}}(D^{g})=\bm{n}_{\mathcal{R}}(D)$ for all Borel subsets  $D\subset \mathcal{R}$ and all $g\in G$.
By Remark \ref{rem:suffice}, $\bm{n}_{\mathcal{R}}$ is a mean witnessing that $\mathcal{R}$ is inner amenable.
\end{proof}

\begin{prop}\label{prop:wreath}
Let $H$ be a non-trivial countable group, let $K$ be a countable group acting on a countable set $V$, and assume that there exists a diffuse $K$-invariant mean on $V$.
Then the generalized wreath product $G\defeq  H\wr_V K$ is orbitally inner amenable.
\end{prop}

\begin{proof}
By assumption, there exists a sequence $(Q_n)_{n\in \N}$ of finite subsets of $V$ such that
\begin{enumerate}
\item[(i)] $|k\cdot Q_n \bigtriangleup Q_n |/|Q_n|\rightarrow 0$ for all $k\in K$, and
\item[(ii)] $1_{Q_n}(v)\rightarrow 0$ for all $v \in V$.
\end{enumerate}
We set $N\defeq \bigoplus _V H$ and view each element of $N$ as a function $z\colon V\rightarrow H$ whose support $\mathrm{supp}\, (z)\defeq  \{ \, v\in V \mid z(v)\neq e_H \, \}$ is finite.
Let $K$ act on $N$ by $(k\cdot z )(v) =z(k^{-1}\cdot v )$, and identify $N$ and $K$ with subgroups of $G$, so that $G=NK$ is the internal semi-direct product of $N$ with $K$ and $kzk^{-1}=k\cdot z$ for $k\in K$ and $z\in Z$.

Let $H\curvearrowright ^{\alpha _0}(Y_0,\nu _0 )$ be a free ergodic p.m.p.\ action.
We set $(Y,\nu ) \defeq (Y_0^V,\nu _0 ^V )$ and let $N\curvearrowright ^{\alpha}(Y,\nu )$ be the componentwise action given by $(z^{\alpha}y)(v)= z(v)^{\alpha _0} y(v)$ for $z\in N$, $y\in Y$ and $v\in V$.
The action $N\curvearrowright ^{\alpha}(Y,\nu )$ is free and ergodic, and naturally extends to an action of $G$, but this extension may not be free in general.
Instead, we set $(X,\mu ) \defeq (Y^K,\nu ^K )$ and let $G\curvearrowright ^{\beta} (X,\mu )$ be the action co-induced from the action $N\c^\alpha (Y, \nu)$.
Explicitly,
\[((zk)^\beta x )(k_0)\defeq (k_0^{-1}\cdot z)^{\alpha} x(k^{-1}k_0)\]
for $k,k_0\in K$, $z\in N$ and $x\in X$.
This is a free ergodic p.m.p.\ action of $G$.
Since $N$ is normal in $G$, the action $N\c^\beta (X, \mu)$ is isomorphic to the product of countably many copies of the action $N\c^\alpha (Y, \nu)$.

For $v\in V$ and $h\in H$, let $z_{v,h}\in N$ be the element determined by $\mathrm{supp}\, (z_{v,h}) = \{ v\}$ and $z_{v,h}(v)=h$.
Fix any non-trivial $h _0 \in H$, and for each $n\in \N$, set $F_n \defeq \{ \, z_{v,h _0} \mid v\in Q_n \, \}$.
Conditions (i) and (ii) imply that $|F_n^g \bigtriangleup F_n |/|F_n|\rightarrow 0$ and $1_{F_n}(g)\rightarrow 0$ for all $g\in G$.
Moreover, condition (ii) implies that if $C$ is any Borel subset of $Y=Y_0^V$ that depends on only finitely many $V$-coordinates, then for all large enough $n$, $C$ is independent of the coordinates in $Q_n$, and hence for every $z\in F_n$, we have $z^{\alpha}C = C$.
It follows that for all Borel subsets $B\subset Y$, we have $\lim _{n\rightarrow \infty}\sup _{z\in F_n}\nu (z^{\alpha}B\bigtriangleup B ) = 0$.
Since the action $N\c^\beta (X,\mu )$ is isomorphic to the product of countably many copies of the action $N\c^\alpha (Y, \nu)$, it follows that for all Borel subsets $A\subset X$, we have
\begin{equation}\label{eqn:betalim}
\lim _{n\rightarrow \infty}\sup _{z\in F_n}\mu (z^{\beta}A\bigtriangleup A ) = 0.
\end{equation}

We set $\calr \defeq \calr(G\c^\beta (X, \mu))$.
For each $n\in \N$, define $\xi _n\in L^1(\calr, \mu^1)$ by $\xi_n(g^\beta x, x)\defeq 1_{F_n}(g)/|F_n|$ for $g\in G$ and $x\in X$.
Then the sequence $(\xi _n )$ is balanced by \eqref{eqn:betalim}.
Moreover, $(\xi _n)$ inherits from $(F_n)$ the properties of being asymptotically diffuse and asymptotically invariant under conjugation by all elements of $G$.
Therefore, by Remark \ref{rem:suffice}, any weak${}^*$-cluster point of $(\xi _n )$ will be a mean witnessing that $\mathcal{R}$ is inner amenable.
\end{proof}

\begin{rem}\label{rem-wreath}
In the proof of Proposition \ref{prop:wreath}, observe that, since $K\curvearrowright ^{\beta} (X,\mu )$ is a Bernoulli shift, if $K$ is non-amenable, then the action $G\curvearrowright ^{\beta} (X,\mu )$ has stable spectral gap.
Thus, if in addition $K$ is finitely generated and center-less, and $K$ acts on $V$ with infinite orbits, then the centralizer of $K$ in $G$ is trivial, and hence Corollary \ref{cor:specgap} shows that the equivalence relation $\mathcal{R}(G\curvearrowright ^{\beta} (X,\mu ))$ is not Schmidt.
Examples of $K$ and $V$ satisfying those conditions are found in \cite{gm}, and for such $K$ and $V$, $\calr(G\c^\beta (X, \mu))$ is inner amenable, but not Schmidt.
\end{rem}


\subsection{Product groups}

Let $L$ be a non-amenable group and let $H$ be a countable group.
We set $G\defeq L\times H$ and identify $L$ and $H$ with their respective images, $L\times \{ e\}$ and $\{ e\} \times H$, in $G$.
Let $H\curvearrowright (X_0,\mu _0)$ be a free ergodic p.m.p.\ action, and let $G\curvearrowright (X,\mu )\defeq (X_0^L, \mu_0^L)$ be the action such that $H$ acts diagonally and $L$ acts by Bernoulli shift.
Explicitly the action $G\c X$ is given by $(lh\cdot x)(k) = h\cdot x(l^{-1}k)$ for $l,k\in L$, $h\in H$ and $x\in X$.
This is a free ergodic action.
Set $\calr \defeq \calr(G\c (X, \mu))$.
We discuss whether $\calr$ is inner amenable or Schmidt.
This kind of action is considered in \cite{cj} in another context.

\begin{prop}\label{prop-product}
With the above notation,
\begin{enumerate}
\item[(i)] if $H$ is inner amenable and the action $H\c (X_0, \mu_0)$ is profinite, then $\calr$ is inner amenable.
\item[(ii)] If $H$ is finitely generated and the center of $H$ is trivial, then the action $G\c (X, \mu)$ is not Schmidt.
\item[(iii)] If the action $H\c (X_0, \mu_0)$ is mildly mixing, then for every ergodic p.m.p.\ action $G\c (Y, \nu)$, the diagonal action $G\c (X\times Y, \mu \times \nu)$ is not inner amenable.
\end{enumerate}
\end{prop}

\begin{rem}\label{rem-product}
In particular, if we let the group $H$ and the action $H\c (X_0, \mu_0)$ satisfy the assumptions in assertions (i) and (ii) simultaneously, then the action $G\c (X, \mu)$ is inner amenable, but not Schmidt.
For example, let $H$ be the lamplighter group $H=(\Z /2\Z) \wr \Z $.
Then $H$ is finitely generated, residually finite, inner amenable group whose center is trivial.
\end{rem}

\begin{proof}[Proof of Proposition \ref{prop-product}]
The action $G\c (X, \mu)$ is isomorphic to the action co-induced from the action $H\c (X_0, \mu_0)$.
The proof of Proposition \ref{prop:resfinex} (with $H$ playing the role of $N$) therefore shows that $\calr$ is inner amenable under the assumption in assertion (i).

We prove assertion (ii).
The action $L\c (X,\mu )$ is isomorphic to a Bernoulli shift.
Since $L$ is non-amenable, the equivalence relation $\mathcal{R}_L\defeq \mathcal{R}(L\curvearrowright (X,\mu ))$ is not inner amenable, and the action $L\curvearrowright (X,\mu )$ has stable spectral gap.
Therefore, the action $G \curvearrowright (X,\mu )$ has stable spectral gap as well.
Since $H$ is finitely generated and the center of $H$ is trivial, Corollary \ref{cor:specgap} shows that if $(T_n )$ is any central sequence in $[\mathcal{R}]$, then there exists a sequence $(l_n)$ in $L$ such that $\mu ( \{ \, x \in X \mid T_n (x)= l_n x \, \} )\rightarrow 1$.
Since $[\mathcal{R}_L]$ has no non-trivial central sequence, we must have $l_n=e$ for all large enough $n$, and hence $(T_n)$ is asymptotically trivial.

Assertion (iii) directly follows from Proposition \ref{prop-diagonal-action}.
\end{proof}

\begin{rem}\label{rem-product-group-action}
In Proposition \ref{prop-product} (ii), if the center of $H$ is non-trivial, then the action $G\c (X, \mu)$ may be Schmidt. 
Indeed, let $H\defeq \Z$ and let $X_0\defeq \Z_2$ be the group of 2-adic integers equipped with the normalized Haar measure $\mu_0$.
Let $H\c (X_0, \mu_0)$ be the action given by translation.
Then the action $G\c (X, \mu)$ is not stable (since it has stable spectral gap), but is Schmidt since the element $2^n$ of $H=\Z$ for $n\in \N$ forms a central sequence in $[\calr]$ (see also \cite[Section 29 (C)]{kec}).

In Proposition \ref{prop-product} (iii), without the action $H\c (X_0, \mu_0)$ being mildly mixing, the conclusion may fail. 
Indeed, keeping the notation in the last paragraph, we set $(Y, \nu)\defeq (X_0, \mu_0)$ and let $G$ act on $(Y, \nu)$ via the projection $G\to H=\Z$ and the translation on $\Z_2$ by $\Z$.
Then the diagonal action $G\c^\alpha (X\times Y, \mu \times \nu)$ is stable, seen as follows:
Let $\calr_\alpha$ be the orbit equivalence relation of this action.
For $n\in \N$, let $T_n\in [\calr_\alpha]$ be the element $2^n$ of $H=\Z$, and let $A_n$ be the set of all $(x, y)\in X\times Y$ with $y_n=0$ when $y$ is written as $y=\sum_{k=0}^{\infty}2^ky_k\in \Z_2$ with $y_k\in \{ 0, 1\}$.
Then $(T_n, A_n)$ is a pre-stability sequence for $\calr_\alpha$.
\end{rem}

We construct the following interesting examples via actions of product groups:

\begin{prop}\label{prop:noinvol}
There exists an ergodic discrete p.m.p.\ equivalence relation whose full group admits a non-trivial central sequence, but admits no non-trivial central sequence consisting of periodic transformations with uniformly bounded periods.

Furthermore, there exists an ergodic discrete p.m.p.\ equivalence relation whose full group admits a non-trivial central sequence, but admits no non-trivial central sequence consisting of aperiodic transformations.
\end{prop}

\begin{proof}
In what follows, for a group $G$, its center is denoted by $Z(G)$.
Let $L_0$ and $L_1$ be a finitely generated, non-amenable group with trivial center.
Let $H_0$ and $H_1$ be finitely generated groups, and suppose that $Z(H_0)$ is torsion-free and infinite (e.g., take $H_0=\Z$) and that $Z(H_1)$ is isomorphic to the direct sum of infinitely many copies of $\Z /2\Z$ (e.g., by \cite[Theorem 6]{hall}, every countable abelian group is isomorphic to the center of some 2-generated central-by-metabelian group).
For each $i=0,1$, we set $G_i\defeq L_i\times H_i$, so that $Z(G_i) = \{ e\} \times Z(H_i)$.
Since $Z(H_i)$ is infinite, we have a free ergodic p.m.p.\ action $H_i\curvearrowright (Y_i, \nu _i )$ such that the image of $Z(H_i)$ in $\mathrm{Aut}(Y_i,\nu _i )$ is precompact (see the proof of \cite[Theorem 15]{td} or Example \ref{ex-infinite-center} below).
Set $(X_i,\mu _i ) \defeq (Y_i^{L_i}, \nu _i ^{L_i})$ and let $G_i$ act on $(X_i,\mu _i )$ via $((l,h)\cdot x)(k)=h\cdot x(l^{-1}k)$ for $l,k\in L_i$, $h\in H_i$ and $x\in X_i$.
Let $\calr_{G_i}$ denote the orbit equivalence relation of this action $G_i\c (X_i, \mu_i)$.

The image of $Z(G_i)=\{ e\} \times Z(H_i)$ in $\mathrm{Aut}(X_i,\mu _i)$ is precompact, so the full group $[\mathcal{R}_{G_i}]$ admits a non-trivial central sequence consisting of non-trivial elements of $Z(G_i)$ which converge to the identity transformation in $\mathrm{Aut}(X_i,\mu _i )$.
The action $G_i\c (X_i,\mu _i )$ is free and has stable spectral gap since $L_i$ is non-amenable and $L_i\times \{ e\}$ acts by Bernoulli shift on $(X_i,\mu _i)$.
If $(T_{n,i})_{n}$ is any central sequence of elements in $[\mathcal{R}_{G_i}]$, then, since $G_i$ is finitely generated, Corollary \ref{cor:specgap} shows that there exists a sequence $(c_{n, i})_n$ of non-trivial elements of $Z(G_i)$ such that $\mu _i ( \{ \, x \in X_i \mid T_{n,i}(x) = c_{n,i} x \, \} ) \rightarrow 1$.
Therefore, since $Z(G_0)$ is torsion-free, the sequence $(T_{n,0})$ cannot consist of periodic transformations with uniformly bounded periods.
Likewise, since every non-trivial element of $Z(G_1)$ has order 2, the sequence $(T_{n,1})$ cannot consist of aperiodic transformations.
\end{proof}


\subsection{Groups with Schmidt or stable actions}

We deal with inner amenable groups which appear in certain contexts.


\begin{ex}\label{ex-infinite-center}
Let $G$ be a countable group with an infinite central subgroup $C$.
A Schmidt action of $G$ is constructed in the proof of \cite[Theorem 15]{td} and Proposition \ref{prop-t-c}, obtained as follows:
Embed $C$ into a compact abelian group $K$, equip $K$ with the normalized Haar measure, and let $C$ act on $K$ by translation.
Let $G\c X=\prod_{G/C}K$ be the action co-induced from the action $C\c K$.
Then $C$ acts on the product $X=\prod_{G/C}K$ diagonally since $C$ is central.
Therefore if $(c_n)$ is any sequence of non-trivial elements of $C$ converging to the identity in $K$, we have $\mu(c_nA\bigtriangleup A)\to 0$ for all Borel subsets $A\subset X$, where $\mu$ is the probability measure on $X$ given by the product of the Haar measure on $K$.
Thus $(c_n)$ is a non-trivial central sequence in the full group $[G\c (X, \mu)]$ by Remark \ref{rem-central}.
\end{ex}

\begin{ex}
Let $H$ be a countable group with property (T) and suppose that $H$ has a central element $a$ of infinite order.
For non-zero integers $p$, $q$ with $|p|\neq |q|$, define $G$ to be the HNN extension $G\defeq \langle \, H,\, t\mid ta^pt^{-1}=a^q\, \rangle$.
The group $G$ is ICC, inner amenable and not stable (\cite{kida-inn}).

We construct a Schmidt action of $G$.
Let $C$ be the group generated by $a$.
Embed $C$ into a compact abelian group $K$ and co-induce the action $G\c X=\prod_{G/C}K$ as in Example \ref{ex-infinite-center}.
Let $\mu$ be the probability measure on $X$.
Let $N$ be the kernel of the homomorphism $G\to \Z$ sending each element of $H$ to $0$ and $t$ to $1$.
We set $a_n\defeq a^{p^nq^n}$.
The sequence $(a_n)_{n\in \N}$ is asymptotically central in $N$, i.e., for every $g\in N$, for all sufficiently large $n$, we have $a_ng=ga_n$.
Hence the sequence of elements $a_ma_n^{-1}$ running through some $m>n$ converges to the identity in $K$ and is central in the full group $[N\c (X, \mu)]$.
Take a free ergodic p.m.p.\ action $G/N\c (Y, \nu)$ and let $G$ act on $X\times Y$ diagonally.
As noted in \cite[Remark 7.4]{td}, using this non-trivial central sequence in $[N\c (X, \mu)]$ and the hyperfiniteness of the equivalence relation $\calr(G/N\c (Y, \nu))$, we can construct a non-trivial central sequence in $[G\c (X\times Y, \mu \times \nu)]$.
\end{ex}

\begin{ex}
For a group $G$, its \textit{FC-center} $R$ is defined as the set of elements of $G$ whose centralizer in $G$ is of finite index in $G$, or equivalently, whose conjugacy class in $G$ is finite.
Then $R$ is a normal subgroup of $G$.
Popa-Vaes \cite[Theorem 6.4]{pv} show that if a countable group $G$ is residually finite and its FC-center is not virtually abelian, then $G$ is \textit{McDuff}, i.e., $G$ admits a free ergodic p.m.p.\ action such that the associated von Neumann algebra tensorially absorbs the hyperfinite $\textrm{II}_1$ factor.
A property (T) group satisfying that assumption does exist (\cite{ershov}), and it provides an example of a non-stable, McDuff group.
Based on this group, Deprez-Vaes \cite[Section 3]{dv} constructed an ICC, non-stable, McDuff group.

Popa-Vaes' construction applies to show that every countable, residually finite group $G$ with infinite FC-center has the Schmidt property.
In fact, in their proof, it is shown that the sequence $(v_h)_h$ in the symbol of that proof, where $h$ runs through some elements of $G$, is a non-trivial central sequence in the full group.
We note that in \cite[Theorem 6.4]{pv}, while the FC-center of $G$ is assumed to be not virtually abelian, this condition is required for the action to be McDuff.
For the action to admit a non-trivial central sequence in its full group, it suffices to assume that the FC-center of $G$ is infinite.

Similarly the group $G$ of Deprez-Vaes \cite[Section 3]{dv} also has the Schmidt property.
In fact, the action of $G$ they constructed admits the central sequence $(v_h)_h$ in the symbol of their proof.
\end{ex}

\begin{ex}
Let $A$ be a countably infinite, amenable group.
Let $H$ be a countable group acting on $A$ by automorphisms and suppose that all $H$-orbits in $A$ are finite.
This condition holds, e.g., if $A$ is an increasing union of its finite subgroups invariant under the action of $H$ (and this holds, e.g., for the action $\mathit{SL}_n(\Z)\c (\Z[1/2]/\Z)^n$ induced from the linear action on $\R^n$).
Let $G\defeq H\ltimes A$ be the semi-direct product.
Then the FC-center of $G$ contains $A$ and is hence infinite.
We present below two constructions of a free ergodic p.m.p.\ action of $G$ which is stable.
It is remarkable that for the first stable action of $G$, its restriction to $A$ is ergodic (or rather mixing), while it is not ergodic for the second one.

\medskip

\noindent \textbf{The first construction.}
We set $X\defeq [0, 1]^A$ and let $\mu$ be the probability measure on $X$ given by the product of the Lebesgue measure on $[0, 1]$.
The group $G$ acts on $A$ by affine transformations:
$(h, a)\cdot b=h(ab)$ for $h\in H$ and $a, b\in A$.
This action induces the action $G\c (X, \mu)$, which is p.m.p.\ and ergodic, but is not necessarily free (e.g., if the action of $H$ on $A$ is not faithful).
We show that the groupoid $G\ltimes (X, \mu)$ is stable, i.e., it absorbs the ergodic p.m.p.\ aperiodic hyperfinite equivalence relation, under the direct product of groupoids.
This is enough for $G$ to admit a free ergodic p.m.p.\ action which is stable, by \cite[Theorem 1.4]{kida-srt}.

Let $K$ be the closure of the image of $H$ in $\aut(A)$, the automorphism group of $A$ endowed with the pointwise convergence topology.
By the assumption that all $H$-orbits in $A$ are finite, the group $K$ is compact.
We also have the p.m.p.\ action $K\ltimes A\c (X, \mu)$ induced by the affine action $K\ltimes A\c A$.

\begin{lem}\label{lem-kafree}
The action $K\ltimes A\c (X, \mu)$ is essentially free.
\end{lem}

\begin{proof}
Let $X'$ be the set of all $x\in X$ such that $x(b_1)\neq x(b_2)$ for all distinct $b_1, b_2\in A$.
Then $X'$ is co-null and $(K\ltimes A)$-invariant.
We verify that the action $K\ltimes A\c X'$ is free.
Suppose that $x\in X$, $k\in K$ and $a\in A$ satisfy $(k, a)x=x$.
Then $x((k, a)^{-1}b)=x(b)$ for all $b\in A$.
Letting $b=e$, we have $x(a^{-1})=x(e)$ and hence $a=e$.
Then $x(k^{-1}b)=x(b)$ for all $b\in A$ and hence $k^{-1}b=b$, and $k=e$.
\end{proof}

Let $\calr$ be the orbit equivalence relation of the action $K\ltimes A\c (X, \mu)$.
Let $Y\defeq X/K$ be the quotient space by the action $K\c X$, which is a standard Borel space since $K$ is compact.
Let $p\colon X\to Y$ be the projection.
Then the discrete p.m.p.\ equivalence relation $\cals$ on $(Y, \nu)$, where $\nu \defeq p_*\mu$, is induced as follows:
The set $\cals$ is given by the quotient of the coordinatewise action $K\times K \c \calr$.
We have the induced map $p\colon \calr \to \cals$, and set $\calr_1\defeq \calr(A\c (X, \mu))$.
Then the restriction $p\colon \calr_1\to \cals$ induces a bijection from $(\calr_1)_x$ onto $\cals_{p(x)}$ for almost every $x\in X$.
Indeed, it is injective by Lemma \ref{lem-kafree}, and is surjective since $K$ normalizes $A$.
Since $A$ is amenable, $\calr_1$ is amenable.
Hence $\cals$ is amenable (this follows from e.g., the characterization of amenability as the existence of an invariant state, given in \cite[Definition 3.2.8]{ar} or \cite[Definition 4.57 (i)]{kerr-li}). 

We have an action of $\cals$ on the fibered space $X$ with respect to $p$:
For $(y_2, y_1)\in \cals$ and $x\in X$ with $p(x)=y_1$, we define a point $(y_2, y_1)\cdot x$ as the unique point of $p^{-1}(y_2)$ that is $\calr_1$-equivalent to $x$.
Then this action of $\cals$ and the action of $K$ commute in the sense that for all $(y_2, y_1)\in \cals$, $k\in K$ and $x\in X$, we have $k((y_2, y_1)\cdot x)=(y_2, y_1)\cdot (kx)$.
This holds because $K$ normalizes $A$.

We construct a cocycle $\alpha \colon \cals \to K$.
Choose a Borel section $q\colon Y\to X$ of the quotient map $p\colon X\to Y$ (\cite[Theorem 12.16]{kec-set}).
For $(y_2, y_1)\in \cals$, with $x_2\defeq (y_2, y_1)\cdot q(y_1)$, we define an element $\alpha(y_2, y_1)$ of $K$ by the equation $\alpha(y_2, y_1)x_2=q(y_2)$.
Then the map $\alpha$ is a cocycle.
Indeed, for $(y_3, y_2), (y_2, y_1)\in \cals$, with $x_3\defeq (y_3, y_1)\cdot q(y_1)=(y_3, y_2)\cdot x_2$, we have
\begin{align*}
\alpha(y_3, y_2)\alpha(y_2, y_1)x_3=\alpha(y_3, y_2)(y_3, y_2)\cdot ( \alpha(y_2, y_1)x_2)=\alpha(y_3, y_2)(y_3, y_2)\cdot q(y_2)=q(y_3).
\end{align*}
By the definition of $\alpha$, this implies that $\alpha(y_3, y_2)\alpha(y_2, y_1)=\alpha(y_3, y_1)$.

Since $\cals$ is amenable and hence stable, by Theorem \ref{thm-main}, after choosing a decreasing sequence $V_1\supset V_2\supset \cdots$ of open neighborhoods of the identity in $K$ with $\bigcap_nV_n=\{ e\}$, we may find a pre-stability sequence $(T_n, B_n)$ for $\cals$ such that $\alpha(T_ny, y)\in V_n$ for every $n$ and almost every $y\in Y$.
We define the lift $\tilde{T}_n\in [\calr_1]$ of $T_n$ by $\tilde{T}_nx\defeq (T_np(x), p(x))\cdot x$ for $x\in X$.
Then $\tilde{T}_n$ commutes with all elements of $K$, asymptotically commutes with the lifts of all elements of $[\cals]$, and satisfies $\mu(\tilde{T}_nB\bigtriangleup B)\to 0$ for all Borel subsets $B\subset X$.
We identify $\calr_1$ with the translation groupoid $A\ltimes (X, \mu)$, and regard $\tilde{T}_n$ as an element of $[A\ltimes (X, \mu)]$.
Then $\tilde{T}_n$ commutes with all elements of $[H\ltimes (X, \mu)]$.
Since $H\ltimes (X, \mu)$ and the lifts of elements of $[\cals]$ generate $G\ltimes (X, \mu)$, $(\tilde{T}_n)$ is a central sequence in $[G\ltimes (X, \mu)]$ by Remark \ref{rem-central}.
Therefore $(\tilde{T}_n, p^{-1}(B_n))$ is a pre-stability sequence for $G\ltimes (X, \mu)$, and $G\ltimes (X, \mu)$ is stable by \cite[Theorem 3.1]{kida-srt}.

\medskip

\noindent \textbf{The second construction.}
Let $K$ be the closure of the image of $H$ in $\aut(A)$ again.
We set $D\defeq H\times A$ and let $K$ act on $D$ by automorphisms such that $K$ acts on $H\times \{ e\}$ trivially and acts on $\{ e\} \times A$ as elements of $\aut(A)$ under the identification of $\{ e\} \times A$ with $A$.
Then $L\defeq K\ltimes D$ is a locally compact second countable group, and clearly $D$ is a lattice in $L$.
Furthermore we have the embedding $i \colon G\to L$ defined by $i(h, a)=(j(h), (h, a))$ for $h\in H$ and $a\in A$, where $j \colon H\to \aut(A)$ is the homomorphism arising from the action of $H$ on $A$.
The image of $i$ is then a lattice in $L$.
Therefore $L$ is a measure-equivalence coupling of $D$ and $G$ with respect to the left and right multiplications.

We have the p.m.p.\ action $D\c L/i(G)$, and its restriction to $\{ e\} \times A$ is trivial.
Indeed, for all $a, a'\in A$, $k\in K$ and $h'\in H$, we have
\begin{align*}
&(e_K, (e_H, a))(k, (h', a'))=(k, (h', (k^{-1}\cdot a)a'))\\
&=(k, (h', a'))(e_K, (e_H, (a')^{-1}(k^{-1}\cdot a)a'))\in (k, (h', a'))i(G).
\end{align*}
Choose a free ergodic p.m.p.\ action $A\c (X, \mu)$ arbitrarily.
Let $D\times G$ act on $L\times X$ by
\[((h, a), g)(l, x)=((h, a)l i(g)^{-1}, ax)\]
for $h\in H$, $a\in A$, $g\in G$, $l\in L$ and $x\in X$.
Then $L\times X$ is a measure-equivalence coupling of $D$ and $G$ such that the groupoid $D\ltimes ((L\times X)/i(G))$ is stable.
Therefore the groupoid $G\ltimes ((L\times X)/D)$ is also stable, and $G$ is stable by \cite[Theorem 1.4]{kida-srt}.
\end{ex}


\section{Topologies of the full group}\label{sec-top}

Let $(\calg, \mu)$ be a discrete p.m.p.\ groupoid.
Extending the definition in Subsection \ref{subsec-schmidt}, we call a sequence $(\phi_n)$ of elements of $[\calg]$ \textit{central} in $[\calg]$ if $\mu (\{ \, x\in \calg^0 \mid (\psi \phi_n)_x\neq (\phi_n \psi )_x \, \} )\to 0$ for all $\psi \in [\calg]$, and call a central sequence $(\phi_n)$ in $[\calg]$ \textit{trivial} if
\[\mu (\{ \, x\in \calg^0 \mid  \text{$(\phi_n)_x$ is a central element of the group $\calg_x^x$} \, \} ) \to 1.\]
In this section, we show that there exists a non-trivial central sequence in $[\calg]$ if and only if the set of inner automorphisms of $\calg$ given by conjugation of an element of $[\calg]$ is not closed in $\aut(\calg)$, the automorphism group of $\calg$, under the assumption that the associated equivalence relation $\calr_\calg =\{ \, (r(g), s(g))\mid g\in \calg\, \}$ is aperiodic.

Let us discuss the topologies of $[\calg]$ and $\aut(\calg)$.
We endow the full group $[\calg]$ with the \textit{uniform topology} induced by the metric $\delta_u(\phi, \psi)\defeq \mu(\{ \, x\in \calg^0\mid \phi_x\neq \psi_x\, \})$.
This metric is bi-invariant, i.e., $\delta_u(\psi \phi_1, \psi \phi_2)=\delta_u(\phi_1, \phi_2)=\delta_u(\phi_1 \psi, \phi_2 \psi)$ for all $\phi_1, \phi_2, \psi \in [\calg]$.

\begin{lem}\label{lem-polish}
With respect to the uniform topology, $[\calg]$ is a Polish group.
\end{lem}

\begin{proof}
We first prove that $[\calg]$ is a topological group.
For all $\phi, \psi \in [\calg]$, we have $\delta_u(\phi, \psi)=\delta_u(\phi^{-1}, \psi^{-1})$, and therefore the map $\phi \mapsto \phi^{-1}$ is continuous.
For all $\phi_1, \phi_2, \psi_1, \psi_2\in [\calg]$, we have $\delta_u(\phi_1\psi_1, \phi_2\psi_2)\leq \delta_u(\phi_1, \phi_2)+\delta_u(\psi_1, \psi_2)$.
Indeed if $x\in \calg^0$ satisfies $(\phi_1\psi_1)_x\neq (\phi_2\psi_2)_x$, then either $(\psi_1)_x\neq (\psi_2)_x$ or $(\psi_1)_x=(\psi_2)_x$ and $(\phi_1)_{r((\psi_1)_x)}\neq (\phi_2)_{r((\psi_1)_x)}$, and this implies the inequality.
Thus the map $(\phi, \psi)\mapsto \phi \psi$ is continuous.

We prove that the topology is Polish.
This is proved by embedding $[\calg]$ into $L^1(\calg, \mu^1)$ by identifying each element of $[\calg]$ with its indicator function.
For all $\phi, \psi \in [\calg]$, we then have $\delta_u(\phi, \psi)=\Vert \phi -\psi \Vert_{L^1(\mu^1)}/2$.
Thus the uniform topology on $[\calg]$ coincides with the relative topology induced by the norm of $L^1(\calg, \mu^1)$.
The set $[\calg]$ is norm-closed in $L^1(\calg, \mu^1)$.
This is because a function $f\in L^1(\calg, \mu^1)$ belongs to $[\calg]$ if and only if it satisfies the following three conditions, each of which defines a norm-closed subset of $L^1(\calg, \mu^1)$: $f$ takes only values $0$ or $1$, the function $x\mapsto c_x^s(f|_{\calg_x})$ on $\calg^0$ is equal to $1$ almost everywhere, and the function $x\mapsto c_x^r(f|_{\calg^x})$ on $\calg^0$ is also equal to $1$ almost everywhere, where $c_x^s$ and $c_x^r$ are the counting measures on the fibers $\calg_x$ and $\calg^x$, respectively.
The second condition defines a norm-closed subset since the linear map $L^1(\calg, \mu^1)\to L^1(\calg^0, \mu)$ sending $f$ to the function in the condition is contractive.
Similarly the third condition defines a norm-closed subset.
\end{proof}

We define $\aut(\calg)$ as the group of Borel groupoid-automorphisms of $\calg$, where two such are identified if they agree $\mu^1$-almost everywhere.
We will endow $\aut(\calg)$ with a topology by embedding it into a larger Polish group.
Let $\Iso([\calg], \delta_u)$ be the group of isometries of the metric space $([\calg], \delta_u)$.
We endow the group $\Iso([\calg], \delta_u)$ with the pointwise convergence topology, which makes it into a Polish group (\cite[9.B, 9)]{kec-set}).
For $F\in \aut(\calg)$ and $\phi \in [\calg]$, we define the section $F(\phi)\in [\calg]$ by $F(\phi)_x=F(\phi_{F^{-1}(x)})$ for $x\in \calg^0$. 
Then we obtain the homomorphism $i\colon \aut(\calg)\to \Iso([\calg], \delta_u)$ defined by $i(F)(\phi)=F(\phi)$ for $F\in \aut(\calg)$ and $\phi \in [\calg]$.

\begin{prop}\label{prop-i}
Suppose that the associated equivalence relation
\[\calr=\calr_\calg \defeq \{ \, (r(g), s(g))\mid g\in \calg\, \}\]
is aperiodic, i.e., almost every equivalence class of $\calr$ is infinite.
Then
\begin{enumerate}
 \item[(i)] the homomorphism $i\colon \aut(\calg)\to \Iso([\calg], \delta_u)$ is injective.
 \item[(ii)] The image $i(\aut(\calg))$ is closed in the Polish group $\Iso([\calg], \delta_u)$.
 \end{enumerate}
\end{prop}

\begin{proof}
For a principal $\calg$, assertion (i) is proved in \cite[Corollary 4.11]{kec}.
We reduce the proof for a general $\calg$ to that for a principal one.
Pick $F\in \aut(\calg)$ such that $i(F)$ is the identity.
The map $F$ induces the automorphism $\bar{F}\in \aut(\calr)$.
Fix a Borel section $\sigma$ of the quotient map $\calg \to \calr$.
For $\phi \in [\calr]$, its lift $\phi^\sim \in [\calg]$ is defined by $(\phi^\sim)_x=\sigma(\phi(x), x)$ for $x\in \calg^0$.
Since $i(F)$ is the identity, for every $\phi \in [\calr]$, we have $F(\phi^\sim)=\phi^\sim$ and thus $\bar{F}(\phi)=\phi$.
By \cite[Corollary 4.11]{kec} (i.e., assertion (i) for a principal $\calg$), $\bar{F}$ is the identity.
Hence $F$ induces the group automorphism of the isotropy group $\calg_x^x$ at every $x\in \calg^0$.
The isotropy subgroupoid $\calg_{\mathrm{isot}}\defeq \bigcup_{x\in \calg^0}\calg_x^x$ is covered by countably many sections in $[\calg_{\mathrm{isot}}]$, and each of those sections are fixed by $F$.
It thus follows that $F$ is the identity on $\calg_x^x$ for almost every $x\in \calg^0$.
Assertion (i) follows.

We prove assertion (ii).
Assertion (ii) is also proved for a principal $\calg$ in \cite[p.41]{kec}, and we will use it in the proof.
Pick a sequence $(F_n)$ in $\aut(\calg)$ such that the sequence $(i(F_n))$ converges in $\Iso([\calg], \delta_u)$.
Then for all $\phi \in [\calg]$, the sequence $(F_n(\phi))$ converges in $[\calg]$.
We have to find an $F\in \aut(\calg)$ such that $i(F_n)$ converges to $i(F)$ in $\Iso([\calg], \delta_u)$.
For each $\phi \in [\calg]$, since $F_n(\phi)$ converges to some $\psi \in [\calg]$, it is natural to define the automorphism $F\colon \calg \to \calg$ so that $F(\phi)=\psi$.
To realize this idea, we need the following claim:

\begin{claim}\label{claim-cover}
Let $\phi_1, \phi_2\in [\calg]$ and assume that for some non-null Borel subset $A\subset \calg^0$, we have $(\phi_1)_x=(\phi_2)_x$ for all $x\in A$.
Let $\psi_1, \psi_2\in [\calg]$ be the limits of the sequences $(F_n(\phi_1))$, $(F_n(\phi_2))$ in the metric $\delta_u$, respectively.
Then $(\psi_1)_x=(\psi_2)_x$ for almost every $x\in A$.
\end{claim}

\begin{proof}
We put $\phi \defeq (\phi_2)^{-1}\phi_1$ and $\psi \defeq (\psi_2)^{-1}\psi_1$.
Then $F_n(\phi)$ converges to $\psi$ in the metric $\delta_u$, and it suffices to show that $\psi_x=x$ for all $x\in A$.
Since $\phi_x=x$ for all $x\in A$, we have $F_n(\phi)_x=x$ for all $x\in A$ and
\[\mu(\{ \, x\in A\mid \psi_x\neq x\, \})\leq \mu(\{ \, x\in A\mid F_n(\phi)_x\neq \psi_x\, \})\leq \delta_u(F_n(\phi), \psi)\to 0\]
as $n\to \infty$.
Thus $\psi_x=x$ for almost every $x\in A$.
\end{proof}

Prior to defining the desired automorphism $F\colon \calg \to \calg$, we need to construct the automorphism $F^o\colon \calg^0\to \calg^0$ that $F$ will extend.
In fact, we will define the map $F$ so that for each $\phi \in [\calg]$, letting $\psi \in [\calg]$ be the limit of $F_n(\phi)$, we have $F(\phi_x)=\psi_{F^o(x)}$ for $\mu$-almost every $x\in \calg^0$.
In the definition of the homomorphism $i\colon \aut(\calg)\to \Iso([\calg], \delta_u)$, replacing $\calg$ into $\calr$, we obtain the homomorphism $\aut(\calr)\to \Iso([\calr], \delta_u)$.
We use the same symbol $i$ to denote this homomorphism if there is no cause for confusion.
Let $\bar{F}_n\in \aut(\calr)$ denote the automorphism induced by $F_n\in \aut(\calg)$.

\begin{claim}\label{claim-ibarfn}
The sequence $(i(\bar{F}_n))$ converges in $\Iso([\calr], \delta_u)$.
\end{claim}

If this is shown, then since the image $i(\aut(\calr))$ is closed in $\Iso([\calr], \delta_u)$ (\cite[p.41]{kec}), we obtain $\bar{F}\in \aut(\calr)$ such that $i(\bar{F}_n)$ converges to $i(\bar{F})$ in $\Iso([\calr], \delta_u)$, i.e., $\bar{F}_n(\phi)$ converges to $\bar{F}(\phi)$ in the metric $\delta_u$ for all $\phi \in [\calr]$.
We then obtain the automorphism $F^o\in \aut(\calg^0)$ that restricts $\bar{F}$.

\begin{proof}[Proof of Claim \ref{claim-ibarfn}]
Choosing a countable dense subset $\{ \theta_k\}$ of $[\calr]$, we obtain the compatible metric $\delta$ on the group $\Iso([\calr], \delta_u)$ defined by
\[\delta(\alpha, \beta)\defeq \sum_{k=1}^\infty 2^{-k}(\delta_u(\alpha(\theta_k), \beta(\theta_k))+\delta_u(\alpha^{-1}(\theta_k), \beta^{-1}(\theta_k)))\]
for $\alpha, \beta \in \Iso([\calr], \delta_u)$.
By \cite[Corollary 1.2.2]{bk}, this metric $\delta$ is complete.
Hence it suffices to show that $(i(\bar{F}_n))$ is a Cauchy sequence with respect to this metric $\delta$, and to show it, it suffices to show that for every $\phi \in [\calr]$, $(\bar{F}_n(\phi))$ is a Cauchy sequence with respect to the metric $\delta_u$.
We note that if $\bar{F}_n$ and $\phi$ are regarded as a map from $\calg^0$ into itself, then $\bar{F}_n(\phi)$ is by definition the section $\{ \, (\bar{F}_n(\phi(x)), \bar{F}_n(x))\mid x\in \calg^0\, \}$, and is identified with the automorphism $\bar{F}_n\circ \phi \circ \bar{F}_n^{-1}\colon \calg^0\to \calg^0$.

Pick $\phi \in [\calr]$.
We then have the lift $\phi^\sim \in [\calg]$ defined by $(\phi^\sim)_x\defeq \sigma(\phi(x), x)$ as before, where $\sigma$ is a fixed Borel section of the quotient map $\calg \to \calr$.
Since the sequence $(i(F_n))$ converges in $\Iso([\calg], \delta_u)$, the sequence $(F_n(\phi^\sim))$ converges in $[\calg]$ in the metric $\delta_u$.
Thus
\begin{align*}
\delta_u(\bar{F}_n(\phi), \bar{F}_m(\phi))&=\mu(\{ \, x\in \calg^0\mid (\bar{F}_n(\phi))(x)\neq (\bar{F}_m(\phi))(x)\, \})\\
&\leq \mu(\{ \, x\in \calg^0\mid F_n(\phi^\sim)_x\neq F_m(\phi^\sim)_x\, \})=\delta_u(F_n(\phi^\sim), F_m(\phi^\sim))\to 0
\end{align*}
as $n, m\to \infty$, where the inequality holds since $(\bar{F}_n(\phi))(x)$ is the range of $F_n(\phi^\sim)_x$.
\end{proof}

We now define a map $F\colon \calg \to \calg$.
First write $\calg$ as the union of countably many sections in $[\calg]$, $\calg =\bigcup_{k\in \N}\phi_k$ (where those sections may not be mutually disjoint).
For each $k\in \N$, let $\psi_k\in [\calg]$ be the limit of the sequence $(F_n(\phi_k))_n$ in $[\calg]$, and set $F((\phi_k)_x)\defeq (\psi_k)_{F^o(x)}$ for $x\in \calg^0$.
By Claim \ref{claim-cover}, this map $F\colon \calg \to \calg$ is well-defined (after discarding a null set), and does not depend on the choice of the sections $\phi_k$ covering $\calg$.
To show that $F$ belongs to $\aut(\calg)$, we first observe the following:

\begin{claim}\label{claim-sr}
For $\mu^1$-almost every $g\in \calg$, we have $s(F(g))=F^o(s(g))$ and $r(F(g))=F^o(r(g))$. 
\end{claim}

\begin{proof}
The first equation follows from the definition of $F$.
We prove the second equation.
For $\phi \in [\calg]$, let $\phi^-\in [\calr]$ be the induced element given by the map $x\mapsto r(\phi_x)$.
Pick $\phi \in [\calg]$ and let $\psi \in [\calg]$ be the limit of $F_n(\phi)$.
Since $i(\bar{F}_n)\to i(\bar{F})$ in $\Iso([\calr], \delta_u)$, we have $\bar{F}_n(\phi^-)\to \bar{F}(\phi^-)$ in $[\calr]$.
We also have $\bar{F}_n(\phi^-)=F_n(\phi)^-\to \psi^-$ in $[\calr]$.
Therefore $\bar{F}(\phi^-)=\psi^-$.
This means that for $\mu^1$-almost every $x\in \calg^0$, we have $\bar{F}(\phi^-(x), x)=(\psi^-(F^o(x)), F^o(x))$, and the range of the left hand side is $F^o(\phi^-(x))$.
Thus $F^o(r(\phi_x))=F^o(\phi^-(x))=\psi^-(F^o(x))=r(\psi_{F^o(x)})$, and this proves the second equation in the claim.
\end{proof}

By the definition of $F$, for every $\phi \in [\calg]$, letting $\psi \in [\calg]$ be the limit of $F_n(\phi)$, we have $\{ \, F(\phi_x)\mid x\in \calg^0\, \} =\psi$.
Let us denote this $\psi$ by $F(\phi)$.
Note that $F_n(\phi)\to F(\phi)$ in $[\calg]$.

\begin{claim}\label{claim-f-aut}
The map $F\colon \calg \to \calg$ belongs to $\aut(\calg)$.
\end{claim}

\begin{proof}
We first verify that the map $F\colon \calg \to \calg$ is a groupoid-homomorphism, i.e., for $\mu^1$-almost every $h\in \calg$, the equation $F(gh)=F(g)F(h)$ holds for all $g\in \calg_{r(h)}$.
Claim \ref{claim-sr} ensures that the product $F(g)F(h)$ is defined.
It suffices to show that for all $\phi, \psi \in [\calg]$, we have $F(\phi_{r(\psi_x)}\psi_x)=F(\phi_{r(\psi_x)})F(\psi_x)$ for $\mu$-almost every $x\in \calg^0$.
Since $F_n\in \aut(\calg)$, we have $F_n(\phi \psi)=F_n(\phi)F_n(\psi)$, and as $n\to \infty$, we have $F(\phi \psi)=F(\phi)F(\psi)$.
We then obtain the desired equation by comparing $F(\phi \psi)_{F^o(x)}$ and $(F(\phi)F(\psi))_{F^o(x)}$.

We verify that $F$ has an inverse.
The map $F$ is constructed from the sequence $(F_n)$ in $\aut(\calg)$ such that $(i(F_n))$ converges in $\Iso([\calg], \delta_u)$.
Similarly from the sequence $(F_n^{-1})$, we can construct the map $G\colon \calg \to \calg$ such that for each $\phi \in [\calg]$, the set $\{ \, G(\phi_x)\mid x\in \calg^0\, \}$ is an element of $[\calg]$, denoted by $G(\phi)$, and it is the limit of $F_n^{-1}(\phi)$.
Pick $\phi \in [\calg]$.
Then for every $n$, by the triangle inequality,
\[\delta_u(G(F(\phi)), \phi)\leq \delta_u(G(F(\phi)), F_n^{-1}(F(\phi)))+\delta_u(F_n^{-1}(F(\phi)), \phi).\]
In the right hand side, the first term converges to $0$ as $n\to \infty$, and the second term is equal to $\delta_u(F(\phi), F_n(\phi))$, which also converges to $0$.
Therefore $G(F(\phi))=\phi$ for all $\phi \in [\calg]$, and this implies that $G\circ F$ is the identity $\mu^1$-almost everywhere.
Similarly $F\circ G$ is also the identity $\mu^1$-almost everywhere.
\end{proof}

Now our remaining task is to prove that $i(F_n)$ converges to $i(F)$ in $\Iso([\calg], \delta_u)$, i.e., for all $\phi \in [\calg]$, $F_n(\phi)$ converges to $F(\phi)$ in $[\calg]$.
However this follows from the definition of $F$, as noted before Claim \ref{claim-f-aut}.
\end{proof}

Recall that the group $\Iso([\calg], \delta_u)$ is endowed with the pointwise convergence topology.
Therefore the following $\delta$ defines a compatible metric on $\Iso([\calg], \delta_u)$:
\[\delta(\alpha, \beta)\defeq \sum_{k=1}^\infty 2^{-k}(\delta_u(\alpha(\theta_k), \beta(\theta_k))+\delta_u(\alpha^{-1}(\theta_k), \beta^{-1}(\theta_k)))\]
for $\alpha, \beta \in \Iso([\calg], \delta_u)$, where $\{ \theta_k\}$ is a countable dense subset of $[\calg]$.
By \cite[Corollary 1.2.2]{bk}, this metric $\delta$ is complete.
In the rest of this section, we suppose that the associated equivalence relation $\calr_\calg$ is aperiodic.
We define the topology $\tau$ on $\aut(\calg)$ as the topology induced by the relative topology of the image $i(\aut(\calg))$ in $\Iso([\calg], \delta_u)$, via the embedding $i$.
By Proposition \ref{prop-i} (ii), we obtain the following:

\begin{prop}
Suppose that the associated equivalence relation $\calr_\calg$ is aperiodic.
With respect to the topology $\tau$, $\aut(\calg)$ is a Polish group, and the following $\delta_\tau$ defines a compatible complete metric on $\aut(\calg)$:
\[\delta_\tau(F, G)\defeq \sum_{k=1}^\infty 2^{-k}(\delta_u(F(\theta_k), G(\theta_k))+\delta_u(F^{-1}(\theta_k), G^{-1}(\theta_k)))\]
for $F, G \in \aut(\calg)$, where $\{ \theta_k\}$ is a countable dense subset of $[\calg]$.
\end{prop}

We define the homomorphism $j\colon [\calg]\to \aut(\calg)$ by conjugation.
That is, for each $\phi \in [\calg]$, we define the automorphism $j(\phi)\colon \calg \to \calg$ by $j(\phi)(g)\defeq \phi_{r(g)}g(\phi_{s(g)})^{-1}$ for $g\in \calg$.
If $\calg$ is principal, then $j$ is injective (\cite[Corollary 4.11]{kec}).
In general, $j$ may not be injective.
In fact, $\ker j$ precisely consists of all $\phi \in [\calg]$ such that for $\mu$-almost every $x\in \calg^0$, $\phi_x$ belongs to the isotropy group $\calg_x^x$ and is central in $\calg_x^x$.
The homomorphism $i\circ j\colon [\calg]\to \Iso([\calg], \delta_u)$ is continuous since for $\phi \in [\calg]$, the isometry $(i\circ j)(\phi)$ is given by the map $\psi \mapsto \phi \psi \phi^{-1}$.
Thus $j$ is also continuous, and $\ker j$ is a closed subgroup of $[\calg]$.
We now arrive at the goal of this section:

\begin{prop}\label{prop-cs-nc}
Suppose that the associated equivalence relation $\calr_\calg$ is aperiodic.
Then there exists a non-trivial central sequence in $[\calg]$ if and only if $j([\calg])$ is not closed in $\aut(\calg)$ with respect to the topology $\tau$.
\end{prop}

\begin{proof}
Let $[\calg]_\star \defeq [\calg]/\ker j$ be the quotient group and endow it with the quotient topology.
Since $\ker j$ is closed in $[\calg]$, the group $[\calg]_\star$ is Polish (\cite[Proposition 1.2.3]{bk}).
We have the compatible metric $d$ on $[\calg]_\star$ defined by
\[d([\phi], [\psi])\defeq \inf \{ \, \delta_u(\phi_1, \psi_1)\mid \phi_1\in [\phi],\ \psi_1\in [\psi]\, \}\]
for $\phi, \psi \in [\calg]$, where $[\phi]\in [\calg]_\star$ denotes the projection of $\phi$.
The homomorphism $j$ induces the injective continuous homomorphism $j_\star \colon [\calg]_\star \to \aut(\calg)$.

\begin{claim}
The image $j([\calg])=j_\star([\calg]_\star)$ is closed in $\aut(\calg)$ if and only if $j_\star$ is a homeomorphism onto its image.
\end{claim}

\begin{proof}
If $j_\star([\calg]_\star)$ is closed in $\aut(\calg)$, then it is Polish with respect to the relative topology, and the map $j_\star \colon [\calg]_\star \to j_\star([\calg]_\star)$ is a continuous bijective homomorphism between Polish groups.
Such a homomorphism is a Borel isomorphism by \cite[Corollary 15.2]{kec-set}, and is a homeomorphism by \cite[Theorem 9.10]{kec-set}.

If $j_\star$ is a homeomorphism onto its image, then the relative topology on $j_\star([\calg]_\star)$ induced from $\aut(\calg)$ is Polish, and by \cite[Proposition 1.2.1]{bk} (together with \cite[Theorem 3.11]{kec-set} for supplementing the proof in \cite{bk}), $j_\star([\calg]_\star)$ is closed in $\aut(\calg)$.
\end{proof}


Suppose that $j_\star$ is a homeomorphism onto its image.
Let $(\phi_n)$ be a central sequence in $[\calg]$.
Then $\delta_\tau(j(\phi_n), \mathrm{id})\to 0$.
By assumption, $d([\phi_n], [1])\to 0$, where $1\in [\calg]$ denotes the trivial element.
Since $d([\phi_n], [1])$ is equal to the infimum of $\delta_u(\phi_n, \psi)$ among $\psi \in \ker j$, the central sequence $(\phi_n)$ is trivial.
Thus we proved the ``only if" part of the proposition.

Suppose that $j_\star$ is not a homeomorphism onto its image.
Since $j_\star$ is continuous, the map $j_\star^{-1}\colon j_\star([\calg]_\star)\to [\calg]_\star$ is not continuous, i.e., there exist $\phi_n, \phi \in [\calg]$ such that $j(\phi_n)\to j(\phi)$ in $\aut(\calg)$ and $[\phi_n]\not\to [\phi]$ in $[\calg]_\star$.
Then for $\psi_n\defeq \phi_n\phi^{-1}$, we have $j(\psi_n)\to \mathrm{id}$ in $\aut(\calg)$ and $[\psi_n]\not\to [1]$ in $[\calg]_\star$.
Thus $(\psi_n)$ is a non-trivial central sequence in $[\calg]$, and this proves the ``if" part of the proposition.
\end{proof}



\begin{thebibliography}{99999}


\bibitem[ADR]{ar}C. Anantharaman-Delaroche and J. Renault, \textit{Amenable groupoids}, Monogr. Enseign. Math., 36. L'Enseignement Math\'ematique, Geneva, 2000.

\bibitem[BK]{bk}H. Becker and A. S. Kechris, \textit{The descriptive set theory of Polish group actions}, London Math. Soc. Lecture Note Ser., 232, Cambridge University Press, Cambridge, 1996.


\bibitem[CH]{ch}J. de Canniere and U. Haagerup, Multipliers of the Fourier algebras of some simple Lie groups and their discrete subgroups, \textit{Amer. J. Math.} \textbf{107} (1985), 455--500.

\bibitem[Ch]{choda}M. Choda, Inner amenability and fullness, \textit{Proc. Amer. Math. Soc.} \textbf{86} (1982), 663--666.

\bibitem[Co]{connes-ap}A. Connes, Almost periodic states and factors of type $\mathrm{III}_1$, \textit{J. Funct. Anal.} \textbf{16} (1974), 415--445.


\bibitem[CFW]{cfw}A. Connes, J. Feldman, and B. Weiss, An amenable equivalence relation is generated by a single transformation, \textit{Ergodic Theory Dynam. Systems} \textbf{1} (1981), 431--450.

\bibitem[CJ]{cj}A. Connes and V. F. R. Jones, A ${\rm II}_1$ factor with two nonconjugate Cartan subalgebras, {\it Bull. Amer. Math. Soc. (N. S.)} \textbf{6} (1982), 211--212.

\bibitem[CK]{ck}A. Connes and W. Krieger, Measure space automorphisms, the normalizers of their full groups, and approximate finiteness, \textit{J. Funct. Anal.} \textbf{24} (1977), 336--352.

\bibitem[dC]{cor}Y. de Cornulier, Finitely presentable, non-Hopfian groups with Kazhdan's property (T) and infinite outer automorphism group, \textit{Proc. Amer. Math. Soc.} \textbf{135} (2007), 951--959.

\bibitem[DV]{dv}T. Deprez and S. Vaes, Inner amenability, property Gamma, McDuff $\textrm{II}_1$ factors and stable equivalence relations, \textit{Ergodic Theory Dynam. Systems} \textbf{38} (2018), 2618--2624.

\bibitem[Ef]{effros}E. G. Effros, Property $\Gamma$ and inner amenability, {\it Proc. Amer. Math. Soc.} \textbf{47} (1975), 483--486.

\bibitem[Er]{ershov}M. Ershov, Kazhdan groups whose FC-radical is not virtually abelian, \textit{J. Comb. Algebra} \textbf{1} (2017), 59--62.

\bibitem[FSZ]{fsz}J. Feldman, C. E. Sutherland, and R. J. Zimmer, Subrelations of ergodic equivalence relations, \textit{Ergodic Theory Dynam. Systems} \textbf{9} (1989), 239--269.

\bibitem[F]{fell}J. M. G. Fell, Weak containment and induced representations of groups, \textit{Canad. J. Math.} \textbf{14} (1962), 237--268.

\bibitem[GdlH]{gdlh}T. Giordano and P. de la Harpe, Groupes de tresses et moyennabilit\'e int\'erieure, \textit{Ark. Mat.} \textbf{29} (1991), 63--72.

\bibitem[GP]{gp}T. Giordano and V. Pestov, Some extremely amenable groups related to operator algebras and ergodic theory, \textit{J. Inst. Math. Jussieu} \textbf{6} (2007), 279--315.

\bibitem[G]{gla}E. Glasner, \textit{Ergodic theory via joinings}, Math. Surveys Monogr., 101, Amer. Math. Soc., Providence, RI, 2003.

\bibitem[GM]{gm}Y. Glasner and N. Monod, Amenable actions, free products and a fixed point property, \textit{Bull. Lond. Math. Soc.} \textbf{39} (2007), 138--150.

\bibitem[Hh]{hahn}P. Hahn, The regular representations of measure groupoids, \textit{Trans. Amer. Math. Soc.} \textbf{242} (1978), 35--72.

\bibitem[Hl]{hall}P. Hall, Finiteness conditions for soluble groups, \textit{Proc. Lond. Math. Soc. (3)} \textbf{4} (1954), 419--436.

\bibitem[Hm]{ha} T. Hamachi, Canonical subrelations of ergodic equivalence relations-subrelations, \textit{J. Operator Theory} \textbf{43} (2000), 3--34.

\bibitem[I]{i}A. Ioana, A relative version of Connes' $\chi(M)$ invariant and existence of orbit inequivalent actions, \textit{Ergodic Theory Dynam. Systems} \textbf{27} (2007), 1199--1213.

\bibitem[IS]{is}A. Ioana and P. Spaas, A class of $\textrm{II}_1$ factors with a unique McDuff decomposition, preprint, to appear in \textit{Math. Ann.}, arXiv:1808.02965.

\bibitem[J]{j}V. F. R. Jones, A converse to Ocneanu's theorem, \textit{J. Operator Theory} \textbf{10} (1983), 61--63.

\bibitem[JS]{js}V. F. R. Jones and K. Schmidt, Asymptotically invariant sequences and approximate finiteness, \textit{Amer. J. Math.} \textbf{109} (1987), 91--114.

\bibitem[Ke1]{kec-set}A. S. Kechris, \textit{Classical descriptive set theory}, Grad. Texts in Math., 156, Springer-Verlag, New York, 1995.

\bibitem[Ke2]{kec}A. S. Kechris, \textit{Global aspects of ergodic group actions}, Math. Surveys Monogr., 160, Amer. Math. Soc., Providence, RI, 2010.

\bibitem[KL]{kerr-li}D. Kerr and H. Li, Ergodic theory. Independence and dichotomies, Springer Monogr. Math., Springer, Cham, 2016.

\bibitem[Ki1]{kida-inn}Y. Kida, Inner amenable groups having no stable action, \textit{Geom. Dedicata} \textbf{173} (2014), 185--192.

\bibitem[Ki2]{kida-stab}Y. Kida, Stability in orbit equivalence for Baumslag-Solitar groups and Vaes groups, \textit{Groups Geom. Dyn.} \textbf{9} (2015), 203--235.

\bibitem[Ki3]{kida-srt}Y. Kida, Stable actions of central extensions and relative property (T), \textit{Israel J. Math.} \textbf{207} (2015), 925--959.

\bibitem[Ki4]{kida-sce}Y. Kida, Stable actions and central extensions, \textit{Math. Ann.} \textbf{369} (2017), 705--722.

\bibitem[M]{mar}A. Marrakchi, Stability of products of equivalence relations, \textit{Compos. Math.} \textbf{154} (2018), 2005--2019.

\bibitem[PP]{pp}M. Pimsner and S. Popa, Entropy and index for subfactors, \textit{Ann. Sci. \'Ec. Norm. Sup\'er. (4)} \textbf{19} (1986), 57--106.

\bibitem[PV]{pv}S. Popa and S. Vaes, On the fundamental group of $\textrm{II}_1$ factors and equivalence relations arising from group actions, in \textit{Quanta of maths}, 519--541, Clay Math. Proc., 11, Amer. Math. Soc., Providence, RI, 2010.

\bibitem[Sc1]{sch-asymp}K. Schmidt, Asymptotically invariant sequences and an action of $\mathrm{SL}(2, \mathbf{Z})$ on the $2$-sphere, \textit{Israel J. Math.} \textbf{37} (1980), 193--208.

\bibitem[Sc2]{sch-mild}K. Schmidt, Asymptotic properties of unitary representations and mixing, \textit{Proc. Lond. Math. Soc. (3)} \textbf{48} (1984), 445--460.

\bibitem[Sc3]{sch-prob}K. Schmidt, Some solved and unsolved problems concerning orbit equivalence of countable group actions, in {\it Proceedings of the conference on ergodic theory and related topics, II (Georgenthal, 1986)}, 171--184, Teubner-Texte Math., 94, Teubner, Leipzig, 1987.

\bibitem[SS]{ss}A. M. Sinclair and R. R. Smith, \textit{Finite von Neumann algebras and masas}, London Math. Soc. Lecture Note Ser., 351, Cambridge University Press, Cambridge, 2008.

\bibitem[Su]{su} C. E. Sutherland, Subfactors and ergodic theory, in \textit{Current topics in operator algebras (Nara, 1990)}, 38--42, World Sci. Publ., River Edge, NJ, 1991.

\bibitem[TD]{td}R. Tucker-Drob, Invariant means and the structure of inner amenable groups, preprint, to appear in \textit{Duke Math. J.}, arXiv:1407.7474.

\bibitem[V]{vaes}S. Vaes, An inner amenable group whose von Neumann algebra does not have property Gamma, \textit{Acta Math.} \textbf{208} (2012), 389--394.

\bibitem[W]{wright} F. B. Wright, A reduction for algebras of finite type, \textit{Ann. of Math. (2)} \textbf{60} (1954), 560--570.

\bibitem[Z]{z}R. J. Zimmer, Extensions of ergodic group actions, \textit{Illinois J. Math.} \textbf{20} (1976), 373--409.

\end{thebibliography}
\end{document}